\documentclass[12pt, reqno]{amsart}%
\usepackage{amssymb}
\usepackage{amsmath,esint}
\usepackage[shortlabels]{enumitem}
\usepackage{amsthm}
\usepackage{amscd,mdwlist}
\usepackage[margin=1in]{geometry}
\usepackage{color}
\usepackage{cite}
\usepackage{bbm, mathrsfs,pgf,tikz}
\usepackage{tcolorbox}
\usepackage{thmtools}
\usepackage[backgroundcolor=white, bordercolor=blue,
linecolor=blue]{todonotes}
\usepackage{url}
\usepackage{dsfont}

\usetikzlibrary{arrows}
\usepackage[
pagebackref=true, 
pdfpagelabels, 
plainpages=false]{hyperref}

\hypersetup{
colorlinks=true,
linkcolor=red,%---------------------%couleur des liens dans le document (references
filecolor=magenta,%-----------------%couleur des files à ouvrir
urlcolor=cyan,%---------------------%couleur des liens internets à ouvrir
citecolor=blue,%--------------------%couleur des liens pour les bibliographies
}%-------------------------------------%pour les courleur des liens ; ne jamias oublier de mettre le package hyperref avant

\theoremstyle{plain}
\declaretheorem[title=Theorem, parent=section]{theorem}
\declaretheorem[title=Lemma,sibling=theorem]{lemma}
\declaretheorem[title=Proposition,sibling=theorem]{proposition}
\declaretheorem[title=Corollary,sibling=theorem]{corollary}

\theoremstyle{definition}
\declaretheorem[title=Definition,sibling=theorem]{definition}
\declaretheorem[title=Remark,sibling=theorem]{remark}
\declaretheorem[title=Remark, numbered=no]{remark*}
\declaretheorem[title=Example, sibling=theorem]{example}

\numberwithin{equation}{section}

\DeclareMathOperator*{\essinf}{ess\,inf}
\DeclareMathOperator{\dist}{dist}
\DeclareMathOperator{\supp}{supp}
\DeclareMathOperator{\diam}{diam}

\DeclareMathOperator{\loc}{loc}

\DeclareMathOperator{\pv}{\operatorname{p.\!v.}}

\DeclareMathOperator{\cE}{\mathcal{E}}

\DeclareMathOperator{\cN}{\mathcal{N}\!\!}

\newcommand{\il}{\int\limits}
\newcommand{\iil}{\iint\limits}
\DeclareMathOperator{\R}{\mathbb{R}}

\renewcommand{\d}{\mathrm{d}}
\newcommand{\nuxminy}{\nu(x\!-\!y)}

\newcommand{\VnuOm}{V_{\nu}(\Omega|\R^d)}
\newcommand{\VnuE}{V_{\nu}(\Omega|E)}
\newcommand{\VnuOmOm}{V_{\nu,0}(\Omega|\R^d)}
\newcommand{\VnuOma}{V_{\nu_\alpha}(\Omega|\R^d)}

\newcommand{\HnuOm}{H_{\nu}(\Omega)}
\newcommand{\HnuOma}{H_{\nu_\alpha}(\Omega)}
\newcommand{\TnuOm}{T_{\nu} (\Omega^c)}
\newcommand{\TnuOma}{T_{\nu_\alpha} (\Omega^c)}
\newcommand{\eps}{\varepsilon}

\newcommand{\classA}[1]{\mathbf{\mathscr{A}_{#1}}}

\newcommand{\vertiii}[1]{{\left\vert\kern-0.25ex\left\vert\kern-0.25ex\left\vert #1 \right\vert\kern-0.25ex\right\vert\kern-0.25ex\right\vert}}

\usepackage[english]{babel}

\addto\extrasenglish{%

}

\begin{document}

\title{A general framework for nonlocal Neumann problems}

\author{Guy Foghem$^{\dagger}$}
\address{{\tiny Technische Universit\"{a}t Dresden, Fakult\"{a}t f\"{u}r Mathematik, Institut f\"{u}r Wissenschaftliches Rechnen,  Zellescher Weg 23-25 01217, Dresden, Germany.}}
\email{guy.foghem@tu-dresden.de}

\author{Moritz Kassmann$^{\ddagger}$\\
	\url{https://dx.doi.org/10.4310/CMS.2024.v22.n1.a2}
	} 
\address{Universit\"at Bielefeld, Fakult\"at f\"ur Mathematik, Postfach 100131, 33501 Bielefeld, Germany}
\email{moritz.kassmann@uni-bielefeld.de}

\thanks{$^{\dagger,\ddagger}$Financial support by the DFG via IRTG 2235: ``Searching for the regular in the irregular: Analysis of singular and random systems'' is gratefully acknowledged. $^{\dagger}$Guy Foghem gratefully acknowledges financial support by the DFG via the Research Group 3013: ``Vector-and Tensor-Valued Surface PDEs''.}

\begin{abstract}
Within the framework of Hilbert spaces, we solve nonlocal problems in bounded domains with prescribed conditions on the complement of the domain. Our main focus is on the inhomogeneous Neumann problem in a rather general setting. We also study the transition from complement value problems to local boundary value problems. Several results are new even for the fractional Laplace operator. The setting also covers relevant models in the framework of peridynamics.
%\begin{center}
%	\url{https://dx.doi.org/10.4310/CMS.2024.v22.n1.a2}
%\end{center}
\end{abstract}


\maketitle

%\tableofcontents
%

\section{Introduction}\label{sec:intro}

\subsection{Main Results}

Over the last years, there have been several studies of nonlocal Neumann problems of the following type: Given a bounded open set $\Omega \subset \R^d$, one is interested in well-posedness for 
\begin{align}\label{eq:nonlocal-Neumann-intro}\tag{$N$}
L u = f \text{ in } \Omega, \qquad \mathcal{N} u= g \text{ on } \R^d\setminus\Omega \,,
\end{align} 
where $L$ is an integral or integro-differential operator and $\cN$ \ is a related integral operator, which plays the role of some kind of normal derivative on $\R^d\setminus\Omega$. The main goal of this article is to prove well-posedness results for \eqref{eq:nonlocal-Neumann-intro} in a general setting. We assume:
\begin{alignat*}{2}
L u(x) &= \pv \int_{\R^d} \big(u(x)-u(y)\big) k(x,y) \mathrm{d} y \qquad &&(x\in \R^d) \,, \\
\mathcal{N} u (y)&= \int_{\Omega}(u(y)-u(x)) k(x,y) \d x \qquad &&(y\in \Omega^c) \,.
\end{alignat*}
Here, $k: \R^d \times \R^d \setminus \operatorname{diag} \to [0, \infty)$ is measurable and satisfies
\begin{align}\label{eq:k-elliptic-intro}\tag{E}
\Lambda^{-1} \nu(y-x) \leq k(x,y) \leq \Lambda \nu(y-x) \quad (x,y \in \R^d) \,,
\end{align}
where $\nu:\R^d \setminus\{0\}\to [0,\infty)$ is the density of a symmetric L\'{e}vy measure, i.e., $\nu$ satisfies 
\begin{align}\label{eq:levy-cond}\tag{L}
\nu(h) = \nu(-h) \text{ for all } h \neq 0 \,\,\,\text{ and } \,\,\,\int_{\R^d}  \big( 1\land |h|^2 \big) \nu(h) \d h < \infty\,. 
\end{align}

The main new contributions of the present article include the following: 

\begin{enumerate}[(a)]
\item Extending previous results, e.g. from \cite{DROV17}, \cite{MuPr19}, \cite{DTZ22}, we treat the inhomogeneous problem for natural choices of data $g$. The corresponding results are new for the fractional Laplace operator.
\item We provide a general framework that includes integrable and singular kernels at the same time. 
\item We show that the trace spaces introduced in \cite{DyKa19} and \cite{BGPR20} coincide. 
\item We introduce a new Dirichlet-to-Neumann operator based on the operator $\cN$.
\item Our well-posedness results are aligned with classical results for second order partial differential operators. We show convergence of nonlocal to local problems, where we treat singular and bounded kernels together. 
\end{enumerate}

\medskip

Let us explain condition \eqref{eq:k-elliptic-intro}. We denote $a\land b= \min(a,b)$ for $a,b\in \R$. In the case $k(x,y) = \nu(y-x)$ with $\nu$ as above, the operator $L$ is translation invariant and generates a symmetric L\'evy process. The density $\nu$ defines the ``order'' of the operator $L$, which becomes apparent in the case of $\nu(h) =C_{d,\alpha} |h|^{-d-\alpha}$ for $h\neq 0$ where $ \alpha \in (0,2)$ is fixed and $C_{d,\alpha}$ is an appropriate constant. The resulting operator is the so-called fractional Laplace operator $(-\Delta)^{\alpha/2}$. 
%The constant $C_{d, \alpha}$ given by 
%\begin{align*}
%C_{d,\alpha}:= \left(\int_{\R^d} \frac{1-\cos(x_1)}{|x|^{d+\alpha}} \,dx\right)^{-1}
%\end{align*}
%
The choice of $C_{d,\alpha}$ ensures the relation $\widehat{(-\Delta)^{\alpha/2} u}(\xi)= |\xi|^\alpha \widehat{u}(\xi)$ for all functions $u$ in $C^\infty_c(\R^d)$. Let us mention that asymptotically one has $C_{d, \alpha}\asymp \alpha (2-\alpha)$. This will play an important role for our analysis. Further details about the fractional Laplacian $(-\Delta)^{\alpha/2}$ and the constant $C_{d,\alpha}$ can be found in \cite{AAS67,Hitchhiker,guy-thesis}. Finally, let us mention that, the assumptions \eqref{eq:k-elliptic-intro} and \eqref{eq:levy-cond} are not sufficient in order to guarantee the existence of the pointwise expression $Lu(x)$ in the general case, even if $u$ is smooth. The Hilbert space approach used in \autoref{sec:existence} avoids this issue because we only deal with the corresponding quadratic forms. 
It is worth to mention that the nonlocal operator $\mathcal{N}$ was initially introduced by \cite{DROV17} and is called the \emph{nonlocal normal derivative operator} across the boundary of  $\Omega$ with respect to $\nu$. Another type of such an operator appeared earlier in the literature see for instance \cite{DGLZ12}. 

\medskip

Let us quickly review the classical Neumann problem, for the reader's convenience. Let $\Omega\subset \R^d $ be a bounded open subset whose boundary $\partial\Omega$ is sufficiently regular. Given $f:\Omega\to\mathbb{R}$ and $g: \partial\Omega \to \mathbb{R}$ measurable, the classical inhomogeneous Neumann problem associated to the data $f$ and $g$ consists in finding a function $u:\Omega\to \mathbb{R}$ satisfying
\begin{align}\label{eq:local-Neumann}
-\Delta u = f \quad\text{in}~~~ \Omega \quad\quad\text{ and } \quad\quad \frac{\partial u}{\partial n}= g ~~~ \text{on}~~~ \partial \Omega.
\end{align}
Here $\frac{\partial u}{\partial n} $ denotes the outward normal derivative of $u$ on $\partial\Omega$.  From a weak formulation point of view, $u$ is said to be a  weak solution of  \eqref{eq:local-Neumann} if  $u\in H^1(\Omega)$ satisfies
\begin{align*}
\int_{\Omega} \nabla u(x)\cdot \nabla v(x) \d x = \int_{\Omega} f(x)v(x)\d x +\int_{\partial\Omega} g(y)v(y)\d \sigma(y),\quad \mbox{for all}~~v \in H^1(\Omega)\,.
\end{align*} 
The Neumann boundary problem has received considerably less attention in the literature compared to the Dirichlet boundary problem. Classical textbooks like \cite{Mikhailov78} treat the basic aspects.  A rigorous treatment including regularity up to the boundary, Schauder estimates, $L^p$ estimates and the variational formulation can be found in the lecture notes \cite{Giovanni13}. A recent article covering classical results for elliptic equations in divergence form is \cite{DjVj09}.

%To be more precise if we assume $u$ is the restriction on $\overline{\Omega}$ of a smooth function $\widetilde{u}:\R^d \to \mathbb{R}$ then $\frac{\partial u}{\partial n} (x)= \nabla \widetilde{u}(x)\cdot n(x)$ where $n(x) $ is the outer normal vector at $x\in \partial\Omega$. 
%It is frequent to consider $ f\in  L^2(\Omega)$ and $g\in H^{1/2}(\partial\Omega)$ the trace spaces of $H^1(\Omega)$. However for the variational formulation the regularity on $f$ and $g$ can be relaxed.

\medskip

Following \cite{FKV15, SV14, SV13} we introduce a bilinear form $\cE$ by 
\begin{align}\label{eq:def-Euu-intro}
\mathcal{E}(u,v) =\frac{1}{2} \!\!\iil_{(\Omega^c\times \Omega^c)^c} \!\! \big(u(x)-u(y) \big) \big(v(x)-v(y) \big) \, \nuxminy \d x \, \d y 
\end{align}
for all smooth functions with compact support. As in the local case, a main tool in the study of Neumann problems, is a Gauss-Green type formula for $u,v \in C^\infty_c(\R^d)$, see \autoref{prop:gauss-green}:
\begin{align}\label{eq:green-gauss-nonlocal-intro}
\int_{\Omega}  Lu(x) v(x)\d x= \mathcal{E}(u,v) -\int_{\Omega^c}\mathcal{N}u(y)v(y)\d y.
\end{align}
Relation \eqref{eq:green-gauss-nonlocal-intro} motivates us to introduce an energy space $ V_{\nu} (\Omega|\R^d)$ as the vector space of all measurable functions $u: \R^d \to \R$ such that the restriction $u|_\Omega$ belongs to $L^2(\Omega)$ and $\cE(u,u)$ is finite. See \autoref{subsec:function-spaces} for more details. The energy space $\VnuOm$ can be seen as a nonlocal analog of $H^1(\Omega)$. Let us make an interesting observation. Let $f \in L^2(\Omega)$ and $u \in \VnuOm$ be a minimizer of the functional $v\mapsto \frac12 \cE(v,v) - \int_{\Omega} fv$ in the space $\VnuOm$. Then $\cE(u,v) = 0$ for all smooth functions with compact support in $\R^d \setminus \overline{\Omega}$. Since $u, v \in \VnuOm$, the Fubini theorem implies 
\begin{align*}
\int_{\Omega^c} \cN u(y) v(y) \d y  = \cE(u,v) = 0  \,,
\end{align*}
which implies $\cN u = 0$ in $\Omega^c$, see Corollary \ref{cor:natural-cond}. On the one hand, this observation is aligned with the classical theory where the normal derivative appears naturally when minimizing the energy. On the other hand, and this is interesting, here we do not need to assume any regularity of the kernel $k(x,y)$ and the boundary $\partial \Omega$. 

\medskip

Let us summarize the main results of this work.
\begin{enumerate}[(i)]
\item The first step is to define a base space $L^2(\R^d, \widetilde{\nu})$, in which we can define the complement value problems. We define $\widetilde{\nu}$ and two alternative options $\overline{\nu}, \nu^*$ in \autoref{def:different-nus}. In \autoref{subsec:function-spaces} we study embedding results of corresponding function spaces. 
\item The next step is to introduce $\TnuOm$ as the trace space of $\VnuOm$ in \autoref{subsec:trace-space}. In this section, we study equivalent norms of the trace space and a density result.
\item An important tool in the proof of well-posedness results is the compact embedding $\VnuOm \hookrightarrow L^2(\Omega)$, which is a core result in \autoref{sec:compactness}, see \autoref{thm:embd-compactness}.
\item \autoref{sec:existence} is dedicated to well-posedness results. We focus on the Neumann problem in \autoref{subsec:neumann}. An existence result for problem \eqref{eq:nonlocal-Neumann-intro} is given in \autoref{thm:nonlocal-Neumann-var}. We also discuss a more general Robin-type complement value problem. 
\item The setup of this work allows to define a fully nonlocal Dirichlet-to-Neumann map with the help of the nonlocal Neumann-type derivative $\cN$. For $\Omega \subset \R^d$, the Dirichlet data are given on $\Omega^c$ and mapped to $\cN u$ on $\Omega^c$, where $u$ satisfies the nonlocal equation in $\Omega$. Thus, this map can be viewed as a nonlocal analog of the well-known Dirichlet-to-Neumann operator given in \cite{CaSi07}. Basic properties are formulated in \autoref{thm:DN-map} together with spectral properties in \autoref{thm:DN-map-spectral}.
\item The analogy between the classical Neumann problem and problem \eqref{eq:nonlocal-Neumann-intro} leads to a convergence result when considering a sequence of complement value problems for the fractional Laplace operator $(-\Delta)^{\alpha/2}$ where $\alpha \to 2$. \autoref{thm:phase-transition} establishes the convergence of the corresponding sequence of solutions $u_{\alpha}$ as $\alpha\to2$.
\end{enumerate}

\subsection{Related literature}
Nonlocal complements value problems have been studied in several works. In particular, the Dirichlet problem is studied in many articles. For translation invariant problems, see the survey \cite{Ros16,AFR20} for fine regularity results and \cite{FKV15, BuVa16} for the Hilbert space approach in a similar setting as in this work.  

\medskip

An early contribution to nonlocal Neumann problems is \cite{DROV17}, where also the Gauss-Green formula appears for a special case. There is a difference between our approach and the one in \cite{DROV17}, which explains why we are able to study the inhomogeneous Neumann problem. Let us explain our approach for the simplest setup of the fractional Laplace operator $(-\Delta)^{\alpha/2}$, i.e., $\nu(h) =C_{d,\alpha} |h|^{-d-\alpha}$. Given $f \in L^2(\Omega)$ and $g \in L^2(\Omega^c, (1+|x|)^{d+\alpha} \d x)$, motivated by the Gauss-Green formula \eqref{eq:green-gauss-nonlocal-intro} we say that $u \in \VnuOm$ is a weak solution or a variational solution of the inhomogeneous Neumann problem \eqref{eq:nonlocal-Neumann-intro} if 
\begin{align}\label{eq:var-nonlocal-Neumann-intro}\tag{$V$}
\mathcal{E}(u,v) = \int_{\Omega} f(x)v(x)\d x +\int_{\Omega^c} g(y)v(y)\d y \quad \text{ for all } v \in \VnuOm\,.
\end{align}
By testing \eqref{eq:var-nonlocal-Neumann-intro} with $v=1$ gives the following necessary compatibility condition 
\begin{align}\label{eq:compatible-nonlocal-intro}\tag{$C$}
\int_{\Omega} f(x)\d x +\int_{\Omega^c} g(y)\d y=0 \,.
\end{align}
Equality \eqref{eq:var-nonlocal-Neumann-intro} should be contrasted with the variational formulation of \eqref{eq:local-Neumann} in the classical case: Given $f \in L^2(\Omega)$ and $g \in L^2(\partial \Omega)$, find  $u\in H^1(\Omega)$ such that
\begin{align*}
\int_{\Omega} \nabla u(x)\cdot \nabla v(x) \d x = \int_{\Omega} f(x)v(x)\d x +\int_{\partial\Omega} g(y)v(y)\d \sigma(y) \quad \text{ for all } v \in H^1(\Omega)\,.
\end{align*} 
Note that \cite[Def. 3.6]{DROV17} and subsequent definitions like \cite[Definition 2.7]{MuPr19} look very similar to \eqref{eq:var-nonlocal-Neumann} at first glance. However, the norm of the test space defined in \cite[Eq. (3.1)]{DROV17}, \cite[Section 2]{MuPr19} depends on the Neumann data $g$, which is not natural. Our test space $\VnuOm$ in the weak formulation \eqref{eq:var-nonlocal-Neumann} does not depend on the Neumann data $g$. For the general case, we refer the reader to \autoref{def:neumann-var-sol}. It is worth mentioning that the weighted space $L^2(\Omega^c, (1+|x|)^{d+\alpha} \d x)$ is the natural function space for the Neumann data $g$. In fact in Theorem \ref{thm:non-existence-Neumann} we are able to find some $g$ not belonging to $L^2(\Omega^c, (1+|x|)^{d+\alpha} \d x)$, $\Omega=B_1(0)$ for which the variational Neumann problem \eqref{eq:var-nonlocal-Neumann-intro} with $f=0$ does not have any weak solution in $\VnuOm$.

\medskip

The aforementioned issue does not show up for \emph{homogeneous} nonlocal Neumann problems. For such problems, several results have been proved, e.g., regularity up to a boundary of a domain for the fractional Laplace operator in \cite{AFR20}. A particular observation linking the homogeneous Neumann problem to the regional fractional Laplace operator is provided in \cite{Aba20}. Eigenvalues of nonlocal mixed problems are studied in \cite{LMPPS18}. Various nonlinear Neumann problems are studied in \cite{Che18, MuPr19, CiCo20, AlTo20, MuPr21, BSS21}. Some higher order nonlocal Neumann problems are treated in \cite{BMPS18}. The classical Neumann problem is closely linked to reflected diffusions. It turns out to be a challenging problem to establish a similar link between the nonlocal Neumann problem and a Markov jump process together with its reflection. An attempt is made in \cite{Von21}, which we comment on in detail in Remark \ref{rem:vondra-strange}.

\medskip

\subsection{Peridynamics and volume constraints on bounded sets}\label{subsec:peri-intro}{\ }

In the literature, several nonlocal problems are studied in the area of peridynamics. Most of these models require complement conditions (a.k.a. volume constraints) not necessarily on the whole complement of the domain $\Omega$ but only often on a part of the complement. Here, we would like to point out that our setting can be adapted to fit such requirements. Let us exemplify this with a simple model. Consider the symmetric kernel of  the form $k(x,y)= \nu(x-y)$ with $\nu$ supported around the origin, say $\supp\,\nu\subset B_{\delta}(0)$ for some $\delta>0$. A popular example of a kernel in peridynamic models is given by $\nu(h)= \mathds{1}_{B_\delta}(h)$. In this case it is natural to assume the complement condition not on the whole complement $\R^d\setminus\Omega$ but only on $\Omega(\delta)=\{x\in \R^d\setminus\Omega: \dist(x, \partial \Omega) <\delta\}$. A nonlocal problem of the form  $Lu = f$ in $\Omega$ is then supplemented with a complement condition prescribed on the volume constraint $\Omega(\delta)$, e.g., $u=g$ on $\Omega(\delta)$ for a Dirichlet problem or  $\mathcal{N}u=g$ on $\Omega(\delta)$ for a Neumann problem. Here, $f:\Omega\to \R$ and $g:\Omega(\delta)\to \R$ are given data. Our approach can easily be adopted to cover this case. In comparison to the weak formulation \eqref{eq:var-nonlocal-Neumann-intro}, one would need to replace
$\VnuOm$ by the space $\VnuE$ defined as in \eqref{eq:def-VnuE}, with $E=\Omega\cup \Omega(\delta)$ 
and recall that $\nu(h)=\mathds{1}_{B_\delta} (h)$. Indeed, assuming for simplicity that $\Omega$ is  bounded Lipschitz and connected,  the well-posedness of the Neumann and the Dirichlet problem can be formulated  as follows. 

\begin{corollary}\label{cor:peri-Neumann} 
%Let $\Omega \subset \R^d$ be open and bounded. Let $\nu(h)=\mathds{1}_{B_\delta} (h)$, $\delta>0$ and denote $E=\Omega\cup \Omega(\delta)$ with $\Omega(\delta)=\{x\in \R^d\setminus\Omega: \dist(x, \partial \Omega) <\delta\}.$ 
Let  $f\in L^2(\Omega)$ and $g\in L^2(\Omega(\delta))$. Then there is a unique variational solution $u\in \VnuE^\perp=\VnuE\cap\{\int_{ \Omega}u\d x=0\}$  to the Neumann problem $Lu = f$ in $\Omega$ and $\mathcal{N}u=g$ on $\Omega(\delta)$, i.e., 
\begin{align}\label{eq:peri-Neumann-intro}%\tag{$V$}
\mathcal{E}(u,v) = \int_{\Omega} f(x)v(x)\d x +\int_{\Omega(\delta)} g(y)v(y)\d y \quad \text{ for all } v \in \VnuE^\perp\,.
\end{align}
Moreover, there is a constant $C>0$ independent of $f$ and $g$ such that 
\begin{align*}
\|u\|_{\VnuE}\leq C\big(\|f\|_{L^2(\Omega)}+ \|g\|_{L^2(\Omega(\delta))}\big). 
\end{align*}
\end{corollary}

\medskip

\begin{corollary} \label{cor:peri-Dirichlet} 
%Let the setting of \autoref{thm:peri-Neumann} be in place.   
Let  $f\in L^2(\Omega)$ and $g\in \VnuE$. Then there is a unique variational solution $u\in \VnuE$ to the Dirichlet  problem 
$Lu = f$ in $\Omega$ and $u=g$ on $\Omega(\delta)$, i.e.,
\begin{align}\label{eq:peri-Dirichlet-intro}%\tag{$V$}
u-g\in V_{\nu,0}(\Omega|E)\qquad\text{and}\qquad \mathcal{E}(u,v) = \int_{\Omega} f(x)v(x)\d x  \quad \text{ for all } v \in V_{\nu,0}(\Omega|\R^d)\,. 
\end{align}
Here we denote $V_{\nu,0}(\Omega|E)= \VnuE\cap\{u|_{\Omega(\delta) }=0\}$. Moreover, there is a constant $C>0$ independent of $f$ and $g$ such that 
\begin{align*}
\|u\|_{\VnuE}\leq C\big(\|f\|_{L^2(\Omega)}+ \|g\|_{\VnuE}\big). 
\end{align*}
\end{corollary}
The proofs of \autoref{cor:peri-Neumann} and \autoref{cor:peri-Dirichlet} are analogous to the ones of \autoref{thm:nonlocal-Neumann-var} and \autoref{thm:nonlocal-Dirichlet-var}; we also refer to \cite{Fog23s} for a more general setting. Both results are well known for special cases of $\nu$ in the area of peridynamics, see \cite[Section 3.2]{DTZ22} and \cite{KMS19}. We refer the reader to the exposition and the references in \cite{Du19}. Let us mention some related results. An early work is \cite{BT10} where several nonlocal complement value problems are studied for integrable kernels with fixed support (horizon). Problems for nonlocal nonlinear problems involving nonlocal operators of regional type are studied in \cite{BM14, BMP15}. Nonlocal Dirichlet problems driven by nonsymmetric singular kernels are considered in \cite{FKV15} for scalar functions and in \cite{KMS19} for vector-valued functions.  The vanishing-horizon limit has been considered in several works, see \autoref{sec:transition}. For references related to numerical results see \cite{DEYu21}. We provide $\Gamma$-convergence results for vanishing horizons in \autoref{exa:vanishing-horizon} and mention related results from peridynamics. Our systematic approach in terms of functional analysis allows to treat general cases of $\nu$ resp. general data $g$ in comparison with \cite{DTZ22}. 
%There  are also several motivations to study nonlocal models in peridynamic ranging from mathematical biology \cite{DLV21, DV21} for the study of logistic and  for ecological models to physics \cite{DTY20, DGL13b, DLT15} for the study of linear elasticity and diffusion of particle in hydrodynamics. 

\medskip

The paper is organized as follows. In \autoref{sec:function-spaces} we introduce some nonlocal function spaces and the corresponding trace spaces that will be used in the sequel. In \autoref{sec:compactness} 
we establish global compact embedding of nonlocal Sobolev spaces in to $L^2$-spaces. This allows us to prove Poincar\'e type inequalities for various ranges of L\'evy integrable kernel $\nu$.  The \autoref{sec:existence} is devoted to the study of the well-posedness of nonlocal problems with Dirichlet, Neumann and Robin conditions associated with the L\'evy operator $L$. Afterwards, we investigate the Dirichlet-to-Neumann  map for the L\'evy operator $L$. In \autoref{sec:transition} we show that local elliptic problems can be viewed as the limit of the nonlocal ones.  Last, in \autoref{sec:appendix} we highlight some elementary properties of the L\'evy operator $L$.

\medskip

\emph{Acknowledgment:} During his PhD studies Guy Foghem has spent a research stay at Seoul National University in the framework of the International Research Training Group 2235 ``Searching for the regular in the irregular: Analysis of random and singular systems'' between Bielefeld University and Seoul National University. The authors thank his host, Prof. Ki-Ahm Lee, for helpful discussions on trace spaces.

\medskip

\emph{Further remarks:} Several results of this work are based on the PhD thesis of the first author  \cite{guy-thesis}.  Related research questions are subject to an ongoing PhD project of Michael Vu from Trier University. Equivalent norms of the trace space $\TnuOm$ in case of the fractional Laplace operator have recently been identified by F. Grube and Th. Hensiek together with convergence results that recover $H^{1/2}(\partial \Omega)$ in the limit $\alpha \to 2^-$.

%%%%%%%%%%%%%%%%%%%%%%%%%%%%%%%%%%%%%%%%%%%%%%%%%%%%%%%%%%%%%%%%%%%%%%%%%%%%%%%
%%%%%%%%%%%%%%%%%%%%%%%%%%%%%%%%%%%%%%%%%%%%%%%%%%%%%%%%%%%%%%%%%%%%%%%%%%%%%%%
\section{L\'{e}vy measures and nonlocal function spaces}\label{sec:function-spaces}
%%%%%%%%%%%%%%%%%%%%%%%%%%%%%%%%%%%%%%%%%%%%%%%%%%%%%%%%%%%%%%%%%%%%%%%%%%%%%%%
%%%%%%%%%%%%%%%%%%%%%%%%%%%%%%%%%%%%%%%%%%%%%%%%%%%%%%%%%%%%%%%%%%%%%%%%%%%%%%%

In this section we introduce generalized Sobolev-Slobodeckij-like function spaces with respect to a L\'{e}vy measure $\nu$ and an open subset $\Omega \subset \R^d$, in particular $\VnuOm$ and nonlocal trace spaces $\TnuOm$. The function spaces are tailor-made for nonlocal elliptic complement value problems including the Neumann problem. We prove the existence of an embedding of $\VnuOm$ into  weighted spaces $L^2(\R^d,\nu')$ for different measures $\nu'$ and into the nonlocal trace space $\TnuOm$. We are able to compare $\TnuOm$ with known trace spaces, see \autoref{prop:trace-comparability}. A main result of this section is \autoref{thm:Dform-reflected}, which proves that the bilinear form $(\cE, \VnuOm)$ is a regular Dirichlet form on $L^2(\R^d, \nu')$. This result allows to construct jump processes with some sort of reflection. 

\medskip

Throughout this work, let $\nu$ be a L\'{e}vy measure whose density is a measurable symmetric function $\nu:\R^d\setminus \{0\} \to [0, \infty)$ satisfying \eqref{eq:levy-cond}. We will impose further conditions on $\nu$ where needed. For simplicity, we assume in our main results that $\nu$ has full support. See \autoref{subsec:peri-intro} for a discussion on how this can be relaxed.

\subsection{L\'{e}vy condition and energy forms}
Before we begin, let us make an observation that nicely links  \eqref{eq:levy-cond} with nonlocal energies.

\begin{theorem}\label{thm:levy-meets-energy}
Assume $\nu: \mathbb{R}^d \setminus \{0\} \to [0,\infty)$ is measurable and radial \footnotemark . Then the energy
\begin{align*}
\iint_{\R^d \R^d}\big(u(y) \!-\! u(x) \big)^2 \nuxminy \; \d y \d x
\end{align*}
is finite for every $u \in C^\infty_c(\R^d)$ if and only if $\nu$ satisfies \eqref{eq:levy-cond}. \end{theorem}
\footnotetext{Note added in proof: As we learned from Florian Grube, the assertion of the theorem holds true for any Borel measure $\nu$. The $L^p$-setting is treated in \cite{Fog23s}.}

\begin{proof}
By \eqref{eq:H1-inequality}, it is easy to see that condition \eqref{eq:levy-cond} implies finiteness of the energy for $u\in C^\infty_c(\R^d)$. For the converse, let $u\in C^\infty_c(\R^d)$ be nontrivial and $\varepsilon>0$, then there is $\delta>0$ such that 
\begin{align*}
\|\nabla u(\cdot+h)-\nabla u\|^2_{L^2(\R^d)}< \varepsilon\quad  \text{ if } |h|\leq \delta\,. 
\end{align*}  

Using the fundamental theorem of calculus, $\frac{1}{2}b^2\leq a^2+(b-a)^2$ and polar coordinates yields 
\begin{align*}%\label{eq:optimal-non-levy}
\begin{split}
\iint_{\R^d \R^d} & \big(u(y) \!-\! u(x) \big)^2 \nuxminy \; \d y \d x \geq \int_{\R^d} \int_{B_\delta(0)} \Big| \int_0^1\nabla u(x+th)\cdot h\d t\Big|^2 \nu(h) \d h\,\d x\\
& \geq \frac{1}{2} \int_{\R^d} \int_{\mathbb{S}^{d-1}} |\nabla u(x)\cdot w|^2 \d\sigma_{d-1}(w) \Big( \int_{0}^{\delta} r^{d+1} \nu(r)\mathrm{d}r \Big) \d x- \varepsilon \int_{B_\delta(0)} |h|^2\nu(h)\d h\\
& \geq \Big(\frac{1}{2} K_{d,2}\|\nabla u\|^2_{L^2(\R^d)}-\varepsilon\Big)\int_{B_\delta(0)} |h|^2\nu(h)\d h.
\end{split}
\end{align*}
Recall that, invariance of the Lebesgue measure under rotations implies for all $z\in \R^d$ 
%and $e\in \mathbb{S}^{d-1}$ 
\begin{align*}
\fint_{\mathbb{S}^{d-1}} |w\cdot z|^2 \d\sigma_{d-1}(w) =K_{d,2} |z|^2\quad \text{with}\quad  K_{d,2}= \fint_{\mathbb{S}^{d-1}} |w\cdot e|^2 \d\sigma_{d-1}(w)=\frac1d .
\end{align*} 
Therefore, choosing $\varepsilon= \frac{1}{4}K_{d,2}\|\nabla u\|^2_{L^2(\R^d)},$ we obtain  
\begin{align*}
\iint_{\R^d \R^d}\big(u(y) \!-\! u(x) \big)^2 \nuxminy \; \d y \d x &\geq \frac{1}{4}K_{d,2}\|\nabla u\|^2_{L^2(\R^d)} \int_{B_\delta(0)} |h|^2\nu(h)\d h.
\end{align*}
It remains to show that $\nu$ is integrable away from the origin.
Consider another $u\in C^\infty_c(\R^d)$ with $\supp u\subset B_\tau(0)$ and $0<\tau<\delta/2$. For all $x\in B_\tau$ we have $B_\tau(0)\subset B_\delta(x)$ and hence 
\begin{align*}
\iint_{\R^d \R^d}  \big(u(y) \!-\! u(x) \big)^2 \nuxminy \; \d y \d x &\geq 2\int_{B_\tau(0)} |u(x)|^2\d x \int_{\R^d\setminus B_\tau(0)} \nu(x-y)\d y\\
&\geq 2\int_{B_\tau(0)} |u(x)|^2\d x \int_{\R^d\setminus B_\delta(0)} \nu(h)\d h.
\end{align*} 
This together with the previous estimate implies that $\nu\in L^1(\R^d, 1\land|h|^2 \d h)$, i.e., $\nu$ satisfies condition $\eqref{eq:levy-cond}$. 
\end{proof}

\subsection{Sobolev-Slobodeckij-like spaces}\label{subsec:function-spaces}

Let $\Omega \subset \R^d$ be open. Define the space $\HnuOm$ by
\begin{align*}%\label{eq:def-HnuOm}
H_{\nu} (\Omega)= \Big\{u \in L^2(\Omega)\,:\, \cE_\Omega(u,u) <\infty \Big \}\,,
\end{align*}
equipped with the norm $\|u\|^2_{H_{\nu} (\Omega)}= \|u\|^2_{L^{2} (\Omega)}+ \cE_\Omega(u,u)$, where 
\begin{align}\label{def:EOm}
\cE_\Omega(u,v) = \iil_{\Omega\Omega} \big(u(x)-u(y) \big)\big(v(x)-v(y) \big) \, \nuxminy \d x\d y \,.
\end{align}

When  $\nu\in L^1(\R^d)$, e.g., in the case $\nu(h) = \mathbbm{1}_{B_1}(h)$, the space $\HnuOm$ equals $L^2(\Omega)$. 

%The same holds true if $\nu \in L^1(\R^d)$ because in this case, for $u\in L^2(\Omega)$ 
%%
%\begin{align}
%\iil_{\Omega\Omega} &\big(u(x)-u(y) \big)^2 \, \nuxminy\d x\,\d y
%\leq 4 \iil_{\Omega\Omega}  |u(x)|^2 \, \nuxminy\d x\,\d y\notag\\
%&\leq 4 \iil_{\Omega\R^d}  |u(x)|^2 \, \nu (h)\d h\,\d x
%=4\|\nu\|_{L^1(\R^d)}\|u\|^2_{L^2(\Omega)}\,.\label{eq:integrable-nu}
%\end{align}
%%
Following \cite{FKV15} we introduce the vector space $ V_{\nu} (\Omega|\R^d)$ as follows:
\begin{align*}%\label{eq:def-VnuOm}
\VnuOm = \Big\lbrace u: \R^d \to \R \text{ meas.} : \, u|_\Omega \in L^2(\Omega), |u|^2_{\VnuOm} < \infty \Big\rbrace,
\end{align*}
where the seminorm is defined by 
\begin{align*}%\label{eq:def-VnuOm-seminorm}
|u|^2_{\VnuOm} = \iil_{\Omega \R^d} \big(u(x)-u(y) \big)^2 \, \nuxminy \d x \, \d y <\infty \,. 
\end{align*}
We endow the vector space $\VnuOm$ with the norm $\|\cdot\|_{\VnuOm}$ given by 
\begin{align*}
\|u\|^2_{\VnuOm} = \|u\|^2_{L^2(\Omega)} + |u|^2_{\VnuOm} \,.
\end{align*}
%
%\begin{align}\label{eq:def-VnuOm}
%\VnuOm = \Big\lbrace u: \R^d \to \R \text{ meas. } | \, \mathcal{E}(u,u) :=\frac{1}{2} \!\!\iil_{(\Omega^c\times \Omega^c)^c} \!\!\big(u(x)-u(y) \big)^2 \, \nuxminy \d x \, \d y <\infty \Big\rbrace \,. 
%\end{align}
 Next, given functions $u,v \in \VnuOm$, we define a bilinear form $\cE$ by 
\begin{align*}%\label{eq:def-Euu}
\mathcal{E}(u,v) =\frac{1}{2} \!\!\iil_{(\Omega^c\times \Omega^c)^c} \!\! \big(u(x)-u(y) \big) \big(v(x)-v(y) \big) \, \nuxminy \d x \, \d y \,.
\end{align*}

\begin{lemma}\label{lem:energy-vs-seminorm}
We have $|u|^2_{\VnuOm} \leq \mathcal{E}(u,u) \leq 2 |u|^2_{\VnuOm}$ for any measurable function $u$. 
\end{lemma}

\begin{proof}
On the one hand, the inequality 
\begin{align*}
\iil_{\Omega\R^d} \!\!\big(u(x)-u(y) \big)^2\, \nuxminy \d y\,\d x \
\leq  \iil_{(\Omega^c\times \Omega^c)^c} \hspace{-2ex}\big(u(x)-u(y) \big)^2 \nuxminy \d y\,\d x = 2 \cE (u,u)
\end{align*}
holds trivially true. On the other hand,
\begin{align*}
\mathcal{E}(u,u) &= \frac{1}{2} \!\!\iil_{\R^d\R^d} \!\!\big(u(x)-u(y) \big)^2 \, \big[\mathds{1}_\Omega (x)\lor \mathds{1}_\Omega(y)\big]\nuxminy \d y\,\d x  \\
&\leq \frac12 \iil_{\R^d\R^d} \!\!\big(u(x)-u(y) \big)^2 \, \big[\mathds{1}_\Omega (x) +\mathds{1}_\Omega(y)\big] \nuxminy \d y\,\d x \\
&= \iil_{\Omega\R^d} \big(u(x)-u(y) \big)^2\, \nuxminy \d y\,\d x \,,
\end{align*}
which completes the proof. 
\end{proof}

%
%\begin{definition}\label{def:main-objects}
%Let $\Omega \subset \R^d$ be open. Then we define $\VnuOm$ to be the vector space $\VnuOm \cap L^2(\Omega) $ equipped with the semi-norm $\|\cdot\|_{\VnuOm}$ and the bilinear form  $\big(u,v\big)_{\VnuOm}$, where
%	\begin{align*}
%	\|u\|^2_{\VnuOm} &= \|u\|^2_{L^{2} (\Omega)}+ \mathcal{E}(u,u) \,, \\%\label{eq:def-VnuOM-norm} \,, \\
%	\mathcal{E}(u,v)&=\frac{1}{2} \iil_{(\Omega^c\times \Omega^c)^c} \big(u(x)-u(y) \big) \big(v(x)-v(y) \big) \, \nuxminy \d x \, \d y \,, \\%\label{eq:def-E-bilin} \\
%	\big(u,v\big)_{\VnuOm} &= \big(u,v\big)_{L^2(\Omega)}+ \mathcal{E}(u,v)\,.
%	\end{align*}
%\end{definition}

Some authors find it convenient to work with the smaller space $\VnuOm \cap L^2(\R^d) $ equipped with its corresponding norm. For the study of nonlocal Dirichlet problems this restriction is not necessary, though. On the other hand, the requirement $u|_\Omega \in L^2(\Omega)$ for $u \in \VnuOm$ is natural as shown by the following observation.

\begin{proposition}\label{prop:natural-norm-on-V}
Let $\nu$ be a unimodal L\'{e}vy measure and $\Omega\subset \R^d$ be a bounded open set. Assume $\Omega\subset B_{R/2}(0) $ for some $ R\geq1 $ with $\nu(R)\neq 0$. Then $\cE(u,u) < \infty$ implies $u|_\Omega \in L^2(\Omega)$. 
\end{proposition}

The condition that $\nu$ is unimodal is not restrictive at all and we recall its definition for the readers' convenience. 

\begin{definition}\label{def:unimodality}
A L\'{e}vy density $\nu$ is called \emph{unimodal} if it is radial with an almost decreasing profile, i.e., there is a constant $c \geq 1$ such that $\nu(r) \leq c\, \nu(s)$ for all $r,s > 0$ with $ s \leq r$. 
\end{definition}

\begin{remark}
There are radial L\'{e}vy measures, which are not almost decreasing such as
\[ \nu (x) = |x|^{-d-1} \big(\tfrac{2+\cos(|x|)}{3}\big)^{|x|^4} \,. \] 
\end{remark}

\begin{proof}[Proof of \autoref{prop:natural-norm-on-V}]
First, since $\Omega\subset B_{R/2}(0)$, then for all $x,y\in \Omega$ we have $\nuxminy\geq c'$ with $c'= c\nu(R)>0 $. By Jensen's inequality, we have 

\begin{align*}
\iil_{(\Omega^c\times \Omega^c)^c} \big(u(x)-u(y)\big)^2 \nuxminy \d x \, \d y&\geq c' \iil_{\Omega\Omega} (|u(x)|-|u(y)|)^2\d x \, \d y\\
&\geq c'|\Omega| \int_{\Omega} \Big(|u(x)|-\hbox{$\fint_{\Omega}|u|$}\Big)^2\d x.
\end{align*}
This shows that the mean value $\fint_{\Omega}|u|$ is finite. We conclude $u \in L^2(\Omega)$ because of 
\[\int_{\Omega} |u(x)|^2 \d x\leq 2\int_{\Omega} \Big(|u(x)|-\hbox{$\fint_{\Omega}$}|u|\Big)^2\d x+ 2|\Omega| \Big(\hbox{$\fint_{\Omega}|u|$}\Big)^2\,. \]
\end{proof}

\begin{definition}
We define $	V_{\nu,0}(\Omega|\R^d)$ as follows:
\begin{align}\label{eq:def-VnuOm-vanish}
\begin{split}
V_{\nu,0}(\Omega|\R^d)
&= \{ u\in \VnuOm~| ~u=0~~\text{a.e. on } \R^d\setminus \Omega\}\\
&=\{ u\in H_\nu(\R^d)~| ~u=0~~\text{a.e. on } \R^d\setminus \Omega\}\,.
\end{split}
\end{align} 
\end{definition}	

\begin{remark}
The nonlocal Sobolev spaces $\HnuOm$, $\VnuOm$ and $\VnuOmOm$ are well suited for nonlocal linear  problems. Interested readers may consult \cite{guy-thesis, Fog21, Fog21b, GrHe22, Fog23s, GrKa23} for further expositions on this type of nonlocal function spaces including the $L^p$-setting and trace resp. extension results. 
%In particular we point out that a Gagliardo-Nirenberg-Sobolev type inequality is established in \cite{Fog21b}.  
\end{remark}

\begin{remark}
The function space $\big(H_\nu (\Omega), \|\cdot\|_{H_{\nu} (\Omega)}\big)$ is a separable Hilbert space, see \cite{FKV15, FGKV20}. The norms $\|\cdot \|_{\VnuOm} $ and $ \|\cdot \|_{H_\nu(\R^d)}$  agree on $V_{\nu,0}(\Omega|\R^d)$ and $V_{\nu,0}(\Omega|\R^d)$ is a closed subspace of $H_\nu(\R^d)$, hence a Hilbert space. 
\end{remark}

\begin{proposition}
If $\nu$ has full support, then $\big(\VnuOm, \|\cdot\|_{\VnuOm}\big)$ is a separable Hilbert space. 
\end{proposition}
We refer the reader to \cite{FKV15, DROV17} for a proof in a special setting and to \cite[Thm 3.46]{guy-thesis} for the general case. 

\medskip

It is worthwhile noticing that $  \|\cdot\|_{\VnuOm}$ is always a norm on  $V_{\nu,0}(\Omega|\R^d)$, but not in general a norm on $\VnuOm$ if $\nu$ is not fully supported. A simple counterexample is given by $\nu(h)=\mathds{1}_{B_1(0)} (h)$ and $\Omega=B_1(0)$. For the function $u(x)= \mathds{1}_{B^c_2(0)}(x)$ we have  $\|u\|_{\VnuOm}=0$ whereas $u\neq 0$. With regard to this comment and the discussion in \autoref{subsec:peri-intro} let us define $V_{\nu}(\Omega|E)$ for $\Omega \subset E \subset \R^d$: 
\begin{align}\label{eq:def-VnuE}
V(\Omega|E) = \Big\lbrace u: E \to \R \text{ meas.} : \, u|_\Omega \in L^2(\Omega), |u|^2_{\VnuE} < \infty \Big\rbrace,
\end{align}
where the seminorm is defined by 

\begin{align*}%\label{eq:def-VnuE-seminorm}
|u|^2_{\VnuE} = \iil_{\Omega E} \big(u(x)-u(y) \big)^2 \, \nuxminy \d y\, \d x <\infty \,. 
\end{align*}
$\VnuE$ is a seminormed space with respect to the seminorm  $\|u\|^2_{\VnuE} = \|u\|^2_{L^2(\Omega)} + |u|^2_{\VnuE}$.
We refer to \cite{Fog23s} for the proof of the next result.
\begin{proposition}
If $E = \Omega + \operatorname{supp}(\nu)$ and $0 \in \operatorname{supp}(\nu)$   then $\big(\VnuE, \|\cdot\|_{\VnuE}\big)$ is a separable Hilbert space, whenever and $\omega>0$ a.e. on $E\setminus\Omega$ where we put 
\begin{align*}
\omega(x)= \int_\Omega(1\land\nu(x-y))\d y. 
\end{align*}
\end{proposition}
As already seen in \autoref{thm:levy-meets-energy}, the condition \eqref{eq:levy-cond} is important for the properties of the spaces $H_\nu(\Omega)$ and $\VnuOm$. 

\begin{proposition} Let $\nu: \R^d\to [0,\infty]$ be symmetric. The following assertions hold true. 
\begin{enumerate}[$(i)$]
\item If $\nu\in L^1(\R^d)$, then $H_\nu(\R^d)= L^2(\R^d)$ with equivalence in norm.
\item If $\nu \in L^1(\R^d, 1\land |h|^2\d h)$ and $\Omega$ is bounded, then $\HnuOm$ and $\VnuOm$ contain all bounded Lipschitz functions. 
\item If $\nu$ is radial and $\int_{B_1} |h|^2\nu(h)\d h= \infty$ and $\int_{B_1 \setminus B_\delta} \nu(h)\d h<  \infty$  for every $\delta>0$, then $u\in C^1(\R^d)\cap H_\nu(\R^d)$ implies that $u$ is constant. 
\item If  $\nu$ is radial, then for every $u\in H^{1} (\R^d)$ there is $\delta=\delta(u)>0$, such that 
\begin{align}\label{eq:equiv-sobolev}
\frac{1}{4d} \Big( \int_{B_\delta(0)} |h|^2\nu(h)\d h \Big) \|\nabla u\|^2_{L^{2} (\R^d)} \leq |u|^2_{H\nu (\R^d)} \leq 4\|\nu\|_{L^1(\R^d, 1\land |h|^2\d h)}\|u\|^2_{H^1 (\R^d)}.
\end{align}
\end{enumerate}
\end{proposition}

The proof is analogous to the one of \autoref{thm:levy-meets-energy}. See also \cite[Proposition 2.14]{Fog21} or \cite[Proposition 3.46]{guy-thesis} for a general setting.

\medskip

For many results it is crucial that smooth functions with compact support are dense in the function space under consideration. Let us summarize some important results in this direction.

\begin{theorem}\label{thm:density} Let $\nu$ satisfies \eqref{eq:levy-cond} with full support and let $\Omega \subset \R^d$ be open.
\begin{enumerate}[$(i)$]
\item  $C^\infty(\Omega) \cap H_\nu(\Omega)$ is dense in $H_\nu(\Omega)$. 
\item  If  $\Omega$ has a compact continuous boundary $\partial \Omega$, then $C_c^\infty(\overline{\Omega})$ is dense in $H_\nu(\Omega)$. 
\item  If  $\Omega$ has a compact continuous boundary $\partial \Omega$, then $C_c^\infty(\Omega)$ is dense in $V_{\nu,0}(\Omega|\R^d)$.
\item  If  $\Omega$ has a compact Lipschitz boundary $\partial \Omega$, then $C_c^\infty(\R^d)$ is dense in $\VnuOm$ with respect to the norms $\|\cdot\|_{\VnuOm}$ and $\vertiii{\cdot }_{\VnuOm}$ with $\vertiii{u}^2_{\VnuOm} = \|u\|^2_{L^2(\R^d)} + |u|^2_{\VnuOm}$. 
\end{enumerate}
\end{theorem}

The proofs of the first and second statement can be found in \cite{guy-thesis} and \cite{DyKi21}. The first statement is similar to a Meyers-Serrin density type result.  Note that $C_c^\infty(\overline{\Omega})$ is defined as $\{v|_{\overline{\Omega}} : v \in C^\infty_c(\R^d)\}$. The proof of the third statement is given in \cite{FSV15,Gri11} for a special choice of $\nu$ and in \cite{guy-thesis}, \cite{BGPR20} for the general case. The proof of the fourth assertion is given in \cite{FGKV20}.

\begin{remark}
Concerning the question, whether is is necessary to assume the continuity of $\partial \Omega$ for the density of $C_c^\infty(\Omega)$ in $V_{\nu,0}(\Omega|\R^d)$ or not, it is interesting to compare \cite[Remark 7]{FSV15} with \cite[Theorem 3.3.9]{CF12}.
\end{remark}
\begin{remark}
For the kernel $\nu(h)=|h|^{-d-\alpha}$, $\alpha\in(0,2)$, let $V^{\alpha/2}(\Omega|\R^d)$, $V^{\alpha/2}_0(\Omega|\R^d)$  and $H^{\alpha/2}(\Omega)$ be the spaces $\VnuOm$, $V_{\nu,0}(\Omega|\R^d)$ and $\HnuOm$ respectively. It is worth nothing that, see \cite{Gri11}, if $\Omega$ has a compact Lipschitz boundary and $\alpha\neq 1$ then 
\begin{align*}
V^{\alpha/2}_0(\Omega|\R^d)=\overline{C^\infty_c(\Omega)}^{V^{\alpha/2}(\Omega|\R^d)}=  \overline{C^\infty_c(\Omega)}^{H^{\alpha/2}(\R^d)}=\overline{C^\infty_c(\Omega)}^{H^{\alpha/2}(\Omega)},
\end{align*}where  the first and the second equality follow from Theorem \ref{thm:density} $(ii)$. Furthermore, if $0<\alpha<1$ then we also have $H^{\alpha/2}(\Omega)= \overline{C^\infty_c(\Omega)}^{H^{\alpha/2}(\Omega)}.$
\end{remark}

\subsection{Weighted $L^2$-spaces}

In order to set up the Dirichlet problem in $L^2$-spaces over $\R^d$, we define a Borel measure on $\R^d$ that captures the behavior of $\nu$ at infinity. There are several possibilities.

\begin{definition}\label{def:different-nus}
Let $\nu$ satisfies \eqref{eq:levy-cond} with full support and $B\subset \R^d$ be non-empty and open. Define the weights $\overline{\nu}, \widetilde{\nu}: \R^d \to [0,\infty]$ by  
\begin{align*}
\widetilde{\nu} (x) &= \int_B \left( 1\wedge \nu(x-y) \right) \d y, \\
\overline{\nu}(x) &= \essinf\limits_{y\in B}\nu(x-y) \,.
%\end{align*}
\intertext{If $\nu$ is a unimodal L\'{e}vy measure, then we define the Borel measure $\nu^*: \R^d \to [0,\infty]$  by} 
%\begin{align*}
\nu^*(x) &= \nu(R(1+|x|)) \,,
\end{align*}
where $R>1$ is an arbitrary fixed number. 
\end{definition}

\begin{example}\label{exa:standard}
Let $0 < \alpha < 2$ and $\nu(h) = |h|^{-d-\alpha}$ for $h \ne 0$. Let $B\subset \R^d$ be open and bounded, and $R > 1$. Then 
\begin{align*}
\widetilde{\nu} (x) \asymp \overline{\nu}(x) \asymp \nu^*(x) \asymp (1+|x|)^{-d-\alpha} \,,
\end{align*}
where the constants behind the relation $\asymp$ depend on the choice of $B$ and $R$. See \autoref{thm:comp-nu-tilde} for the general case.
\end{example}

Let us discuss important properties of the three measures $\widetilde{\nu}$, $\overline{\nu}$, and $\nu^*$.
The following lemmas show that it is possible to define certain norms on $\VnuOm $ (with nice properties), which are equivalent to the norm $\|\cdot\|_{\VnuOm}$.

\begin{lemma}[Properties of $\widetilde{\nu}$]\label{lem:natural-norm-on-V}
Let $\Omega\subset \R^d $ be open and $\nu$ satisfies \eqref{eq:levy-cond} with full support. Assume $B\subset \Omega$. 

\begin{enumerate}[$(i)$]
\item We have $\widetilde{\nu}\in L^\infty(\R^d)$. Moreover, if $|B|<\infty$, then $\widetilde{\nu}\in L^1(\R^d)$.
\item The embedding $\VnuOm \hookrightarrow L^2(\R^d,\widetilde{\nu})$ is continuous. If $|B|<\infty$, then $L^2(\R^d,\widetilde{\nu}) \hookrightarrow L^1(\R^d,\widetilde{\nu})$ is continuous.
\item If $\nu$ is unimodal and $B$ is bounded, then on $\VnuOm$, the norms $\|\cdot\|^{\#}_{\VnuOm}$ and $\|\cdot\|^{*}_{\VnuOm}$ are equivalent, where
\begin{align*}
\|u\|^{*2}_{\VnuOm}&=\int_{\R^d}  |u(x)|^2 \widetilde{\nu}(x)\d x+ \iil\limits_{(\Omega^c\times \Omega^c)^c} (u(x)-u(y))^2\nuxminy\d x\d y\,,\\
\|u\|^{\#2}_{\VnuOm}&=\int_{\Omega}  |u(x)|^2 \widetilde{\nu}(x)\d x+ \iil\limits_{(\Omega^c\times \Omega^c)^c} (u(x)-u(y))^2\nuxminy\d x\d y\,.
\end{align*} 
\end{enumerate}
Furthermore, if  $\Omega$ is bounded then the norms $\|\cdot\|_{\VnuOm}$ and $\|\cdot\|^{*}_{\VnuOm}$ are also equivalent.
\end{lemma}

\begin{proof} Firstly, we observe $1\land \nu\in L^1(\R^d)$. It follows $\widetilde{\nu}(x)\leq  \|1\land\nu\|_{L^1(\R^d)} $ for almost every $x\in \R^d$. If $|B|<\infty$,  then Fubini's theorem implies 
\begin{align*}
\int_{\R^d} \widetilde{\nu}(x)\d x\leq |B|\|1\land\nu\|_{L^1(\R^d)} < \infty \,. 
\end{align*} 

%	since \begin{align*}
%	\int_{\R^d} 1\land \nu(h)\d h\leq \int_{B_1(0)} \d h+\int_{B_1^c(0)} \nu(h)\d h<\infty.
%	\end{align*}

The continuous embedding $ L^2(\R^d, \widetilde{\nu}) \hookrightarrow L^1(\R^d, \widetilde{\nu})$ follows directly.  The continuity of the embedding 
$ \VnuOm\hookrightarrow L^2(\R^d, \widetilde{\nu})$ is obtained as follows
\begin{align*}
\int_{\R^d} |u(x)|^2\widetilde{\nu}(x)\d x
%& = \int_\Omega   \int_{\R^d}  |u(x)-u(y)+u(y)|^21\land \nu(x-y)\d y\d x\\
& \leq 2\int_{B} |u(y)|^2 \left( \int_{\R^d} 1\land \nu(x-y)\d x \right) \d y + 2 \iil_{B\R^d}  (u(x)-u(y))^2 1\land \nu(x-y)\d x\d y\\
&\leq C_1\int_{\Omega} |u(y)|^2 \d y+ C_1\iil_{\Omega\R^d}  (u(x)-u(y))^2 \nu(x-y)\d x\,\d y
= C_1\|u\|^2_{\VnuOm}.
\end{align*}
Here $C_1=2\|1\land\nu\|_{L^1(\R^d)} +2$. The following inequalities obviously hold
\begin{align*}
\sqrt{C_1 + 1} \, \|u\|_{\VnuOm}\geq  \| u\|^{*}_{\VnuOm}\geq  \| u\|^{\#}_{\VnuOm}. 
\end{align*}

Next if $\nu$ is unimodal and $B$ is bounded, then there is a constant $c'>0$ such that  $\widetilde{\nu}(x)\geq c'$ for all $x\in B$.  
The following estimates hold
\begin{align*}
\int_{\Omega} |u(x)|^2 & \widetilde{\nu}(x)\d x  + \iint\limits_{\Omega\Omega^c} (u(x)-u(y))^2\nu(x-y)\d y\d x
\\
&\geq c'\int_{B} |u(x)|^2 \d x+\iint\limits_{\Omega\Omega^c} (u(x)-u(y))^2\nu(x-y)\d y\d x
\\
&\geq c'\|1\land\nu\|_{L^1(\R^d)}^{-1}\iint\limits_{B\Omega^c}  |u(x)|^21\land \nu(x-y)\d y\d x
+\iint\limits_{B\Omega^c} (u(x)-u(y))^2\nu(x-y)\d y\d x\\
&\geq ( 1\land c'\|1\land\nu\|_{L^1(\R^d)}^{-1}) \iint\limits_{\Omega^cB}\Big[  |u(x)|^2+ (u(x)-u(y))^2 \Big] 1\land \nu(x-y)\d x\d y\\
&\geq \frac{1}{2} ( 1\land c'\|1\land\nu\|_{L^1(\R^d)}^{-1})\int_{\Omega^c}u^2(y) \widetilde{\nu}(y)\d y\, .
\end{align*}
The first and the last line imply that $\| u\|^{*}_{\VnuOm}\leq  C\| u\|^{\#}_{\VnuOm}$ for some constant $C>0$. Thus the norms $\| \cdot \|^{*}_{\VnuOm}$ and $\| \cdot \|^{\#}_{\VnuOm}$ are equivalent. If in addition $\Omega$ is bounded, then $\|1\land\nu\|_{L^1(\R^d)}\geq \widetilde{\nu}(x)\geq c'$ for all $x\in \Omega$ for some $c'>0$. The equivalence of the  norms  $\| \cdot\|_{\VnuOm}$and $\| \cdot\|^{\#}_{\VnuOm}$ is thus proved. 
%
%	 so that $\| u\|^{\#}_{\VnuOm}\geq c'\| u\|_{\VnuOm}$ since 
%	\begin{align*}
%	\int_{\Omega} |u(x)|^2\widetilde{\nu}(x)\d x\geq c'\int_{\Omega} |u(x)|^2\d x.
%	\end{align*}
%	The equivalence of the  norms  $\| \cdot\|_{\VnuOm}$and $\| \cdot\|^{\#}_{\VnuOm}$ is thus proved.  
\end{proof}

\begin{lemma}[Properties of $\overline{\nu}$]\label{lem:natural-norm-on-V-mubarbis}
Assume $\Omega\subset \R^d$ is open.  Let $\nu$ satisfies \eqref{eq:levy-cond} with full support. Assume $B\subset \Omega$ is open and nonempty. 
\begin{enumerate}[$(i)$]
\item  $\overline{\nu}\in L^1 (\R^d)$ and, if $\nu$ is unimodal, then $\overline{\nu}\in L^\infty(\R^d)$. 
\item The embeddings $\VnuOm \hookrightarrow L^2(\R^d,\overline{\nu}) \hookrightarrow L^1(\R^d,\overline{\nu})$ are continuous.
\item If $\nu$ is unimodal and $B$ is bounded then the  norms $\|\cdot\|^{\#}_{\VnuOm}$ and $\|\cdot\|^{*}_{\VnuOm}$ are equivalent, where $\|u\|^{*}_{\VnuOm}$ and $\|u\|^{\#}_{\VnuOm}$ are defined as in \autoref{lem:natural-norm-on-V}(iii) with $\widetilde{\nu}$ replaced by $\overline{\nu}$.
%		Here, 
%		\begin{align*}
%		\|u\|^{*2}_{\VnuOm}&=\int_{\R^d}  |u(x)|^2 \overline{\nu}(x)\d x+ \iil\limits_{(\Omega^c\times \Omega^c)^c} (u(x)-u(y))^2\nuxminy\d x\d y\,,\\
%		\|u\|^{\#2}_{\VnuOm}&=\int_{\Omega}  |u(x)|^2 \overline{\nu}(x)\d x+ \iil\limits_{(\Omega^c\times \Omega^c)^c} (u(x)-u(y))^2\nuxminy\d x\d y\,.
%		\end{align*} 
\end{enumerate}
Furthermore, if $\Omega$ is bounded then the norms $\|\cdot\|_{\VnuOm}$ and $\|\cdot\|^{*}_{\VnuOm}$ are equivalent.
\end{lemma}

\begin{proof}
To prove $(i)$, select $x_0\in B$ and $r>0$ such that $B_{8r}(x_0)\subset B$. Note that, for  $x\in B_{2r}(x_0)$ and $y\in B^c_{4r}(x_0)\cap B$ or for $x\in B^c_{2r}(x_0)$ and $y\in B_{r}(x_0)$ we have  $|h|\geq r$ where  $h=x-y$. Therefore, for almost all $\xi\in B^c_{4r}(x_0)$ and $z\in  B_{r}(x_0)$ we find that 
\begin{align*}
\int_{B_{2r}(x_0)}\underset{y\in B^c_{4r}(x_0)\cap B}{\essinf}\nu(x-y)\d x
&\leq 
\int_{|x-y|\geq r}\nu(x-\xi)\d x=
\int_{|h|\geq r}\nu(h)\d h,
\\
\int_{B^c_{2r}(x_0)}\underset{y\in B_{r}(x_0)}{\essinf}\hspace{2mm}\nu(x-y)\d x
&\leq 
\int_{|x-z|\geq r}\nu(x-z)\d x=
\int_{|h|\geq r}\nu(h)\d h.
\end{align*}
The integrability of $\overline{\nu}$ follows since 
\begin{align*}
\int_{\R^d}\overline{\nu}(x)\d x&\leq \int_{B_{2r}(x_0)}\underset{y\in B^c_{4r}(x_0)\cap B}{\essinf}\nu(x-y)\d x+ \int_{B^c_{2r}(x_0)}\underset{y\in B_{r}(x_0)}{\essinf} \nu(x-y)\d x
%&\leq %\mathbbm{1}_{B^c_{4r}(x_0)\cap B}(y)
%\int_{|x-y|\geq r}\nu(x-y)\d x+ %\mathbbm{1}_{B_{r}(x_0)}(y)
%\int_{|x-z|\geq r}\nu(x-z)\d x, \quad \forall \, y\in B^c_{4r}(x_0)\cap B, \, z\in B_{r}(x_0) \\
\leq 2\int_{|h|\geq r}\hspace{-2ex}\nu(h)\d h<\infty.
\end{align*}
Analogously, assume $\nu$ is unimodal. Since  for  $(x,y)\in B_{2r}(x_0)\times B^c_{4r}(x_0)\cap B$ or for $(x,y)\in B^c_{2r}(x_0)\times B_{2r}(x_0)$  we have $|x-y|\geq r$, it follows that   $\overline{\nu}_B(x)\leq \nu(x-y)\leq c\nu(r)$.
%\begin{align*}
% \overline{\nu}(x)
%&\leq \mathbbm{1}_{B_{2r}(x_0)}(x) \underset{y\in B^c_{4r}(x_0)\cap B}{\essinf}\nu(x-y)+ \mathbbm{1}_{B^c_{2r}(x_0)}(x) \underset{y\in B_{r}(x_0)}{\essinf} \nu(x-y) \leq 2c\nu(r).
%\end{align*}

\noindent Next, consider $K'\subset B$ to  be a measurable subset such that  $0<|K'|<\infty$, then we have  
\begin{align*}
\int_{\R^d}|u(x)|^2 \overline{\nu} (x)\d x
&\leq 2|K'|^{-1}\|\overline{\nu}\|_{L^1(\R^d)} \int_{K'}|u(y)|^2\d y+ 2 |K'|^{-1}\iil_{K'\R^d}(u(x)-u(y))^2\nuxminy\d x \d y\\
&\leq C\|u\|^2_{\VnuOm}, \quad\qquad 
\end{align*}
where $C=2|K'|^{-1}( \|\overline{\nu}\|_{L^1(\R^d) }+1).$  This together with the previous step imply the continuity of the embeddings  $\VnuOm \hookrightarrow L^2(\R^d,\overline{\nu})\hookrightarrow L^1(\R^d,\overline{\nu})$. The rest of the proof is analogous to  that of \autoref{lem:natural-norm-on-V}.
\end{proof}

\begin{lemma}[Properties of $\nu^*$]\label{lem:natural-norm-on-V-mustarbis}
Assume $\Omega\subset \R^d$ is open and $R > 1$ satisfies $|B_R(0) \cap \Omega| > 0$ and $|B_R(0) \cap \Omega^c| > 0$.  Let $\nu$ satisfies \eqref{eq:levy-cond} with full support. 
\begin{enumerate}[$(i)$]
\item  $\nu^*\in L^1 (\R^d)\cap L^\infty(\R^d)$. 
\item The embeddings $\VnuOm \hookrightarrow L^2(\R^d,\nu^{*}) \hookrightarrow L^1(\R^d,\nu^{*})$ are continuous.
\item On $\VnuOm$, the norms $\|\cdot\|^{\#}_{\VnuOm}$ and $\|\cdot\|^{*}_{\VnuOm}$ are equivalent, where  $\|u\|^{*}_{\VnuOm}$ and $\|u\|^{\#}_{\VnuOm}$ are defined as in \autoref{lem:natural-norm-on-V}(iii) with $\widetilde{\nu}$ replaced by $\nu^*$.
\end{enumerate}
Furthermore, if $\Omega$ is bounded then the norms $\|\cdot\|_{\VnuOm}$ and $\|\cdot\|^{*}_{\VnuOm}$ are equivalent.
\end{lemma}

The proof of \autoref{lem:natural-norm-on-V-mustarbis} is analogous to  that of  \autoref{lem:natural-norm-on-V} and can be found in \cite{FGKV20} or \cite[Lemma 3.24]{guy-thesis}. In order to show $\nu^*\in L^1(\R^d)\cap L^\infty(\R^d)$, one notes that for all $h\in \R^d$ we have $R\leq R(1+|h|)$ and $|h|\leq R(1+|h|)$ and hence $\nu^*(h)\leq C(1\land \nu(h))$. This implies the claim because of  $1\land\nu\in L^1(\R^d)\cap L^\infty(\R^d)$. 

\medskip

The measures $\nu^*$, $\overline{\nu}$ and $\widetilde{\nu}$ turn out to be comparable if $\nu$ satisfies a certain doubling condition at infinity.

\begin{definition}\label{def:doublingatinfinity}
A radial L\'{e}vy density $\nu$  satisfies a doubling condition at infinity if:
\begin{align}\label{eq:global-scaling_infinity}
\text{For every } \theta\geq 1 \text{ there exist } c_1, c_2 > 0 \text{ with } c_1\nu(r)\leq \nu(\theta r)\leq c_2\nu(r) \text{ for all } r\geq 1\,.
\end{align}
Not that the property \eqref{eq:global-scaling_infinity} is indeed equivalent  to saying that 
\begin{align}\label{eq:doubling-scaling_infinity}
\text{ There exist } c_1, c_2 > 0 \text{ with } c_1\nu(r)\leq \nu(2 r)\leq c_2\nu(r) \text{ for all } r\geq 1. 
\end{align}
\end{definition}

\begin{remark}
The doubling condition at infinity imposes some decay of $\nu$ at infinity. The example $\nu(h) = |h|^{-d-1} \mathds{1}_{\{|h| \leq 7\}}$ satisfies \autoref{def:unimodality} but not \eqref{eq:global-scaling_infinity}. Unimodality bounds one-sided oscillations of $\nu$ for all values of $|x|$. Fix $0 < \alpha < \beta < 2$. Define $\nu_1(r) = r^{-d-\alpha}$ for $r \geq 1$. Define $\nu_1(r) = r^{-d-\beta}$ for $\frac{1}{2k+1} \leq r < \frac{1}{2k}$ and $\nu_1(r) = r^{-d-\alpha}$ for $\frac{1}{2k+2} \leq r < \frac{1}{2k+1}$ for $k \in \mathbb{N}$. Then $\nu_1$ is not unimodal but it trivially does satisfy \eqref{eq:global-scaling_infinity}.
\end{remark}

\begin{theorem}\label{thm:comp-nu-tilde}
Assume that $\nu$ is unimodal and satisfies \eqref{eq:global-scaling_infinity}. For $B\subset \R^d$ is bounded  (e.g. B is a ball) and $R\geq 1$ we have $\widetilde{\nu}(x) \asymp \overline{\nu}(x)\asymp \nu^*(x)\asymp 1\land \nu(x)$.
\end{theorem}

\begin{proof}
Let us observe that, $ \widetilde{\nu},\overline{\nu},\nu^*$ and $1\land\nu$ are all bounded above.  Indeed, for $x\in \R^d$ we have $R\leq R(1+|x|)$ and $|x|\leq R(1+|x|)$ and hence $\nu^*(x)\leq C(1\land \nu(x))\leq C$. Obviously, $\widetilde{\nu}(x)\leq \|1\land\nu\|_{L^1(\R^d)}$.  Now, let $r>0$ sufficiently small and $x_0\in B$ such that  $B_{8r}(x_0)\subset B$. If $x\in B_{2r}(x_0)$ and $y\in B^c_{4r}(x_0)\cap B$ or if $x\in B^c_{2r}(x_0)$ and $y\in B_{2r}(x_0)$  then $|x-y|\geq r$. In both cases, $\overline{\nu}(x)\leq \nu(x-y)\leq c\nu(r)$.

Next, there is no loss of generality if we assume that $B\subset B_R.$ Assume  $|x|\leq 4R$ then $|x-y|\leq 5R$ for all  $y\in B$. The unimodality and the foregoing boundedness imply that   $c^{-1}\nu(R(1+R))\leq \nu(R(1+|x|))\leq C$, $c^{-1}\nu(5R)\leq \overline{\nu}(x)$, $1\land \nu(4R)\leq c 1\land \nu(x)\leq C$ and that $1\land \nu(5R)\leq c 1\land \nu(x-y)$ that is $|B|1\land \nu(5R)\leq c \widetilde{\nu}(x)\leq C.$  Thus $\widetilde{\nu}(x)\asymp \overline{\nu}(x)\asymp \nu^*(x)\asymp 1\land\nu(x)$ for $|x|\leq 4R$.

Now assume that $|x|\geq 4R,$ then we have $1\leq \frac{|x|}{2}\leq |x-y|\leq 2|x|$ for all $y\in B$ and $|x|\leq R(1+|x|)\leq 2R|x|$. The doubling condition \eqref{eq:global-scaling_infinity} implies that for some constants $0<c_1<1<c_2$,  we have $c_1\nu(x)\leq \nu(R(1+|x|))\leq c_2\nu(x)$ and $c_1\nu(x)\leq \nu(x-y)\leq c_2\nu(x)$ for all $y\in B$. We get $c_1\nu(x)\leq \overline{\nu}(x)\leq c_2\nu(x)$ and integrating over $B$ implies that $c_1|B|1\land \nu(x)\leq c \widetilde{\nu}(x)\leq c_2|B|1\land \nu(x).$  Together  with the boundedness implies $\widetilde{\nu}(x)\asymp \overline{\nu}(x) \asymp\nu^*(x)\asymp 1\land\nu(x)$ for $|x|\geq 4R$.

%as $1\land \nu$ is also unimodal, we get
%$R\leq R(1+|x|)\leq R(1+R) $ so that

\end{proof}

\begin{example}
As in Example \ref{exa:standard}, let $0 < \alpha < 2$ and $\nu(h) = |h|^{-d-\alpha}$ for $h \ne 0$. Then $\widetilde{\nu}(x) \asymp 1\land \nu (x) \asymp (1+|h|)^{-d-\alpha}$ and the space $\HnuOm$ equals the classical Sobolev-Slobodeckij space $H^{\alpha/2}(\Omega)$. In this case, we denote the space $\VnuOm$ by $V^{\alpha/2}(\Omega|\R^d)$. We have $V^{\alpha/2}(\Omega|\R^d) \hookrightarrow L^2(\R^d, (1+|h|)^{-d-\alpha})$.
\end{example}

\begin{remark}{\ } Note that  \autoref{lem:natural-norm-on-V} $(i)$,  \autoref{lem:natural-norm-on-V-mubarbis} $(i)$ and  \autoref{lem:natural-norm-on-V-mustarbis} $(i)$ imply that the weights $\widetilde{\nu}, \nu^*, \overline{\nu}$ respectively define Radon measures on $\R^d$. 
\end{remark}

\subsection{Dirichlet forms}
The discussion of the $L^2$-spaces related to $\widetilde{\nu}, \nu^*, \overline{\nu}$ together with density results in \autoref{thm:density} allows us to define a new interesting Dirichlet form. We refer to \cite{FOT11} for the general theory of Dirichlet forms and their corresponding Markov processes.  

\medskip

The following well-known result is a direct consequence of \autoref{thm:density}, (ii) and (iii) .

\begin{proposition}{\ }
Let $\Omega$ be open and bounded with a continuous boundary. Let $\nu$ be any L\'{e}vy measure. Then each of the three bilinear forms $(\cE, V_{\nu,0}(\Omega|\R^d))$, $(\cE_\Omega, \overline{C^\infty_c(\Omega)}^{\HnuOm})$ and $(\cE_\Omega, \HnuOm)$ is a regular Dirichlet form on $L^2(\Omega)$. The corresponding Markov processes are often called \emph{killed}, \emph{censored} resp. \emph{reflected} L\'{e}vy process.
\end{proposition}

An important side result of our work is the following theorem, which implies the existence of a strong Markov process, which can be seen as another kind of \emph{reflected} jump process with regards to $\Omega$. The theorem is an improvement over \cite[Theorem 4.4]{Von21}, see  Remark \ref{rem:vondra-strange}. Its proof follows from Theorem  \ref{thm:density}, (iv).

\begin{theorem}\label{thm:Dform-reflected}{\ } Let $\nu$ be unimodal with full support and $\Omega \subset \R^d$ be open and bounded with a Lipschitz-continuous boundary. Let $\nu'$ be any of the measures $\widetilde{\nu}, \nu^*, \overline{\nu}$ on $\R^d$. Then the bilinear form $(\cE, \VnuOm)$ is a regular Dirichlet form on $L^2(\R^d, \nu')$.
\end{theorem}

\subsection{Classical Sobolev spaces}

Let us comment on the connection of the spaces under consideration with classical Sobolev spaces. Recall that for an open set $\Omega\subset \R^d$, $H^1(\Omega)$ denotes the classical Sobolev space endowed with the norm 
\[ \|u\|^2_{H^1(\Omega)} = \|u\|^2_{L^2(\Omega)} +\|\nabla u\|^2_{L^2(\Omega)} \,. \]

\begin{proposition} The following embeddings hold true:
\begin{align*}
H^1(\R^d) \hookrightarrow H_{\nu}(\R^d) \hookrightarrow \VnuOm \hookrightarrow \HnuOm \hookrightarrow L^2(\Omega)\,.
\end{align*}
Here we equip $\VnuOm $ with the norm $\|\cdot \|_{\VnuOm}$.
\end{proposition}

\begin{proof}
The proof is standard. For $u \in H^1(\R^d)$ and $h\in \R^d$ we have 
\begin{align*}
\int_{\R^d} (u(x+ h)-u(x))^2\d x= \|u(\cdot+ h)-u(\cdot)\|^2_{L^2(\R^d)} \leq 4(1\land |h|^2)\| u \|^2_{H^1(\R^d)}\,.
\end{align*} 
Integrating both sides over $\R^d$ with respect to the measure $\nu(h)\d h$ yields 
\begin{align}\label{eq:H1-inequality}
\iint\limits_{\R^d\R^d} (u(x)-u(y))^2\nuxminy\d x\d y \leq 4\| \nu\|_{L^1(\R^d, 1\land |h|^2\d h)}  \| u \|^2_{H^1(\R^d)}\,.
\end{align} 
This proves the first embedding, the remaining ones are trivial.  
\end{proof}

Recall that $H^1_0(\Omega)$ is the closure of $C_c^\infty(\Omega)$ with respect to the $H^1(\Omega)$. The space $H^1_0(\Omega)$ also coincides with the closure of $C_c^\infty(\Omega)$ in $H^1(\R^d)$. In addition, the zero extension to $\R^d$ of any function in $ H_0^1(\Omega)$ belongs to $H^1(\R^d)$. Recall the definition of $V_{\nu,0}(\Omega|\R^d)$ from \eqref{eq:def-VnuOm-vanish}. 

\begin{proposition} 
Let $\Omega \subset \R^d$ be open. The following embeddings hold true:
\begin{align*}
&H^1_0(\Omega)\hookrightarrow V_{\nu,0}(\Omega|\R^d) \hookrightarrow H_{\nu}(\Omega) \,\\
\intertext{ where elements of $H_0^1(\Omega)$ are extended by zero off $\Omega$. If additionally $\partial \Omega$ is continuous, }
&V_{\nu,0}(\Omega|\R^d)\hookrightarrow \overline{C^\infty_c(\Omega)}^{\HnuOm} \hookrightarrow L^2(\Omega) \,.
\end{align*}
\end{proposition}

It is worth noticing that not every function $u \in \overline{C^\infty_c(\Omega)}^{\HnuOm}$ has its extension by zero in $H_{\nu}(\R^d)$. Indeed for this to hold, one would need
\begin{align*}
\int_{\Omega}|u(x)|^2\d x\int_{\Omega^c} \nuxminy\d y<\infty. 
\end{align*}
This condition is not always true since the measure $\nu$ might be very singular at the origin. This observation shows that for some appropriate domain $\Omega$ and for some appropriate measure $\nu$, e.g. $\nu(h) = |h|^{-d-3/2}$, the spaces $\overline{C^\infty_c(\Omega)}^{\HnuOm}$ and $ V_{\nu,0}(\Omega|\R^d)$ are strictly different although they both possess $C_c^\infty(\Omega)$ as dense subspace. This effect is purely nonlocal. Recall that elements of $H_0^1(\Omega)$ can be isometrically extended by zero on $\R^d$ as functions of $H^1(\R^d)$. 

\medskip

Now assume $\Omega$ is a Lipschitz domain (or more generally an $H^1$-extension domain). Let $\overline{u}\in H^{1}(\R^d)$ be an extension of a function $u\in H^{1}(\Omega) $ with			 $\|\overline{u}\|_{H^1(\R^d)} \leq C \|u\|_{H^1(\Omega)}$ for some constant $C$ depending only on $\Omega$ and $d$. Within the estimate \eqref{eq:H1-inequality} we easily get the following continuous embedding 
\begin{proposition} Assume $\Omega\subset \R^d$ is an $H^1$-extension domain then 
\begin{align*}
H^1(\Omega)\hookrightarrow H_{\nu}(\Omega)\,. 
\end{align*}
\end{proposition}
The latter embedding may fail when $\Omega$ is not an extension domain (see \cite[Counterexample 3.8]{Fog21} or 
\cite[Example 9.1]{Hitchhiker}). Note that $H^1(\Omega)$ can be viewed as limiting space of a sequence of nonlocal spaces of type $\HnuOm$ and $\VnuOm$ see \cite{Fog21,Fog21b,Fog23s} for additional results.

\subsection{Trace space of $\VnuOm$}\label{subsec:trace-space}

The main goal of this subsection is to discuss the trace space for $\VnuOm$ similarly as one does for the classical Sobolev space $H^1(\Omega)$. Note that elements of $\VnuOm$ are defined on the whole of $\R^d$, thus the trace space consists of functions defined on $\Omega^c$. This contrasts with the local situation, where the trace space of $H^1(\Omega)$ consists of functions defined on the boundary $\partial\Omega$. Unless otherwise stated, we assume that $\nu$ is fully supported on $\R^d$ and $\VnuOm$ is endowed with the norm $\|\cdot \|_{\VnuOm} $. Note that, when studying the fractional Laplace operator, trace spaces related to $\VnuOm$ when $\nu(h) = |h|^{-d-\alpha}$ for $h \ne 0$ and $0 < \alpha < 2$, have already been considered in \cite{DyKa19}, \cite{BGPR20} and \cite{DTCY21}. Below, we comment on how our general approach relates to these studies. 
\begin{definition}
We define $\TnuOm$ as the vector space of restrictions to $\R^d\setminus \Omega$ of functions of $\VnuOm$, i.e., 
\begin{align*}
\TnuOm = \{v: \Omega^c\to \mathbb{R} \text{ measurable } \; | \; v = u|_{\Omega^c} \text{ with } u \in \VnuOm\}.
\end{align*}
We endow $\TnuOm $ with its natural norm, 
\begin{align*}
\|v\|_{\TnuOm } = \inf\{ \|u\|_{\VnuOm }: ~~ u \in \VnuOm ~~ \hbox{ with }~~ v = u|_{\Omega^c} \}. 
\end{align*}
\end{definition}

\begin{theorem}
The space $\TnuOm$ is a separable Hilbert space with the scalar product
\begin{align*}
(u,v)_{\TnuOm}= \frac{1}{4}\Big( \|u+v\|^{2}_{\TnuOm }-\|u-v\|^{2}_{\TnuOm } \Big).
\end{align*}
\end{theorem}
\begin{proof}
Since the norm $\|\cdot \|_{\VnuOm }$  verifies the parallelogram law so does $\|\cdot \|_{\TnuOm }$. Thus $\big(\cdot, \cdot\big)_{\TnuOm}$ is a scalar product on $\TnuOm$ with associated norm $\|\cdot \|_{\TnuOm }$. Noting that $\TnuOm$ and the quotient space $\VnuOm/V_{\nu,0}(\Omega|\R^d)$ are identical with equal norm in space  and that $V_{\nu,0}(\Omega|\R^d)$ is a closed subspace of $\VnuOm,$ one concludes that $\TnuOm$ is complete.
\end{proof}

\medskip

The main question now is whether the same space $\TnuOm$ can be defined intrinsically. In other words, given a measurable function $v: \Omega^c \to \mathbb{R}$, how can one decide whether the function belongs to $\TnuOm$ or not. In the local situation, it is possible to define a scalar product on the space $H^{1/2}(\partial \Omega)$ when $\Omega$ is a Lipschitz domain, see \cite{Din96} for a proof. 

\medskip

We study this question in two settings, the one of \cite{DyKa19} and the one of \cite{BGPR20}. For the special case $\nu(h)= |h|^{-d-\alpha}$ for $h \ne 0$ with $0 < \alpha < 2$, it is proved in \cite[Theorem 3]{DyKa19} that for $v\in \VnuOm$ it holds
\begin{align}\label{eq:finiteness-ext}
\iil\limits_{\Omega^c \Omega^c} \frac{\big(v(x)-v(y) \big)^2 }{(|x-y|+\delta_x+\delta_y)^{d+\alpha}} \d x\d y<\infty\,.
\end{align}
Moreover, it is shown that, if \eqref{eq:finiteness-ext} holds true for $v=g$ on $\Omega^c$, there exists $u_g\in \VnuOm$ such that $u_g|_{\Omega^c}=g$ and 
\begin{align}\label{eq:compare-semi-douglas}
\iil\limits_{(\Omega^c \times \Omega^c)^c} \frac{\big(u_g(x)-u_g(y) \big)^2 }{|x-y|^{d+\alpha}} \d x\d y \asymp \iil\limits_{\Omega^c \Omega^c} \frac{\big(v(x)-v(y) \big)^2 }{(|x-y|+\delta_x+\delta_y)^{d+\alpha}} \d x\d y
\end{align}
% \todo[inline]{Danger! Only one inequality might be true}
with the constants independent of $g$ and $u_g$. Therefore we obtain in this case
\begin{align}\label{eq:trace-DK20}
\TnuOm= \Big\{ v:\Omega^c\to \mathbb{R}~\text{meas.} \big | \, \iil_{\Omega^c\Omega^c} \!\!\frac{\big(v(x)-v(y) \big)^2 }{(|x-y|+\delta_x+\delta_y)^{d+\alpha}}\, \d x \, \d y<\infty \Big\}\,, 
\end{align}
which we will make use of in \autoref{prop:trace-comparability}.

\medskip

Next, let us summarize the results of  \cite{BGPR20}, which are established under the following condition: One assumes that $\nu$ is radial and its profile $\nu \in C^2 ((0,\infty))$ satisfies for some $C_1, C_2>0$ and $\beta \in (0,2)$ 
\begin{align}\label{eq:cond-BGPR}\tag{$A_\nu$}
\begin{split}
|\nu'(r)| +|\nu''(r)| &\leq C_1 \nu(r) \qquad (r > 1) \,, \\
\nu(\lambda r) &\leq C_2 \lambda^{d-\beta}\nu(r)\qquad  (0< r,\lambda\leq 1),\\
\nu(r) &\leq C_2\nu(r+1) \qquad (r\geq 1) \,.
\end{split}
\end{align}

Assume $\Omega^c$ satisfies the volume density condition (in some works, $\Omega^c$ is called a $d$-set), i.e., there exists a constant $c>0$ such that $|\Omega^c\cap B_r(x)|\geq cr^d$ for all $x\in \partial\Omega$ and all $r>0$. By the Lebesgue density theorem, the latter condition automatically implies that $|\partial\Omega|=0$. Then under \eqref{eq:cond-BGPR} \cite[Theorem 2.3]{BGPR20} proves that, for any $g\in \TnuOm$ there  exists a unique $u_g\in\VnuOm $ such that $u_g|_{\Omega^c}= g$ with 
\begin{align}\label{eq:douglas-formula}
\mathcal{H}_\Omega(g,g):= \iil_{\Omega^c\Omega^c} \!\!\big(g(x)-g(y) \big)^2 \, \gamma_\Omega (x, y) \d x \, \d y=\hspace*{-2ex} \iil_{(\Omega^c\times \Omega^c)^c} \!\!\big(u_g(x)-u_g(y) \big)^2 \, \nuxminy \d x \, \d y. 
\end{align}
The function $u_g$ satisfies the weak formulation 
\begin{align}
\iil_{(\Omega^c\times \Omega^c)^c} \!\!\big(u_g(x)-u_g(y) \big)(\phi(x)-\phi(y)) \, \nuxminy \d x \, \d y=0 \quad \text{ for all } \phi\in V_{\nu, 0}(\Omega|\R^d)
\end{align}
and the interaction kernel $\gamma_\Omega (x, y)$ is given via the Poisson kernel of $\Omega$ by the formula
\begin{align*}
\gamma_\Omega (x, y) = \int_{\Omega}P_\Omega(x,z)\nu(z-y) \d z\,\qquad x,y\in \Omega^c \,.
\end{align*}
Furthermore, a precise formula for $u_g$ in $\Omega $ is given by the Poisson integral
\begin{align*}
u_g(x)= P_\Omega[g](x) = \int_{\Omega^c}g(y)P_\Omega(x,y)\d y\, \qquad x\in \Omega\,. 
\end{align*}
From this, it is easy to show that 
\begin{align*}
\TnuOm= \Big\{ v:\Omega^c\to \mathbb{R}~\text{meas.}~~\mathcal{H}_{\Omega}(v,v)= \iil_{\Omega^c\Omega^c} \!\!\big(v(x)-v(y) \big)^2 \, \gamma_\Omega (x, y) \d x \, \d y<\infty \Big\}\,
\end{align*}
which is precisely the exterior space introduced in \cite{BGPR20}.
With this definition, the connection between $\TnuOm$ and $\VnuOm$ is less visible. For $v\in \TnuOm$, by definition of $\|\cdot\|_{\TnuOm }$, we have 
\begin{align*}
\|v\|^{2}_{\TnuOm } &= \inf\{ \|u\|^{2}_{\VnuOm } :\, u \in \VnuOm \text{ with } v = u|_{\Omega^c} \}\\
&\geq \inf\Big\{ \int_{\Omega}|u(x)|^2\d x: \, u \in \VnuOm \text{ with } v = u|_{\Omega^c} \Big\}+ \mathcal{H}_{\Omega}(v,v)\,.
\end{align*}
It is rather challenging to find or to estimate the quantity 
\begin{align*}
\inf\Big\{ \int_{\Omega}|u(x)|^2\d x: ~~ u \in \VnuOm ~~ \hbox{with}~ v = u|_{\Omega^c} \Big\}.
\end{align*}
Remember that our goal here is to explicitly define a norm which is equivalent to $\|\cdot\|_{\TnuOm }$ and has less visible connection to $\VnuOm$. To this end, we bring into play the norm $\|\cdot\|^{*}_{\VnuOm }$ defined in \autoref{lem:natural-norm-on-V}. 
\medskip

\begin{proposition}\label{prop:trace-comparability}
Assume $\Omega$ is open and bounded, such that $\Omega^c$ satisfies the volume density condition. Assume $\nu$ satisfies \eqref{eq:cond-BGPR}. Let $\widetilde{\nu}$ and $\|\cdot\|^{*}_{\VnuOm }$ be respectively the measure and the norm given in \autoref{lem:natural-norm-on-V}. Then 
\begin{align*}
\TnuOm= \Big\{ v:\Omega^c\to \mathbb{R}~\text{meas.}~~\mathcal{H}_{\Omega}(v,v)= \iil_{\Omega^c\Omega^c} \!\!\big(v(x)-v(y) \big)^2 \, \gamma_\Omega (x, y) \d x \, \d y<\infty \Big\}\,
\end{align*}
and the norms $\|\cdot\|_{\TnuOm }$, $\|\cdot\|^{*}_{\TnuOm }$ and $\|\cdot\|^{\dagger}_{\TnuOm }$ are all equivalent, where
\begin{align*}
\|v\|^{*}_{\TnuOm } &= \inf\{ \|u\|^{*}_{\VnuOm } :~~ u \in \VnuOm ~~ \hbox{with}~~ v = u|_{\Omega^c} \}\\
\|v\|^{\dagger 2}_{\TnuOm } &=\int_{\Omega^c} |v(x)|^2\widetilde{\nu}(x)\d x+ \iil_{\Omega^c\Omega^c} \!\!\big(v(x)-v(y) \big)^2 \, \gamma_\Omega (x, y) \d x \, \d y\,.
\end{align*}
Next, consider $\nu(h)= (2-\alpha)|h|^{-d-\alpha}$ for $h \ne 0$ with $0 < \alpha < 2$ fixed and $\widetilde{\nu}(h) = \frac{1}{(1+|h|)^{d+\alpha}}$. Set $\delta_z = \operatorname{dist}(z, \partial\Omega)$. Then 
\begin{align*}
\TnuOm= \Big\{ v:\Omega^c\to \mathbb{R}~\text{meas.} \big| \; \iil_{\Omega^c\Omega^c} \!\!\frac{\big(v(x)-v(y) \big)^2 }{(|x-y|+\delta_x+\delta_y)^{d+\alpha}}\, \d x \, \d y<\infty \Big\}\,
\end{align*}
and the aforementioned norms are equivalent to the norm 
\begin{align*}
\|v\|^{'2}_{\TnuOm} = \int_{\Omega^c} \frac{|v(x)|^2}{(1+|x|)^{d+\alpha}}\d x+ \iil\limits_{\Omega^c \Omega^c} \frac{\big(v(x)-v(y) \big)^2 }{(|x-y|+\delta_x+\delta_y)^{d+\alpha}} \d x\d y\,.
\end{align*}

\end{proposition}

\begin{remark}
In the case $\nu_\alpha(h)=(2-\alpha)|h|^{-d-\alpha}$ for $h \ne 0$ with given $\alpha \in (0,2)$, it is interesting to understand the limiting behaviour of the comparability estimate for $\|v\|^{'}_{\TnuOma}$ and  $\|v\|^{*}_{\TnuOma}$ as $\alpha \to 2^-$. Recent upcoming results of Th. Hensiek and F. Grube show that one can modify the norm $\|v\|^{'}_{\TnuOm}$ so that $\|v\|^{'}_{\TnuOma}$ would converge to $\|v\|_{H^{1/2}(\partial \Omega)}$ as $\alpha \to 2^-$.
\end{remark}

\begin{proof}
The equivalence between $\|\cdot\|_{\TnuOm }$ and $\|\cdot\|^{*}_{\TnuOm }$ is an immediate consequence of \autoref{lem:natural-norm-on-V}. By \eqref{eq:douglas-formula} it follows that, 
\begin{align*}
\|v\|^{*2}_{\TnuOm } &= \inf\{ \|u\|^{*2}_{\VnuOm }: \, u \in \VnuOm \text{ with } v = u|_{\Omega^c} \}\\
&\geq \inf\Big\{ \int_{\R^d}|u(x)|^2\widetilde{\nu}(x) \d x : \, u \in \VnuOm  \text{ with } v = u|_{\Omega^c} \Big\}+ \mathcal{H}_{\Omega}(v,v)\\
&\geq \int_{\Omega^c} |v(x)|^2\widetilde{\nu}(x)\d x+ \mathcal{H}_{\Omega}(v,v) \,,
\end{align*}
which establishes $\|v\|^{\dagger}_{\TnuOm }\leq \|v\|^{*}_{\TnuOm }$. Hence the identity $I: (\TnuOm, \|\cdot\|^{*}_{\TnuOm }) \to (\TnuOm, \|\cdot\|^{\dagger}_{\TnuOm })$ is continuous. The space $(\TnuOm, \|\cdot\|^{*}_{\TnuOm })$ is a Hilbert space since $\|\cdot\|_{\TnuOm }$ and $\|\cdot\|^{*2}_{\TnuOm }$ are equivalent. Also, using the Fatou lemma one can easily show that $(\TnuOm, \|\cdot\|^{\dagger}_{\TnuOm })$ is a Hilbert space. As consequence of the open mapping theorem the norms $ \|\cdot\|^{\dagger}_{\TnuOm }$ and $ \|\cdot\|^{*}_{\TnuOm }$ are equivalent. 

Next, let us consider $\nu(h)= (2-\alpha)|h|^{-d-\alpha}$ for $h \ne 0$ with $0 < \alpha < 2$ fixed. From \cite[Theorem 3]{DyKa19}, see \eqref{eq:compare-semi-douglas} and \eqref{eq:trace-DK20},  we conclude that there exists a constant $C>0$ such that for all $v \in \TnuOm$, $\|v\|^{*}_{\TnuOm} \leq C \|v\|^{'}_{\TnuOm} $. The equivalence between $\|\cdot \|^{*}_{\TnuOm} $ and $ \|\cdot \|^{'}_{\TnuOm} $ follows once again by the open mapping theorem.
\end{proof}

\medskip

\begin{remark}
We emphasize that the nonlocal trace does not need any special construction via 
functional analysis or density arguments. Since $\Omega^c$ is a $d$-dimensional manifold, it makes sense to consider the restriction of a measurable function on $\Omega^c$. No regularity of $\Omega$ resp. $\partial \Omega$ is required. In the classical local situation, the definition of a trace of a Sobolev function $u$ on the boundary $\partial\Omega$ requires some smoothness of both, $u$ and $\partial\Omega$.
\end{remark}

\medskip

Let us collect some basics results results concerning the trace space $\TnuOm$. With the aid of \autoref{lem:natural-norm-on-V} we get the following. 

\begin{proposition}
The trace map $\operatorname{Tr} : \VnuOm \to L^2(\Omega^c, \widetilde{\nu})$ with $u \mapsto \operatorname{Tr}(u) = u\mid_{\Omega^c}$ has the following properties: (a) $\operatorname{Tr}(\VnuOm)= \TnuOm$, (b) $ \ker(\operatorname{Tr} ) = V_{\nu,0}(\Omega|\R^d)$ and (c) $\operatorname{Tr}$ is linear and continuous. Moreover, $ \TnuOm$ is dense in $ L^2(\Omega^c, \widetilde{\nu})$. 
\end{proposition}
\begin{proof}
This is indeed, is a direct consequence of \autoref{lem:natural-norm-on-V} since $u\in L^2(\R^d, \widetilde{\nu})$ for all $u \in \VnuOm$ so that $ \operatorname{Tr}( u) \in L^2(\Omega^c, \widetilde{\nu})$ in particular $\operatorname{Tr}$ is well defined. Moreover, by \autoref{lem:natural-norm-on-V} there exists a constant $C>0$ such that,
\begin{align}\label{eq:trace-inequality} 
\| \operatorname{Tr}( u)\|_{L^2(\Omega^c, \widetilde{\nu})} \leq \| u\|_{L^2(\R^d, \widetilde{\nu})} \leq C \| u\|_{\VnuOm} \qquad \text{for all }~~~u \in \VnuOm. 
\end{align}
The zero extension to $\Omega$ of elements $C_c^\infty(\overline{\Omega}^c )$ are in $\VnuOm$. Thus $C_c^\infty(\overline{\Omega}^c )$ is contained in $\TnuOm$ which implies that $\TnuOm$ is dense in $ L^2(\Omega^c, \widetilde{\nu})$ since $C_c^\infty(\overline{\Omega}^c )$ is dense in $ L^2(\Omega^c, \widetilde{\nu})$ 
\end{proof}
\begin{remark}
One may view the objects $L^2(\Omega^c, \widetilde{\nu})$, $\TnuOm$, $\VnuOm$ and $V_{\nu,0}(\Omega|\R^d)$ respectively as the nonlocal counterpart of $L^2(\partial\Omega)$, $H^{1/2}(\partial\Omega)$, $H^1(\Omega)$ and $H^1_0(\Omega)$.  Indeed,  $(i)$ the classical trace operator $\gamma_0 :H^1(\Omega)\to L^2(\partial\Omega) $ whenever it exists is linear continuous, $(ii)$ $\gamma_0(H^1(\Omega))= H^{1/2}(\partial\Omega)$ and $(iii)$ $\ker(\gamma_0)= H^1_0(\Omega)$. 
\end{remark}

\begin{proposition}
Let $C^\infty_c(\overline{\Omega^c} )=C^\infty_c(\R^d)|_{\Omega^c} $ be set of restrictions on $\Omega^c$ of $C^\infty$ functions on $\R^d$ with compact support. If $\Omega$ is bounded and Lipschitz then $C^\infty_c(\overline{\Omega^c} )$ is dense in $\TnuOm$. 
\end{proposition}

\begin{proof}
For $v \in \TnuOm$ we write $v= u|_{\Omega^c}$ with $u \in \VnuOm$. From \cite{FGKV20} we know that there exists $u_n\in C^\infty_c(\R^d)$ such that, $\|u_n -u\|_{\VnuOm}\to 0.$ Put, $v_n= u_n|_{\Omega^c}$ by \eqref{eq:trace-inequality} we get
\begin{align*}
\|v_n -v\|_{\TnuOm}\leq \|u_n -u\|_{\VnuOm}\to 0.
\end{align*}
\end{proof}

\begin{remark}\label{rem:vondra-strange} Let us comment on certain function spaces that are introduced in \cite{Von21} in order to study some reflected jump Markov processes. Instead of the natural energy space $\VnuOm$ from \cite{FKV15}, the author considers $\VnuOm \cap L^2(\R^d, m)$ with $m(x) =\mathds{1}_{\Omega}(x)+ \mu(x)\mathds{1}_{\Omega^c}(x)$ and $\mu(x) = \int_{\Omega}\nu (x-y) \d y$ for $x \in \Omega^c$. It is proved in \cite[Lemma 2.2 (iii)]{Von21} that $L^2(\Omega^c, \mu)$ is the trace space of $\VnuOm \cap L^2(\R^d, m)$. Note that $\VnuOm \cap L^2(\R^d, m)$ and its trace space $L^2(\Omega^c, \mu)$ are much smaller than $\VnuOm$ resp. $\TnuOm$, which leads to the following issues. 
\begin{itemize}
\item When $\Omega$ is bounded, constant functions belong to 
$\VnuOm$ and do not belong to $\VnuOm \cap L^2(\R^d, m)$ in general. Thus in term of trace,  $x \mapsto \mathbbm{1}_{\Omega^c}(x)$ belongs to $\TnuOm$ but not necessarily to  $L^2(\Omega^c, \mu)$. Several natural Dirichlet problems, e.g. for the fractional Laplace operator, cannot be formulated with the help of $L^2(\Omega^c, \mu)$.
\item See \autoref{thm:Dform-reflected} for a regular Dirichlet form leading to the existence of reflected jump processes.
\item Given a function $v\in L^2(\Omega^c, \mu)$, its extension by zero $v_0= v\mathds{1}_{\Omega^c} $ belongs to $\VnuOm \cap L^2(\R^d, m)$ because of
\begin{align*}
\iil_{(\Omega^c\times \Omega^c)^c} \!\!\big(v_0(x)-v_0(y) \big)^2\, \nuxminy \d x \, \d y= 2\int_{\Omega^c} |v(x)|^2\mu(x)\d x.
\end{align*}
Moreover, the space $L^2(\Omega^c, \mu)$ is continuously embedded in $\TnuOm$, indeed, 
\begin{align*}
\|v\|_{\TnuOm}\leq \|v_0\|_{\VnuOm}\leq \sqrt{2}\|v\|_{L^2(\Omega^c, \mu)}.
\end{align*}
The fact that the extension by zero belongs to the energy space for any given function in the trace space, is rather particular. 
\end{itemize}
Since $\VnuOm \cap L^2(\R^d, m)$ resp. its trace space $L^2(\Omega^c, \mu)$ are small compared to the spaces $\VnuOm$ resp. $\TnuOm$, the range of possible nonlocal Dirichlet and Neumann problems is rather small.
\end{remark}

\section{Compact embeddings and {P}oincar\'e inequality}\label{sec:compactness}
%%%%%%%%%%%%%%%%%%%%%%%%%%%%%%%%%%%%%%%%%%%%%%%%%%%%%%%%%%%%%%%%%%%%%%%%%%%%%%%
%%%%%%%%%%%%%%%%%%%%%%%%%%%%%%%%%%%%%%%%%%%%%%%%%%%%%%%%%%%%%%%%%%%%%%%%%%%%%%% 

In this section we prove compact embeddings of the spaces $\HnuOm$, $\VnuOm$ and $V_{\nu,0}(\Omega|\R^d)$ into $L^2(\Omega)$. Our result on global compactness, \autoref{thm:embd-compactness}, requires some regularity assumptions on $\Omega$ and $\nu$, which we introduce and discuss in \autoref{subsec:assum-levy}. In \autoref{subsec:compact-poincare} we establish global compactness using ideas from \cite{JW19} and \cite{DMT18}. Note that \cite[Theorem 2.2]{CP18} is a related result. However, the proof therein seems to be valid only for domains that can be decomposed as a finite union of cubes. We circumvent this issue by an approximation argument near the boundary of $\Omega$.

%%%%%%%%%%%%%%%%%%%%%%%%%%%%%%%%%%%%%%%%%%%%%%%%%%%%%%%%%%%%%%%%%%%%%%%%%%%%%%%
\subsection{Assumptions on the L\'{e}vy measure}\label{subsec:assum-levy}
%%%%%%%%%%%%%%%%%%%%%%%%%%%%%%%%%%%%%%%%%%%%%%%%%%%%%%%%%%%%%%%%%%%%%%%%%%%%%%%

The definitions and most of the results of \autoref{subsec:function-spaces} do not require assumptions on the L\'{e}vy measure $\nu$ beyond the classical L\'{e}vy condition \eqref{eq:levy-cond}. In this subsection we collect further conditions on $\nu$ required for the compactness results in \autoref{subsec:compact-poincare}. Recall the concept of unimodality from \autoref{def:unimodality}. We will prove a Poincar\'{e}-inequality in \autoref{thm:poincare-inequality} for unimodal L\'{e}vy measures with full support.

\begin{definition}\label{def:classA0} Assume $\Omega\subset \R^d$ is open and bounded, and $\nu: \R^d\setminus\{0\}\to [0, \infty)$ satisfies \eqref{eq:levy-cond}. We say that $(\nu, \Omega)$ is in the class $\classA{0}$ if 
\begin{itemize}
\item [($\classA{0}$)] $\nu$ is unimodal and has full support.
\end{itemize}	
\end{definition}	

Note that, in the class $\classA{0},$ $\nu$ is not necessarily singular near $0$. In order to establish compactness results in \autoref{subsec:compact-poincare} we will discuss different assumptions. Note that \eqref{eq:levy-cond} and unimodality do not  imply any lower bound on $\nu$, even $\nu = 0$ would be allowed. Furthermore, if $\nu\in L^1(\R^d)$, then the spaces $\VnuOm \cap L^2(\R^d)$ and $H_{\nu}(\R^d)$ coincide with $L^2(\R^d)$, which is not locally compactly embedded into $L^2(\Omega)$. For the remainder of this section, we assume that $\nu$ satisfies \eqref{eq:levy-cond} and 
\begin{align}\tag{I}\label{eq:non-integrability-condition}
\int_{\R^d}\nu(h)\d h= \infty.
\end{align}
Under \eqref{eq:levy-cond}, condition \eqref{eq:non-integrability-condition} obviously follows if $|h|^d \nu(h) \to \infty$ as $|h|\to 0$. As explained in Corollary \ref{cor:local-compatcness}, \eqref{eq:levy-cond} and \eqref{eq:non-integrability-condition} imply local compactness of $\HnuOm$ and $\VnuOm$ in $L^2(\Omega)$. Let us introduce conditions on $\Omega$ and $\nu$ under which we are able to establish global compactness results.

\begin{definition}\label{def:classesAi} Assume $\Omega\subset \R^d$ is open and bounded, and $\nu: \R^d\setminus\{0\}\to [0, \infty)$ satisfies \eqref{eq:levy-cond} and \eqref{eq:non-integrability-condition}. We say that $(\nu, \Omega)$ is in the class $\mathscr{A}_i$ $(i=1,2,3)$, if 
\begin{itemize}
\item [($\classA{1}$)] $\ldots$ there exists an $\HnuOm$-extension operator $E: \HnuOm\to H_{\nu}(\R^d)$, i.e., there is $C(\nu, \Omega,d)>0$ such that for every $u\in \HnuOm$, $\|u\|_{H_\nu(\R^d)}\leq C\|u\|_{\HnuOm}$ and $Eu|_\Omega =u$. 
\item[($\classA{2}$)] $\ldots$  $\partial \Omega$ is Lipschitz-continuous, $\nu$ is radial and $q(\delta)\xrightarrow[]{\delta \to 0}\infty $ where
\begin{align}\label{eq:class-lipschitz}
q(\delta):= \frac{1}{\delta^2}\int\limits_{B_\delta(0)} |h|^2\nu(h)\d h\,.
\end{align}
\item [($\classA{3}$)] $\ldots$ the following condition holds true: $\widetilde{q}(\delta)\xrightarrow[]{\delta \to 0}\infty $ where
\begin{align}\label{eq:class-sing-boundary}
\widetilde{q}(\delta): = \inf_{a\in \partial\Omega}\int_{\Omega_\delta}\nu(h-a)\d h
\end{align}
with $\Omega_\delta= \{x\in \Omega: \operatorname{dist}(x,\partial\Omega)>\delta\}$.
\end{itemize}
\end{definition}

Note that monotonicity of $\nu$ is not necessarily required by any of the conditions above. The class $\classA{1}$ is well studied in the literature for the case of the fractional Laplace operator. For example, it is shown in \cite{Zho15} that $\Omega$ is an extension domain for $H^{\alpha/2}(\Omega),~\alpha\in (0,2)$, if and only if $\Omega$ is a $d$-set and thus, $ (|\cdot|^{-d-\alpha}, \Omega)$ is an element of $\classA{1}$.

\medskip
The class $\classA{2}$ is easy to understand because the conditions on $\Omega$ and $\nu$ do not interact. If $\Omega$ is a bounded Lipschitz domain and $\nu$ satisfies \eqref{eq:levy-cond} and 
\begin{align}\tag{$I'$}\label{eq:limit-at-0-explode}
\lim_{|h|\to 0}|h|^d \nu(h)= \infty \,, 
\end{align}
then $(\nu, \Omega)$ is in the class $\classA{2}$. Indeed, for $R>0$ sufficiently large there is $\delta_0>0$ such that $|h|^d\nu(h)\geq 2R$ whenever $|h|\leq \delta_0$. Thus 
$q(\delta)\geq|\mathbb{S}^{d-1}| R$ if $0<\delta<\delta_0$.
This shows that \eqref{eq:class-lipschitz} is verified.

\medskip

The class $\classA{3}$ and condition \eqref{eq:class-sing-boundary} are more involved due to a certain correlation between $\Omega$ and the singularity of $\nu$ near the origin. Let us first provide an example of $\nu$ and $\Omega$ such that $\classA{3}$ fails. In the Euclidean plane consider $\nu(h)= |h|^{-2-\alpha}\mathds{1}_V(h)$ with $V=\{(x_1,x_2)\in\mathbb{R}^{2}:~ |x_1|<|x_2|\}$ and $\Omega= \{(x_1,x_2)\in\mathbb{R}^{2}:~ 4|x_2-6|<x_1,~0<x_1<4\}$ whose boundary is continuous. Considering $a=(0,6)\in\partial \Omega$ one has $V\cap (\Omega_\delta-a)= \emptyset$ for every $\delta>0$, see Figure \ref{fig:cone-contra-example}. Therefore $\widetilde{q}(\delta)\leq \int_{\Omega_\delta}\nu(h-a)\d h=0$ and the condition \eqref{eq:class-sing-boundary} fails.

%\begin{figure}
%	\centering
%	\includegraphics[scale=0.35]{pic-counter-example-A3} 
%	\caption{ }\label{fig:cone-contra-example}
%\end{figure}

\definecolor{zzttqq}{rgb}{0.6,0.2,0.}
\definecolor{ududff}{rgb}{0.30196078431372547,0.30196078431372547,1.}
\definecolor{qqqqff}{rgb}{0.,0.,1.}

\begin{figure}[ht!]
\centering
\resizebox{40ex}{40ex}{
\begin{tikzpicture}[line cap=round,line join=round,>=triangle 45,x=1.0cm,y=1.0cm]
%\draw[step=1.0, gray,very thin,xshift=1mm,yshift=1mm] (0.5,0.5) grid (5.5,4.5)
\clip(-8.21682458369895,-5.99754616319308) rectangle (10.973555333341592,7.1980649192076);
\fill[line width=2.pt,color=zzttqq,fill=zzttqq,fill opacity=0.10000000149011612] (0.,6.) -- (4.,5.) -- (4.,7.) -- cycle;
\fill[line width=2.pt,color=zzttqq,fill=zzttqq,fill opacity=0.10000000149011612] (0.,0.) -- (4.262144356803506,1.0450488105997016) -- (4.262144356803506,-0.9549511894002986) -- cycle;
\fill[line width=2.pt,color=zzttqq,fill=zzttqq,fill opacity=0.10000000149011612] (0.5,0.) -- (3.995173232772751,0.7932112702301504) -- (3.9951732327727525,-0.7067887297698496) -- cycle;
\draw[line width=2.pt,dash pattern=on 4pt off 4pt,color=qqqqff,fill=qqqqff,fill opacity=0.25](-8.272328513062774,8.11980649192076)--(-9.71682458369895,8.11980649192076)--(-8.133233107220217,8.11980649192076)--(-8.13323310722022,8.11980649192076)--(-8.272328513062774,8.11980649192076);
\draw[line width=2.pt,dash pattern=on 4pt off 4pt,color=qqqqff,fill=qqqqff,fill opacity=0.25](8.272328627307077,8.11980649192076)--(-8.133233107220217,8.11980649192076)--(-8.13323310722022,8.11980649192076)--(-1.4479309161685987E-7,-1.819987695346988E-7)--(8.272328627307077,8.11980649192076);
\draw[line width=2.pt,dash pattern=on 4pt off 4pt,color=qqqqff,fill=qqqqff,fill opacity=0.25](-9.71682458369895,-6.186754616319308)--(-6.339276648144188,-6.186754616319308)--(-6.219972376982327,-6.186754616319308)--(-6.219972376982321,-6.186754616319308)--(-9.71682458369895,-6.186754616319308);
\draw[line width=2.pt,dash pattern=on 4pt off 4pt,color=qqqqff,fill=qqqqff,fill opacity=0.25](-1.4479309161685987E-7,-1.819987695346988E-7)--(-6.219972376982327,-6.186754616319308)--(-6.219972376982321,-6.186754616319308)--(6.339276049093379,-6.186754616319308)--(-1.4479309161685987E-7,-1.819987695346988E-7);
\draw [line width=2.pt,color=zzttqq] (0.,6.)-- (4.,5.);
\draw [line width=2.pt,color=zzttqq] (4.,5.)-- (4.,7.);
\draw [line width=2.pt,color=zzttqq] (4.,7.)-- (0.,6.);
\draw [line width=2.pt,color=zzttqq] (0.,0.)-- (4.262144356803506,1.0450488105997016);
\draw [line width=2.pt,color=zzttqq] (4.262144356803506,1.0450488105997016)-- (4.262144356803506,-0.9549511894002986);
\draw [line width=2.pt,color=zzttqq] (4.262144356803506,-0.9549511894002986)-- (0.,0.);
\draw [line width=2.pt,dash pattern=on 7pt off 7pt,color=zzttqq] (0.5,0.)-- (3.995173232772751,0.7932112702301504);
\draw [line width=2.pt,dash pattern=on 7pt off 7pt,color=zzttqq] (3.995173232772751,0.7932112702301504)-- (3.9951732327727525,-0.7067887297698496);
\draw [line width=2.pt,dash pattern=on 7pt off 7pt,color=zzttqq] (3.9951732327727525,-0.7067887297698496)-- (0.5,0.);
\draw (2.3624731976874207,6.4420647841314125) node[anchor=north west] {\huge $\Omega$};
\draw (5.354249737968875,-2.3737052804344287) node[anchor=north west] {\huge $\Omega_\delta-a$};
\draw (-3.2471078153403057,6.564082362879729) node[anchor=north west] {\huge $V$};
\draw (1.0161737545607663,-4.661534881965356) node[anchor=north west] {\huge $V$};
\draw (-2.4520418911714332,0.5242122148380798) node[anchor=north west] {\huge $(0,0)$};
\draw [->,line width=2.pt] (2.7738424719761205,0.15815947859313134) -- (5.3916469447223925,-2.2821920963731914);
\begin{scriptsize}
\draw [fill=ududff] (0.,6.) circle (2.5pt);
\draw[color=ududff] (0.2682296194904028,6.579334560223269) node {$a$};
\draw[color=black] (4.307127948870365,-0.7112157699886212) node {$v$};
\end{scriptsize}
\end{tikzpicture}
} %end of scalebox
\caption{Example of $(\nu, \Omega) \notin \classA{3} $}\label{fig:cone-contra-example}
\end{figure}

Next, let us provide a positive result. Note that for every domain $\Omega$ and every $\delta > 0$ we know $\widetilde{q}(\delta)<\infty$ because, for each $a\in \partial\Omega$ and each $\delta>0$, $ \Omega_\delta\subset B^c_\delta(a)$, which implies by \eqref{eq:levy-cond} $\widetilde{q}(\delta)\leq \int_{ B^c_\delta(0) }\nu(h)\d h<\infty$. We will show that $\partial \Omega \in C^{1,1}$ is sufficient. Recall that $\Omega$ is of class $C^{1,1} $ if for every $a\in \partial \Omega$ there is $r>0$ for which $B_r(a)\cap \partial \Omega= \{x=(x',x_d)\in B_r(a)~: x_d=\gamma(x')\}$ represents the graph of a $C^{1,1}$ function $\gamma: \mathbb{R}^{d-1}\to \mathbb{R} $. That is to say $\gamma$ is a $C^1$ function whose gradient is Lipschitz. 
The main result in \cite{BaS09} shows that an open set $\Omega$ is $C^{1,1} $ if and only if $\Omega$ satisfies the interior and exterior sphere condition. We say that $\Omega$ satisfies the interior and exterior sphere condition at some scale $r>0$ if for every $ a\in \partial \Omega$ one can find $a'\in \Omega$ and $a''\in \overline{\Omega}^c$ for which $B_r(a')\subset \Omega $, $B_r(a'')\subset \overline{\Omega}^c$ and $ \overline{B_r(a')}\cap \overline{B_r(a'')}= \{a\}$. 
Note that, the interior and exterior sphere condition holds for every scale $r\in (0,r_0)$ once it holds for $r_0$. This characterization entails that a $C^{1,1}$ set $\Omega$ is a $d$-set (or \emph{volume density condition} according to some authors): that is, there exist two positive constants $ c>0$ and $ r_0>0$ such that for every $r\in (0,r_0)$ and every $a \in \partial \Omega$
\begin{align*}
|\Omega \cap B_r(a)|\geq cr^d.
\end{align*}

\begin{proposition}
Assume $\nu$ satisfies \eqref{eq:levy-cond} and \eqref{eq:limit-at-0-explode}. Assume that $\Omega$ satisfies the following strong volume density condition: there exist positive constants $\tau>1$, $\delta_0>0$ and $c>0$ such that for all $\delta\in (0, \delta_0)$ and 
$a\in \partial\Omega$ 
\begin{align*}
|\Omega_\delta \cap B_{\tau \delta}(a)|\geq c\delta^d.
\end{align*}
Then $(\nu, \Omega) \in \classA{3}$.
\end{proposition}

\begin{remark}{\ }
\begin{enumerate}[(i)]
\item Any bounded $C^{1,1}$-domain $\Omega \subset \R^d$ satisfies the aforementioned strong volume density condition.  
Fix $a\in \partial\Omega$, by the interior sphere condition, consider $\delta\in (0,\delta_0/4)$ for some $\delta_0$ sufficiently small. Let $x\in \Omega$  depending on $a$ and $\delta$ such that $B_{2\delta}(x) \subset \Omega$, $\operatorname{dist}(x, \partial\Omega)= |x-a|=2\delta$ and $\overline{B_{2\delta}(x)} \cap \partial\Omega= \{a\}$ then obviously, $B_\delta(x) \subset \Omega_\delta \cap B_{2\delta}(x) \subset \Omega_\delta \cap B_{4\delta}(a) $. This yields
\begin{align}\label{eq:strong-density}
|\Omega_\delta \cap B_{4\delta}(a)|\geq c_d\delta^d, \quad\text{with}~~c_d=|B_1(0)|.
\end{align}
\item It is interesting to know whether for small $\delta>0$, $\Omega_\delta$ inherits the regularity of $\Omega$. As proven in \cite[Section 6.14]{GT15} if $\Omega$ is of class $C^k$ with $k\geq 2$ then so is $\Omega_\delta$. 
\end{enumerate}
\end{remark}

\begin{proof}
Let  $R>0$ and consider $\tau>1$,  $\delta_0>0$ as above such that  if $0<\delta <\delta_0$ then $ |h|^d\nu(h)\geq R$ for $|h|<\delta$. Fix $a\in \partial\Omega$, since $|\Omega_\delta \cap B_{\tau\delta}(a)|\geq c\delta^d$  for all $0<\delta <\delta_0$. Therefore, recalling that $\nu(h-a)\geq R|h-a|^{-d}\geq \frac{R}{\tau^d\delta^d}$ when $h\in B_{\tau\delta}(a)$ we have 

\begin{align*}
\int_{\Omega_\delta}\nu(h-a)\d h \geq \frac{R}{\tau^d\delta^d} \int_{\Omega_\delta\cap B_{\tau \delta}(a) } \hspace{-3ex}\d h= \, \frac{R}{\tau^d\delta^d}|\Omega_\delta \cap B_{\tau\delta}(a)|\geq \frac{ c}{\tau^d}R.
\end{align*}
Finally, 
\begin{align*}
\widetilde{q}(\delta) \geq \frac{ c}{\tau^d}R
\end{align*}
which means that \eqref{eq:class-sing-boundary} is verified since $R$ can be arbitrarily large.
\end{proof}

%%%%%%%%%%%%%%%%%
\subsection{Local and global compactness results}\label{subsec:compact-poincare}
%%%%%%%%%%%%%%%%%
%We intend to provide alternative approachs to the compactness result in \cite[Theorem 2.2]{CP18}. 
%The technique therein is adapted from the one in \cite{Hitchhiker} for fractional Sobolev spaces which uses the Sobolev extensibility property of the corresponding domain. However, the proof provided by the authors seems to be valid only for domains which can be written as finite union of cubes; unless the corresponding domain possesses a Sobolev extension property for the space of functions $u\in L^2(\Omega)$ whose the semi-norm \linebreak $\iint\limits_{\Omega\Omega}(u(x)-u(y))^2J(x,y)\,\d x\,\d y$ is finite. We will use approximation argument near the boundary of $\Omega$. 

%\medskip

%The following result can be found in \cite[Corollary 4.28]{Bre10}. 
%\cite[Theorem 6.23]{ACS14},

Before citing a result on local compactness, we recall a well-known result about convolutions.  

\begin{lemma}[Corollary 4.28 of \cite{Bre10}]\label{lem:compactness-convolution} 
Let $w \in L^1(\R^d)$. Then the convolution operator
\begin{align*}
T_w : L^2(\R^d )\to L^2(\R^d), \qquad T_wu=w*u 
\end{align*}
is continuous with $\|T_w\|_{\mathcal{L}(L^2(\R^d), L^2(\R^d))} \leq \|w\|_{L^1(\R^d)}$. Moreover, $T_w : L^2(\R^d) \to L^2(K)$ is compact for any compact subset $K \subset \R^d$. 
\end{lemma}

We present the local compactness result \cite{JW19} that we are going to use in the sequel.

\begin{theorem}\label{thm:local-compactness}
Let $\nu:\R^d\setminus\{0\}\to [0, \infty)$ be a measurable symmetric function such that \eqref{eq:non-integrability-condition} holds and   
\begin{align}\label{eq:integrability-condition-near-zero}
\int_{\R^d \setminus B_\delta(0)} \nu (h) \d h<\infty~\text{ for every $\delta>0$} .
\end{align}
Then the embedding $H_\nu(\R^d) \hookrightarrow L^2(\R^d)$ is locally compact. As a consequence, for $\Omega\subset \R^d$ open and bounded, the embedding $V_{\nu,0}(\Omega|\R^d) \hookrightarrow L^2(\Omega)$ is compact. 
\end{theorem}

It is worth mentioning that an earlier analogous result is provided in \cite[Proposition 6]{PZ17} and 
\cite[Proposition 1]{BJ16} for periodic functions on the  torus. This technique, which consist of killing the singularity, is also used in \cite[Lemma 3.1]{BJ13}. The assertion of \autoref{thm:local-compactness} is proved in \cite[Theorem 1.1]{JW19} under the additional assumption that $\nu$ satisfies \eqref{eq:levy-cond}.
An analogous result is also proved in \cite[Theorem 2.1]{CP18} under restrictive assumptions on the kernel, using the Pego criterion for compact compactness in $L^2(\R^d)$. Looking at the proof carefully one sees that conditions \eqref{eq:non-integrability-condition} and \eqref{eq:integrability-condition-near-zero} is sufficient. This would allow to consider densities $\nu$ with a very strong singularity at the origin, e.g., $\nu(h)= |h|^{-d-\beta}$ for $h \neq0$ with any $\beta>0$. 
\medskip

As a straightforward consequence of \autoref{thm:local-compactness} we have the following local compactness of $\HnuOm$ in $L^2(\Omega)$. 

\begin{corollary}\label{cor:local-compatcness}
Let $\Omega\subset \R^d$ be open but not necessarily bounded. Assume $\nu:\R^d\setminus\{0\} \to \mathbb{R}$ fulfills conditions \eqref{eq:levy-cond} and \eqref{eq:non-integrability-condition}. The embedding $\HnuOm\hookrightarrow L^2(\Omega) $ is locally compact. Furthermore, for every bounded sequence $(u_n)_n$ there exists $u\in \HnuOm$ and subsequence $(u_{n_j})_j$ converging to $u$ in $L^2_{\operatorname{loc}}(\Omega)$. 
\end{corollary}

\begin{proof}
There is no loss of generality if we assume that  a function $u\in\HnuOm$ is extended by zero outside of $\Omega$. For $\varphi \in C_c^\infty(\R^d)$,  with $\supp \varphi\subset \Omega$,  the map $J_\varphi: \HnuOm\to H_{\nu}(\R^d)$, with $J_\varphi u = u\varphi $ is continuous and is thus locally compact by \autoref{thm:local-compactness}. Therefore the embedding $\HnuOm\hookrightarrow L^2(\Omega)$ is locally compact. Indeed, for $u \in \HnuOm$ we have 
\begin{align*}
\big[u(x)\varphi(x)- u(y)\varphi(y)\big]^2& =\big[u(x)(\varphi(x)- \varphi(y))+\varphi(y)(u(x)- u(y))\big]^2\\
&\leq 2\|\varphi\|^2_{W^{1,\infty}(\R^d)}\big[  |u(x)|^2(1\land |x-y|^2)+ \mathds{1}_{\operatorname{supp}\varphi}(y) (u(x)- u(y))^2 \big].
\end{align*}
As $\operatorname{supp}\varphi\subset \Omega$ is compact, consider $0<r\leq \dist(\supp\varphi, \partial\Omega)$. Then  integrating both sides of the above estimate over $\Omega\times \R^d$ with respect to the measure $\nuxminy\d y\d x$, yields the continuity of $J_\varphi$ as follows

%$|u\varphi|^2_{H_{\nu}(\R^d)}\leq C_\varphi \|u\|^2_{\HnuOm}$. 
\begin{align*}
\iint\limits_{\R^d\R^d} \big[u(x)\varphi(x)&- u(y)\varphi(y)\big]^2\nuxminy\d y\d x \\
& \leq 2\|\varphi\|^2_{W^{1,\infty}(\R^d)} \int_{\Omega}  |u(x)|^2\d x\int_{\R^d}(1\land |h|^2)\nu(h)\d h\\
&+ \iint\limits_{\Omega\Omega}(u(x)-u(y))^2\nuxminy\d y\d x+\int_\Omega |u(x)|^2\d x\int_{B_r(0)}\nu(h)\d h\\
&\leq C_\varphi \|u\|^2_{\HnuOm}\,.
\end{align*}
%
%
%Here $0<r\leq \dist(\supp\varphi, \partial\Omega)$. 
\smallskip 

Next, we  prove the second statement.  Consider $\Omega'_\delta=\big\{x\in \Omega: |x| <\frac{1}{\delta},\,\, \dist(x,\partial\Omega)>\delta\big\}=\Omega_\delta \cap B_{\frac{1}{\delta}}(0)$ and define $\varphi_\delta(x) =\eta_{\delta/4}*\mathds{1}_{\Omega'_{\delta/2}}(x) $  for $\delta>0$ small enough, where $ \eta_{\delta}(x)=\frac{1}{ \delta^d}\eta(\frac{x}{\delta})$ with $\eta \in C_c^\infty(\R^d)$ is supported in the unit ball $B_1(0)$, $\eta\geq 0$ and $\int_{\R^d} \phi(x)\d x =1$.  So that $\varphi_\delta \in C_c^\infty(\Omega)$, $\varphi_\delta =1$ on $\Omega'_\delta$,  $0\leq \varphi_\delta\leq 1$  and $|\nabla \varphi_{\delta} |\leq c/\delta$. Given a sequence $(u_n)_n$ that is bounded in $\HnuOm$, the previous observation entails that for each $\delta>0$, for the sequence $(u_n\varphi_\delta)_n$ there exists  a subsequence $n_j=n_j(\delta), $ $j\geq 1$, and $u^\delta\in L^2(\Omega'_\delta)$ such that  the sequence $(u_n)_n$, 
(as $u_n\varphi_\delta=u_n$ in $\Omega'_\delta$) converges to some $u^\delta$ in $L^2(\Omega'_\delta)$ and almost everywhere in $\Omega'_\delta.$ Employing the standard  the Cantor's diagonalization procedure with $\delta=\frac{1}{2^k}$, one can construct a subsequence $(u_{n_j})_j$  converging subsequence in $L_{\operatorname{loc}}^2(\Omega)$ and almost everywhere in $\Omega$ to some function $u$. Fatou's lemma implies that $u\in \HnuOm$ since 
\begin{align*}
\|u\|_{\HnuOm} \leq \liminf_{j\to \infty}\|u_{n_j}\|_{\HnuOm}<\infty.
\end{align*}

\end{proof}

Let us turn to the result on global compactness. We will need some estimates near the boundary $\partial\Omega$. We begin with the following estimate involving cut-off functions.

\begin{lemma}\label{lem:estimate-cut-off}
Let $\Omega \subset \R^d$ be open and bounded. Assume $\nu:\R^d\setminus\{0\} \to [0, \infty)$ is even and measurable. Let $0 < \delta < \frac13 \diam (\Omega)$. Let $\varphi \in C^\infty_c(\Omega)$ be such that \footnote{Take $\varphi = 1-\varphi_\delta$ with $\varphi_\delta= \eta_{\delta/4}*\mathds{1}_{\Omega'_{\delta/2}}$ as in the proof of Corollary \ref{cor:local-compatcness}.} $ \varphi=0$ on $\Omega_\delta$, $ \varphi=1$ on $\Omega\setminus \Omega_{\delta/2}$, $0\leq \varphi\leq 1$ and $|\nabla \varphi|\leq c/\delta$. For every $u\in \HnuOm$, the following estimate holds true
%
%\begin{relsize}{-1}
\begin{align}\label{eq:estimate-cut-off}
\iil_{\Omega\Omega}([u\varphi](x)-[u\varphi](y))^2\nuxminy\d x\d y
\leq \frac{C}{\delta^2}\int\limits_{\Omega_{\delta/2}} |u(x)|^2\d x+ 8\iint\limits_{\Omega\Omega}(u(x)-u(y))^2\nuxminy\d x\d y
\end{align}
%\end{relsize}
where, $C=8c^2 \int_{B_R(0)}|h|^2\nu(h)\d h$ and $R=\operatorname{diam}(\Omega)$.

\end{lemma}
\begin{proof}
Firstly, since $\varphi=1$ on $\Omega\setminus \Omega_{\delta/2}$ we have 

\begin{align*}
\iint\limits_{\Omega\setminus \Omega_{\delta/2}~ \Omega\setminus \Omega_{\delta/2}}\hspace*{-4ex} ([u\varphi](x)-[u\varphi](y))^2\nuxminy\d x\d y 
&=\hspace*{-3ex} \iint\limits_{\Omega\setminus \Omega_{\delta/2}~ \Omega\setminus \Omega_{\delta/2}} \hspace*{-4ex}(u(x)-u(y))^2\nuxminy\d x\d y\\
&\leq \iint\limits_{\Omega\Omega}(u(x)-u(y))^2\nuxminy\d x\d y\,.
\end{align*}

In view of the fact that $0\leq \varphi\leq 1$ and $|\varphi(x)-\varphi(y)|\leq c/\delta|x-y|$ for every $x,y \in \Omega$, we have
\begin{align}
([u\varphi](x)-[u\varphi](y))^2 &= \big( \varphi(y)(u(x)-u(y)) + u(x)(\varphi(x)-\varphi(y)) \big)^2 \notag \\
&\leq 2(u(x)-u(y))^2+ \frac{2c^2}{\delta^2}  |u(x)|^2|x-y|^2\,.\label{eq:split}
\end{align} 
Secondly, noticing that $\Omega\subset B_R(x)$ for all $ x\in \Omega$ where $R= \operatorname{diam}(\Omega)$ and integrating both sides of \eqref{eq:split} over $\Omega_{\delta/2}\times \Omega_{\delta/2}$ we obtain the following estimate

\begin{align*}
&\iint\limits_{\Omega_{\delta/2} \Omega_{\delta/2}}([u\varphi](x)-[u\varphi](y))^2\nuxminy\d x\d y\\
&\leq 2\iint\limits_{\Omega\Omega}(u(x)-u(y))^2\nuxminy\d x\d y+ \frac{2c^2}{\delta^2} \int\limits_{ \Omega_{\delta/2}} |u(x)|^2\d x \int\limits_{B_R(x)} |x-y|^2\nuxminy\d y \\
&= 2\iint\limits_{\Omega\Omega}(u(x)-u(y))^2\nuxminy\d x\d y+ \frac{2c^2}{\delta^2} \Big(\int\limits_{B_R(0)} |h|^2\nu(h)\d h\Big) \int\limits_{ \Omega_{\delta/2}} |u(x)|^2\d x \,.
\end{align*}
Likewise to the previous estimate, using \eqref{eq:split} we get 
\begin{align*}
%&\Big(\iint\limits_{\Omega_{\delta/2}\times\Omega\setminus \Omega_{\delta/2}} + \iint\limits_{\Omega\setminus\Omega_{\delta/2}\times \Omega_{\delta/2}} \Big) ([u\varphi](x)-[u\varphi](y))^2\nuxminy\d x\d y\\&=
& \iint\limits_{\Omega_{\delta/2}\times \Omega\setminus \Omega_{\delta/2}} ([u\varphi](x)-[u\varphi](y))^2\nuxminy\d x\d y\\
%
%&\leq 2\iint\limits_{\Omega\Omega}(u(x)-u(y))^2\nuxminy\d x\d y+ \frac{2c^2}{\delta^2} \int\limits_{ \Omega_{\delta/2}} |u(x)|^2\d x \int\limits_{B_R(x)} |x-y|^2\nuxminy\d y\\
%
&\leq 2\iint\limits_{\Omega\Omega}(u(x)-u(y))^2\nuxminy\d x\d y+ \frac{2c^2}{\delta^2} \Big(\int\limits_{B_R(0)} |h|^2\nu(h)\d h\Big) \int\limits_{ \Omega_{\delta/2}} |u(x)|^2\d x \,.
\end{align*}
Altogether, the desired estimate follows as claimed since by symmetry we can use the split
% $\Omega=\Omega_{\delta/2}\cup \Omega\setminus\Omega_{\delta/2}$.

\begin{align*}
\iint\limits_{\Omega\times \Omega}= \iint\limits_{\Omega_{\delta/2}\times \Omega_{\delta/2}} +\, 2
\hspace*{-2ex} \iint\limits_{\Omega_{\delta/2}\times \Omega\setminus \Omega_{\delta/2}}+ \iint\limits_{\Omega\setminus\Omega_{\delta/2}\times\Omega\setminus\Omega_{\delta/2}}\,.
\end{align*}
\end{proof}

The next lemma plays a crucial role in the sequel. 

\begin{lemma}\label{lem:estimate-near-boundary}
Assume that $ \Omega\subset \R^d$ is open and bounded, and $\nu:\R^d\setminus\{0\}\to [0, \infty)$ is radial. There exists $C>0$ such that for every $u \in L^2(\Omega)$ and every positive $\delta < \frac13 \diam (\Omega)$
\begin{align}\label{eq:estimate-boundary}
\int_{\Omega} |u(x)|^2\d x\leq \frac{C}{\delta^2 \widetilde{q}(2\delta)}\int_{\Omega_{\delta/2}}|u(x)|^2 \d x+ \frac{8}{\widetilde{q}(2\delta)}\iint\limits_{\Omega\Omega}(u(x)-u(y))^2\nuxminy\d x\d y \,,
\end{align}
with $\widetilde{q}$ as is \eqref{eq:class-sing-boundary}. Moreover, if $\Omega$ has a Lipschitz boundary, then 
\begin{align}\label{eq:estimate-boundary-lip}
\int_{\Omega} |u(x)|^2\d x\leq \frac{C}{\delta^2 q(2\delta)}\int_{\Omega_{\delta/2}}|u(x)|^2 \d x+ \frac{8}{q(2\delta)}\iint\limits_{\Omega\Omega}(u(x)-u(y))^2\nuxminy\d x\d y \,,
\end{align}
with $q$ as in \eqref{eq:class-lipschitz}. 
\end{lemma}

\begin{proof}
Let $\varphi$ be as in \autoref{lem:estimate-cut-off} and fix $a\in \partial\Omega$. A routine check reveals that $\Omega_{2\delta}-a\subset \Omega_{\delta}-x$ for every $x\in \Omega \cap B_\delta(a)$ which yields,
\begin{align*}
\int_{\Omega \cap B_\delta(a)} & [\varphi u]^2 (x) \d x \int_{\Omega_\delta} \nuxminy\d y \geq \int_{\Omega \cap B_\delta(a)} [\varphi u]^2 (x) \d x \int_{\Omega_\delta-x} \nu(h)\d h\\
&\geq \int_{\Omega \cap B_\delta(a)} [\varphi u]^2 (x) \d x \int_{\Omega_{2\delta}-a } \nu(h)\d h \geq \widetilde{q}(2\delta) \int_{\Omega \cap B_\delta(a)} [\varphi u]^2 (x) \d x.
\end{align*}
%
%where 
%\begin{align*}
%\widetilde{q}(2\delta) = \inf_{a\in \partial\Omega}\int_{\Omega_{2\delta}}\nu(h-a)\d h\,.
%\end{align*}

By a compactness argument there exist $a^1,a^2,\cdots a^n\in \partial\Omega$ such that $\partial\Omega\subset \bigcup\limits_{i=1}^n B_{\delta/2}(a^i)$. So that, $\Omega\setminus \Omega_{\delta/2}\subset \bigcup\limits_{i=1}^n \Omega\cap B_{\delta}(a^i) \subset \Omega\setminus \Omega_{\delta}$.
$\varphi u=0$ on $\Omega_{\delta/2}$ trivially implies $\varphi u=0$ on $\Omega_{\delta}$. Therefore with the aid of the above estimate we obtain the following estimate
\begin{align*}
&\iint\limits_{\Omega\Omega} \big( [\varphi u] (x)- [\varphi u] (y)\big)^2\nuxminy\d x\d y
\geq 2 \iint\limits_{\Omega\setminus \Omega_{\delta} ~ \Omega_\delta} [\varphi u]^2 (x) \nuxminy\d x\d y\\
&\geq \int\limits_{ \bigcup\limits_{i=1}^n \Omega\cap B_{\delta}(a^i)} [\varphi u]^2 (x) \d x \int_{ \Omega_\delta} \nuxminy\d y
\geq 2\widetilde{q}(2\delta)\int\limits_{ \bigcup\limits_{i=1}^n \Omega\cap B_{\delta}(a^i)} [\varphi u]^2 (x) \d x \\
&\geq 2\widetilde{q}(2\delta)\int\limits_{ \Omega\setminus \Omega_{\delta/2}} [\varphi u]^2 (x) \d x= 2\widetilde{q}(2\delta)\int\limits_{ \Omega\setminus \Omega_{\delta/2}} |u(x)|^2 (x) \d x \,.
%&=\widetilde{q}(2\delta) \|\varphi u\|^2_{L^2(\Omega)}\,.
\end{align*}
This combined with \eqref{eq:estimate-cut-off} gives \eqref{eq:estimate-boundary}. Next, let us assume that $\Omega$ is a Lipschitz domain. Following the same procedure as in \cite[Eq. (22) and Eq. (23)]{Ponce2004} one arrives at 

\begin{align*}
2\delta^2 \iint\limits_{\Omega\Omega} \big( [\varphi u] (x)- [\varphi u] (y)\big)^2\nuxminy\d x\d y
\geq \Big(\int\limits_{B_{2\delta}(0)} |h|^2\nu(h)\d h\Big) \int\limits_{ \Omega\setminus \Omega_{\delta/2}} [\varphi u]^2 (x) \d x.
\end{align*}
that is, 
\begin{align*}
\iint\limits_{\Omega\Omega} \big( [\varphi u] (x)- [\varphi u] (y)\big)^2\nuxminy\d x\d y
\geq 2q(2\delta) \int\limits_{ \Omega\setminus \Omega_{\delta/2}} [\varphi u]^2 (x) \d x, 
\end{align*}
which combined with \eqref{eq:estimate-cut-off} implies \eqref{eq:estimate-boundary-lip}.
\end{proof}

\medskip

Here is our global compactness result.
\begin{theorem}\label{thm:embd-compactness}
Let $\Omega$ be an open bounded subset of $\R^d$ and $\nu:\R^d\setminus\{0\} \to [0,\infty]$ be a measurable function. If the couple $(\nu, \Omega)$ belongs to the class $\mathscr{A}_i,~ i=1,2,3$ then the embedding $\HnuOm \hookrightarrow L^2(\Omega)$ is compact. In particular, the embedding $\VnuOm \hookrightarrow L^2(\Omega)$ is compact.
\end{theorem}

\begin{proof}%[Proof of \autoref{thm:embd-compactness}]
Given the continuous embedding $\VnuOm\hookrightarrow\HnuOm$, it will be sufficient only to prove that the embedding $\HnuOm\hookrightarrow L^2(\Omega)$ is compact. For $(\nu, \Omega)$ belonging to the class $\classA{1}$ the result is a direct consequence of \autoref{thm:local-compactness}. Now assume $(\nu, \Omega)$ belongs to the class $\classA{2}$ (resp. $\classA{3}$) then for $\varepsilon>0$ there is $\delta>0$ small enough such that $8q^{-1}(2\delta)<\varepsilon$ (resp. $8 \widetilde{q}^{-1}(2\delta)<\varepsilon$) If $(u_n)_n$ is a bounded sequence of $\HnuOm$ then Corollary \ref{cor:local-compatcness} infers the existence of a subsequence $(u_{n_j})_j$ of $(u_n)_n$ converging to some $u \in \HnuOm$ in $L^2(\Omega_{\delta/2})$ i.e $\|u_{n_j}-u\|_{L^2(\Omega_{\delta/2})}\to0$ as $j\to \infty$. In any case, in view of \autoref{lem:estimate-near-boundary}, passing to the limsup in \eqref{eq:estimate-boundary} or in \eqref{eq:estimate-boundary-lip} applied to $u_{n_j}-u$ we get 
\begin{align*}
\limsup_{j\to \infty}\int_{\Omega}|u_{n_j}(x)-u(x)|^2\d x\leq M\varepsilon
\end{align*}
where $$M= 2\|u\|^2_{\HnuOm}+ 2\sup_{n}\|u_n\|^2_{\HnuOm}<\infty.$$ Finally, $\limsup\limits_{j\to \infty}\|u_{n_j}-u\|_{L^2(\Omega)}=0$ since $\varepsilon>0$ is arbitrarily chosen, which achieves the proof. 
\end{proof}

\medskip 

\begin{remark}
A noteworthy consequence of what we have obtained so far is that, for an appropriate choice of $\nu,$ the well-known Rellich-Kondrachev compact embeddings $H^1_0(\Omega)\hookrightarrow L^2(\Omega) $ and $H^1(\Omega)\hookrightarrow L^2(\Omega)$ when $\Omega$ is Lipschitz, respectively derive from \autoref{thm:local-compactness} combined with the continuous embedding $H^1_0(\Omega)\hookrightarrow V_{\nu,0}(\Omega|\R^d)$ and from \autoref{thm:embd-compactness} combined with the continuous embedding $H^1(\Omega)\hookrightarrow \HnuOm$ when $\Omega$ is Lipschitz. 
\end{remark}

\medskip

The efforts made to establish \autoref{thm:embd-compactness} will be rewarded for the elaboration of the Poincar\'e type inequality which will be useful in the forthcoming section.

\begin{theorem}[Poincar\'e inequality]\label{thm:poincare-inequality}
Let $\Omega$ be an open bounded subset of $\R^d$ and $\nu:\R^d\setminus\{0\} \to [0,\infty]$ be a measurable function with full support. Assume the couple $(\nu, \Omega)$ belongs to one of the class $\mathscr{A}_i,~ i=0,1,2,3$. Then there is exists a positive constant $C= C(d,\Omega, \nu )$ depending only on $d,~ \Omega$ and $\nu$ such that

\begin{align}\label{eq:poincare-inequalityH}
\big\|u-\mbox{$\fint_{\Omega}$}u\big\|^2_{L^2(\Omega)} \leq C\iint\limits_{\Omega\Omega}(u(x)-u(y))^2\nuxminy\d x\d y\quad\text{for all }~ u \in L^2(\Omega)\,, 
\end{align}
and hence 
\begin{align}\label{eq:poincare-inequalityV}
\big\|u-\mbox{$\fint_{\Omega}$}u\big\|^2_{L^2(\Omega)} \leq C\mathcal{E}(u,u)\,\qquad\text{for all }~ u \in \VnuOm\, .
\end{align}

\end{theorem}

\begin{proof}
Assume such constant does not exist then we can find a sequence $(u_n)_n$ elements of $\HnuOm$ such that for every $n$, $\fint_{\Omega}u_n =0$, $\|u_n\|_{L^2(\Omega)}=1$ and 

\begin{align*}
\iint\limits_{\Omega\Omega}(u_n(x)-u_n(y))^2\nuxminy\d x\d y\leq \frac{1}{2^n}\,.
\end{align*}

The sequence $(u_n)_n$ is thus bounded in $\HnuOm$ which by \autoref{thm:embd-compactness} is compactly embedded in $L^2(\Omega)$ whenever $(\nu,\Omega)$ is in the class $\mathscr{A}_i,~ i=1,2,3$. Therefore, if it is the case, passing through a subsequence, $(u_n)_n$ converges in $L^2(\Omega)$ to some function $u$. Clearly it follows that $\fint_{\Omega}u =0$ and $\|u\|_{L^2(\Omega)}=1$. Moreover, by Fatou's Lemma we have 
\begin{align*}
\iint\limits_{\Omega\Omega}(u(x)-u(y))^2\nuxminy\d x\d y\leq \liminf_{n \to \infty} \iint\limits_{\Omega\Omega}(u_n(x)-u_n(y))^2\nuxminy\d x\d y=0
\end{align*}
which implies that $u$ equals the constant function $x\mapsto \fint_{\Omega}u =0$ almost everywhere on $\Omega $. This goes against the fact that $\|u\|_{L^2(\Omega)}=1$ hereby showing that our initial assumption was wrong. 

\smallskip

Next assume $(\nu,\Omega)$ belongs to the class $\classA{0}$ then, as $\nu$ has full a support, is unimodal and $\Omega$ is bounded, there is a constant $c>0$ such that $\nuxminy\geq c$ for all $x,y \in \Omega$. Using this and Jensen's inequality we obtain the desired inequality as follows:
\begin{align*}
\iint\limits_{\Omega\Omega}(u(x)-u(y))^2\nuxminy\d x\d y&\geq c |\Omega| \int_{\Omega} \fint_{\Omega}(u(x)-u(y))^2\d x\d y\\
&\geq c |\Omega| \|u-\mbox{$ \fint_{\Omega} u $} \|^2_{L^2(\Omega)}\,.
\end{align*}
The proof is complete because \eqref{eq:poincare-inequalityV} is a consequence of \eqref{eq:poincare-inequalityH}.
\end{proof}

\medskip

The above Poincar\'e inequality \eqref{eq:poincare-inequalityH}-\eqref{eq:poincare-inequalityV} can be seen as the nonlocal counterpart of the classical Poincar\'e inequality which states that, for a connected bounded Lipschitz domain $\Omega$, there is $C>0$ for which 
\begin{align*}
\big\|u-\mbox{$\fint_{\Omega}$}u\big\|_{L^2(\Omega)} \leq C\|\nabla u\|_{L^2(\Omega)}\,,\qquad\text{for all }~ u \in L^2(\Omega) 
\end{align*}
where by convention we assume $\|\nabla u\|_{L^2(\Omega)} =\infty$ if $|\nabla u|$ is not in $L^2(\Omega)$. Alongside to this we also recall the classical Poincar\'e-Friedrichs inequality: there is $C>0$ such that
\begin{align*}
\|u\|_{L^2(\Omega)} \leq C\|\nabla u\|_{L^2(\Omega)}\,\qquad\text{for all }~ u \in H_0^1(\Omega)\,. 
\end{align*}
In the same spirit, as we will see below the corresponding nonlocal Poincar\'e-Friedrichs inequality $V_{\nu,0}(\Omega|\R^d)$ (which we recall is the closure of the $C_c^\infty(\Omega)$ in $\VnuOm$) is much more easier to obtain and no compactness argument is required. This provides an easier alternative proof to the Poincar\'e-Friedrichs  inequality from \cite[Lemma 2.7]{FKV15}. Furthermore, under the condition that the embedding is $V_{\nu,0}(\Omega|\R^d) \hookrightarrow L^2(\Omega)$ is compact, a similar inequality is proved in \cite{JW19} wherein the authors only assume $\Omega$ to be bounded in one direction. 

\medskip  

\begin{theorem}[Poincar\'e-Friedrichs inequality]\label{thm:poincare-friedrichs-ext}
Let $\Omega\subset \R^d$ be open and bounded. Let $\nu:\R^d\setminus\{0\}\to [0, \infty)$ be a symmetric function such that one of the two conditions holds true:
\begin{enumerate}[(i)]
\item $\nu \mathds{1}_{\R^d \setminus B_{R}}$ is nontrivial and integrable for $R=\operatorname{diam}(\Omega)$ ,	
\item $\nu \in L^1(\R^d)$ and $|\{\nu > 0\}| > 0$.
\end{enumerate}
Then for some constant $C=C(d,\Omega, \nu)>0$ 
\begin{align}
\|u\|^2_{L^2(\Omega)}\leq C\mathcal{E}(u,u)\quad \text{for all $u\in V_{\nu,0}(\Omega|\R^d)$.}
\end{align}

\end{theorem}

\medskip

\begin{proof}
Set $R=\operatorname{diam}(\Omega)$. Then for all $x\in \Omega$ we have $B^c_R(x)\subset \Omega^c$. For $u\in V_{\nu,0}(\Omega|\R^d)$ we recall that $u=0~a.e$ on $\Omega^c$. Thus, 
\begin{align*}
\mathcal{E}(u,u)&=\frac{1}{2} \iint\limits_{\Omega\Omega}(u(x)-u(y))^2\nuxminy\d x \d y+ \int_{\Omega} |u(x)|^2\d x\int_{\Omega^c} \nuxminy\d y\\
&\geq 2\int_{\Omega} |u(x)|^2\d x\int_{B^c_R(x)} \nuxminy\d y = 2\|\nu_R\|_{L^1(\R^d)}\|u\|^2_{L^2(\Omega)}.
\end{align*}
Take C= $(2\|\nu_R\|_{L^1(\R^d)})^{-1}$ with $\nu_R=\nu \mathds{1}_{\R^d\setminus B_R(0)}$. This settles the first case. The second case is treated in \cite[Lemma 2.7]{FKV15}.
\end{proof}

\begin{remark}
Note that a Poincar\'e-Friedrichs inequality of the form  
\begin{align}\label{eq:poincare-friedrichs}
\|u\|^2_{L^2(\Omega)}\leq C\iint\limits_{\Omega\Omega}(u(x)-u(y))^2\nuxminy\d x\d y \qquad (u\in C_c^\infty(\Omega))\,.
\end{align}
does not hold in general, independently of whether the embedding $H_{\nu}(\Omega) \hookrightarrow L^2(\Omega)$ is compact or not. For example, consider $\nu(h)= |h|^{-d-\alpha}$ with $0<\alpha <1$. Then, $C_c^\infty(\Omega)$ is dense in $H^{\alpha/2}(\Omega)$ but $H^{\alpha/2}(\Omega)$ contains all constant functions. Thus \eqref{eq:poincare-friedrichs} fails. Note that, in the case $\nu(h)= |h|^{-d-\alpha}, \alpha\in (0,2)$, a necessary and sufficient condition on $\Omega$ for \eqref{eq:poincare-friedrichs} to hold is provided in \cite{DK21}.
\end{remark}

%%%%%%%%%%%%%%%%%%%%%%%%%%%%%%%%%%%%%%%%%%%%%%%%%%%%%%%%% 
%%%%%%%%%%%%%%%%%%%%%%%%%%%%%%%%%%%%%%%%%%%%%%%%%%%%%%%%% 
\section{Existence of weak solutions and spectral decomposition}\label{sec:existence}
%%%%%%%%%%%%%%%%%%%%%%%%%%%%%%%%%%%%%%%%%%%%%%%%%%%%%%%%% 
%%%%%%%%%%%%%%%%%%%%%%%%%%%%%%%%%%%%%%%%%%%%%%%%%%%%%%%%% 

This section is devoted to the following results: well-posedness of the Neumann problem in \autoref{thm:nonlocal-Neumann-var}, the spectral decomposition of the corresponding operator in \autoref{thm:existence-of-eigenvalue-Neumann}, the Robin problem in \autoref{thm:nonlocal-Robin-var}, and the definition of the nonlocal Dirichlet-to-Neumann map in \autoref{thm:DN-map} together with its spectral decomposition in \autoref{thm:DN-map-spectral}. We refer the reader to \autoref{sec:intro} for comments about related expositions in the literature.

\medskip

Throughout this section, $\Omega\subset \R^d$ is assumed to be open. We recall that the function $\nu:\R^d \setminus\{0\}\to [0,\infty] $ is assumed to be symmetric and to satisfy the L\'{e}vy integrability condition \eqref{eq:levy-cond}. Let $k: \R^d \times \R^d  \setminus \operatorname{diag} \to [0, \infty)$ be symmetric and measurable such that for some $\Lambda \geq 1$
\begin{align}\label{eq:k-elliptic}\tag{E}
\Lambda^{-1} \nu(y-x) \leq k(x,y) \leq \Lambda \nu(y-x) \quad (x,y \in \R^d).
\end{align}
We will formulate well-posedness results for equations $L u = f$ in $\Omega$, where 
\begin{align}
Lu(x) = \pv \int_{\R^d} \big(u(y) - u(x) \big) k(x,y) \d y \,.
\end{align}
Note that the expression $Lu(x)$ does not exist in general if $u$ is smooth. One would require additional assumptions on $k$. Note that $L$ can be understood as an integro-differential operator. The aforementioned phenomenon is similar to the fact that  expressions like $\operatorname{div} \big( A(x) \nabla u (x) \big)$ do not exist in general for smooth functions $u$ without further assumptions on the matrix $A(x)$. Given functions $u,v \in \VnuOm$, we define a bilinear form $\cE$ by 
\begin{align}\label{eq:def-Euu-k}
\mathcal{E}(u,v) =\frac{1}{2} \!\!\iil_{(\Omega^c\times \Omega^c)^c} \!\! \big(u(x)-u(y) \big) \big(v(x)-v(y) \big) \, k(x,y)  \d x \, \d y \,.
\end{align}
Note that under the condition \eqref{eq:k-elliptic} the expression $\mathcal{E}(u,v)$ is well defined for $u,v \in \VnuOm$.

\begin{definition}\label{def:nonlocal-normal}
We define a nonlocal operator $\cN$ acting on functions $v: \R^d \to \R$ by 
\begin{align}\label{eq:def-N}
\mathcal{N} v (y)&= \int_{\Omega}(v(y)-v(x)) k(x,y) \d x \qquad (y\in \Omega^c).
\end{align}
\end{definition}
Note that \eqref{eq:def-N} requires some integrability condition of $v$. If $\nu$ is a unimodal L\'{e}vy measure, then $\mathcal{N} v (y)$ is well defined for $v \in L^1(\R^d, \widehat{\nu})$, see the beginning of the proof of \autoref{prop:fund-prop}. Furthermore, the definition of $\mathcal{N} v (y)$, $y \in \Omega^c$, does not require any principal value integral, because there is a positive distance between $y$ and $\overline{\Omega}$.

\begin{remark} (i) We work under the assumption \eqref{eq:k-elliptic} in order to establish well-posedness for complement value problem. One could replace this assumption by the assumption 
\begin{align}\label{eq:k-elliptic-weak}\tag{E'}
\begin{split}
\Lambda^{-1} \!\!\iil_{(\Omega^c\times \Omega^c)^c} \!\! \big(u(x)-u(y) \big)^2  \, \nuxminy  \d x \, \d y \leq \mathcal{E}(u,u) \leq \Lambda \!\!\iil_{(\Omega^c\times \Omega^c)^c} \!\! \big(u(x)-u(y) \big)^2  \, \nuxminy \d x \, \d y 
\end{split}
\end{align}
for all functions $u \in L^2_{loc}(\R^d)$. This assumption allows for many more general cases of $k$, see the discussions in \cite{DyKa20, ChSi20}. 
(ii)Throughout this section we will work with the weight $\widetilde{\nu}$. Analogous results hold true when choosing $\overline{\nu}$ or $\nu^*$ from \autoref{def:different-nus}.
\end{remark}

\subsection{Neumann boundary condition}\label{subsec:neumann}

In light of the Gauss-Green formula \eqref{eq:green-gauss-nonlocal} it is reasonable to define weak solutions of the Neumann problem under consideration as follows.  Assume $\Omega\subset \R^d$ is an open set. Let $f:\Omega\to \mathbb{R}$ and $g: \R^d\setminus \Omega\to \mathbb{R}$ be two measurable functions. The Neumann problem for the operator $L$ associated to the data $f$ and $g$ is to find a measurable  function $u:\R^d\to \mathbb{R}$ such that 
\begin{align}\label{eq:nonlocal-Neumann}\tag{$N$}
L u = f \quad\text{in}~~~ \Omega \quad\quad\text{ and } \quad\quad \mathcal{N} u= g ~~~ \text{on}~~~ \R^d\setminus\Omega.
\end{align}

\begin{definition}\label{def:neumann-var-sol} Let $ f \in \VnuOm'$ and $g \in \TnuOm'$. We say that $u \in \VnuOm$ is a weak solution or a variational solution of the inhomogeneous Neumann problem \eqref{eq:nonlocal-Neumann} if 
\begin{align}\label{eq:var-nonlocal-Neumann-gen}\tag{$V'$}
\mathcal{E}(u,v) = \langle f , v \rangle  + \langle g , v \rangle  \quad \mbox{for all}~~v \in \VnuOm\,, 
\end{align}
where we use the natural embedding $\VnuOm \hookrightarrow \TnuOm$. Note that the existence of a solution $u \in \VnuOm$ implies the compatibility condition $\langle f , 1 \rangle  + \langle g , 1 \rangle = 0$.

\medskip

If, in particular,  $ f \in L^2(\Omega)$ and $g \in L^2(\Omega^c, \widetilde{\nu}^{-1})$, then $u \in \VnuOm$ is a weak solution of \eqref{eq:nonlocal-Neumann} if 
\begin{align}\label{eq:var-nonlocal-Neumann}\tag{$V$}
\mathcal{E}(u,v) = \int_{\Omega} f(x)v(x)\d x +\int_{\Omega^c} g(y)v(y)\d y,\quad \mbox{for all}~~v \in \VnuOm\,.
\end{align}
In this case, the compatibility condition reads
\begin{align}\label{eq:compatible-nonlocal}\tag{$C$}
\int_{\Omega} f(x)\d x +\int_{\Omega^c} g(y)\d y=0 \,.
\end{align}
\end{definition}

\begin{remark}
The compatibility condition \eqref{eq:compatible-nonlocal} is an implicit necessary requirement. Recall that the local counterpart of this compatibility condition associated with \eqref{eq:local-Neumann}, where $g$ is defined on $\partial\Omega$, is given by 
\begin{align}
\int_{\Omega} f(x)\d x +\int_{\partial\Omega} g(y)\d \sigma(y)=0.
\end{align}
\end{remark}

\begin{remark}
(i) Note that \cite[Def. 3.6]{DROV17} looks very similar to \eqref{eq:var-nonlocal-Neumann} at first glance. However, the norm of the test space defined in \cite[Eq. (3.1)]{DROV17} depends on the Neumann data $g$, which is not natural. Our test space $\VnuOm$ in the weak formulation \eqref{eq:var-nonlocal-Neumann} does not depend on the Neumann data $g$. Moreover for the existence of weak solutions to \eqref{eq:nonlocal-Neumann}, it is sufficient to choose $f\in L^2(\Omega)$ and $g\in L^2(\Omega^c, \widetilde{\nu}^{-1})$, see \autoref{thm:nonlocal-Neumann-var-weighted}. (ii) For non-singular kernels, \autoref{def:neumann-var-sol} coincides with the definition in \cite[Section 3.2]{DTZ22}.
\end{remark}

The next result shows that both problems \eqref{eq:nonlocal-Neumann} and \eqref{eq:var-nonlocal-Neumann} are related under additional regularity assumption. 

\begin{proposition}\label{prop:neumann-strong-weak}
Let $\Omega$ be an open bounded subset of $\R^d$ with Lipschitz boundary. Assume $k(x,y) = \nu(y-x)$, i.e., $\Lambda =1$. Let $u\in C^2_b(\R^d)$, $ f\in L^{2}(\Omega)$ and $g\in L^2(\Omega^c, \widetilde{\nu}^{-1})$. Then $u$ satisfies \eqref{eq:nonlocal-Neumann} if and only if $f$ and $g$
are compatible in the sense of \eqref{eq:compatible-nonlocal} and $u$ satisfies \eqref{eq:var-nonlocal-Neumann}. 

\end{proposition}
\medskip

\begin{proof} If $u$ solves \eqref{eq:nonlocal-Neumann} i.e. $Lu=f$ in $\Omega$ and $\mathcal{N}u =g$ on $\Omega^c$, then by the Gauss-Green formula \eqref{eq:green-gauss-nonlocal} we obtain the following 
\begin{align}\label{eq:var-Neumann-regular}
\mathcal{E}(u,v)= \int_{\Omega} f(x)v(x)\d x + \int_{\Omega^c} g(y)v(y)\d y, \quad \mbox{for all}~~v \in C^1_b(\R^d).
\end{align}
As shown in \eqref{eq:linearform-f}-\eqref{eq:linearform-g} below, all terms involved in \eqref{eq:var-Neumann-regular} are linear and continuous on $\VnuOm$ with respect to the variable $v$. Moreover smooth functions of compact support are dense in $\VnuOm$ hence the relation in \eqref{eq:var-Neumann-regular} remains true for functions $v$ in $\VnuOm$ that is \eqref{eq:var-nonlocal-Neumann} is satisfied. In particular taking $v=1$ one gets the condition \eqref{eq:compatible-nonlocal}. 

Conversely, assume $u$ solves \eqref{eq:var-nonlocal-Neumann} then inserting the Gauss-Green formula \eqref{eq:green-gauss-nonlocal} with $v\in C^1_b(\R^d)\subset \VnuOm$ in \eqref{eq:var-Neumann-regular} yields 
\begin{small}
\begin{align*}
\int_{\Omega} L u(x) v(x) \d x - \int_{\Omega} f(x)v(x)\d x = \int_{\Omega^c} g(y)v(y)\d y-\int_{\Omega^c}\mathcal{N} u(y)v(y)\d y, \quad \mbox{for all}~~v \in C^1_b(\R^d).
\end{align*}
\end{small}
Specializing this relation for $ v \in C^\infty_c(\Omega)$ and $v \in C^\infty_c(\R^d\setminus \overline{\Omega})$ respectively we end up with 
\begin{alignat*}{2}
\int_{\Omega} L u(x) v(x) \d x - \int_{\Omega} f(x)v(x)\d x &= 0 \qquad &&\mbox{ for all}~~v \in C^\infty_c(\Omega),
\\
\int_{\Omega^c} g(y)v(y)\d y-\int_{\Omega^c}\mathcal{N} u(y)v(y)\d y &= 0 && \mbox{ for all}~~v \in C^\infty_c(\R^d\setminus \overline{\Omega}). 
\end{alignat*}
Recall that, by \autoref{prop:uniform-cont}, $Lu$ is well defined and bounded. Hence, $Lu$ belongs to $L^2(\Omega)$. Similarly $\mathcal{N}u$ is well defined and bounded, i.e., it belongs to $L^\infty(\Omega^c)$. Thus, up to null sets, we conclude from the above equations $L u = f$ in $\Omega$  and $\mathcal{N} u= g$ on $\R^d\setminus\Omega$, which proves \eqref{eq:nonlocal-Neumann}. 
\end{proof}

Both integrodifferential operators $L$ and $\mathcal{N}$ annihilate additive constants. Whence as long as $u$ is a solution to the system \eqref{eq:nonlocal-Neumann} or to the variational problem \eqref{eq:var-nonlocal-Neumann} so is the function $\widetilde{u}= u+c$ for any $c\in \mathbb{R}$. Accordingly, both problems are ill-posed in the sense of Hadamard. The situation is likewise in the local setting with the operators $L$ and $\mathcal{N}$ respectively replaced by the operators $-\Delta$ and $\frac{\partial}{\partial n}$. In order to overcome this issue, it is common to introduce an appropriate function space $ \VnuOm^{\perp}$ as follows:
%consisting of functions in $ \VnuOm$ with zero mean over $\Omega$. To be more precise, 
%
\begin{align*}
\VnuOm^{\perp}:= \Big\{ u\in \VnuOm: \int_{\Omega}u(x)\d x=0\Big\}.
\end{align*}
Assuming that $\Omega$ is bounded, the space $\VnuOm^{\perp}$ endowed with the scalar product of $\VnuOm$ is Hilbert space as well. Instead of \eqref{eq:var-nonlocal-Neumann} we need to consider the following weak formulations:
\begin{alignat}{2}
\mathcal{E}(u,v) &= \langle f , v \rangle  + \langle g , v \rangle   && \text{ for all } v \in \VnuOm^\perp \,, \label{eq:var-nonlocal-Neumann-gen-bis}\tag{$V^{' \perp}$} \\
\mathcal{E}(u,v) &= \int_{\Omega} f(x)v(x)\d x +\int_{\Omega^c} g(y)v(y)\d y \quad && \text{ for all } v \in \VnuOm^{\perp} \,. \label{eq:var-nonlocal-Neumann-bis}\tag{$V^\perp$}
\end{alignat}

In contrast to \eqref{eq:var-nonlocal-Neumann}, the variational problem \eqref{eq:var-nonlocal-Neumann-bis} possesses at most one solution since $\mathcal{E}(\cdot, \cdot)$ defines a scalar product on $\VnuOm^\perp$. Analogous observations can be made in the local setting by introducing the space $H^1(\Omega)^\perp=\big\{ u\in H^1(\Omega): \int_{\Omega}u(x)\d x=0\big\}$. 

By standard procedure, a solution of the variational problem \eqref{eq:var-nonlocal-Neumann} is characterized as a critical point (a minimizer) of the functional 
\begin{align}
\mathcal{J}(v) &= \frac{1}{2} \mathcal{E}(v,v) - \int_\Omega f v\, \d x- \int_{\Omega^c} g v \d x\\
&= \frac{1}{4}\iint\limits_{(\Omega^c\times\Omega^c)^c} (v(x)-v(y))^2\nuxminy\,\d x\,\d y - \int_\Omega f v\, \d x- \int_{\Omega^c} g v \d x\notag.
\end{align}
\begin{proposition}\label{prop:neumann-min}
Let $\Omega\subset \R^d$ be an open set. Then a function $ u\in\VnuOm^{\perp} $ is a solution to \eqref{eq:var-nonlocal-Neumann-bis} if and only if $u$ is a solution of the minimization problem
\begin{align}\label{eq:nonlocal-Neumann-min}\tag{$M^\perp$}
\mathcal{J}(u) =\min_{v\in \VnuOm^{\perp}}\mathcal{J}(v). 
\end{align}
Moreover, if $ f: \Omega\to \mathbb{R}$ and $g:\Omega^c\to \mathbb{R}$ are compatible in the sense of \eqref{eq:compatible-nonlocal}, $ u\in\VnuOm^{\perp} $ solves \eqref{eq:var-nonlocal-Neumann-bis} if and only if for any $c\in \mathbb{R}$, $u+c$ solves the variational problem \eqref{eq:var-nonlocal-Neumann} and the latter problem is equivalent to the minimization problem 
\begin{align}\label{eq:nonlocal-Neumann-min-const}\tag{$M$}
\mathcal{J}(u) =\min_{v\in \VnuOm}\mathcal{J}(v).
\end{align}
\end{proposition} 

\begin{proof}
Let $ u\in\VnuOm^{\perp}$ so that \eqref{eq:var-nonlocal-Neumann-bis} holds true for all $ v\in\VnuOm^{\perp}$. Employing Cauchy-Schwartz inequality yields
\begin{align*}
\mathcal{E}(u,v) 
%&\leq \sqrt{\mathcal{E}(u,u)} \sqrt{\mathcal{E}(v,v)}\\
%
&\leq \frac{1}{2} \mathcal{E}(u,u)+ \frac{1}{2} \mathcal{E}(v,v)= \mathcal{E}(u,u)- \frac{1}{2} \mathcal{E}(u,u)+ \frac{1}{2} \mathcal{E}(v,v).
\end{align*} 
In virtue of \eqref{eq:var-nonlocal-Neumann-bis} we get $\mathcal{J}(u)\leq \mathcal{J}(v)$ and thus $u $ solves \eqref{eq:nonlocal-Neumann-min}.

Conversely assume that $ u $ satisfies \eqref{eq:nonlocal-Neumann-min} which means that $\mathcal{J}(u)\leq \mathcal{J}(v)$ for all $ v\in\VnuOm^{\perp}$. 
For fixed $ v\in\VnuOm^{\perp}$ the mapping $\mathcal{J}(u+\cdot v): \mathbb{R} \to \mathbb{R}$, 
\begin{align*}
t\mapsto \mathcal{J}(u+tv)= \mathcal{J}(u)+ t\left[\mathcal{E}(u,v)- \int_{\Omega} f(x)v(x)\d x -\int_{\Omega^c} g(y)v(y)\d y \right] + \frac{t^2}{2} \mathcal{E}(v,v) 
\end{align*}
is a polynomial of second order. For all $t \in \mathbb{R}$, $u+tv\in \VnuOm$ and since $u$ minimizes $\mathcal{J}$ we get that $\mathcal{J}(u)\leq \mathcal{J}(u+tv)$ for all $t\in \mathbb{R}$. Thus $\mathcal{J}(u+\cdot v): \mathbb{R} \to \mathbb{R}$ has a critical point at $t=0$ which implies that 
\begin{align*}
0= \lim_{t\to 0}\frac{\mathcal{J}(u+tv)-\mathcal{J}(u)}{t} = \lim_{t\to 0} \left[\mathcal{E}(u,v)- \int_{\Omega} f(x)v(x)\d x -\int_{\Omega^c} g(y)v(y)\d y + \frac{t}{2} \mathcal{E}(v,v) \right]
\end{align*}
equivalently 
\begin{align*}
\mathcal{E}(u,v)= \int_{\Omega} f(x)v(x)\d x +\int_{\Omega^c} g(y)v(y)\d y \,.
\end{align*}
This shows the equivalence between variational problem \eqref{eq:var-nonlocal-Neumann-bis} and the minimization problem \eqref{eq:nonlocal-Neumann-min}. Meanwhile, if the compatibility condition \eqref{eq:compatible-nonlocal} holds, then it is easy to observe that the relation in \eqref{eq:var-nonlocal-Neumann-bis} remains unchanged under additive constant and $\mathcal{J}(v+c)=\mathcal{J}(v)$ for all $v \in\VnuOm$ and all $c\in \mathbb{R}.$ Accordingly, if $u\in \VnuOm^\perp$ solves \eqref{eq:var-nonlocal-Neumann-bis} then we have $\mathcal{J}(u+c) =\min\limits_{v\in \VnuOm}\mathcal{J}(v)$ which, by similar arguments as above, is equivalent to \eqref{eq:var-nonlocal-Neumann}. 
\end{proof}

From \autoref{prop:neumann-strong-weak} and \autoref{prop:neumann-min} we deduce that, analogous to the case of the Laplace operator, the complement condition $\cN u = 0$ turns out to be a natural condition in the variational context:
\begin{corollary}\label{cor:natural-cond}
Let $f \in L^2(\Omega)$. Assume $u \in \VnuOm$ minimizes the functional $v\mapsto \frac12 \cE(v,v) - \int_{\Omega} fv$ in the space $\VnuOm$. Then $\cN u = 0$ in $\Omega^c$.
\end{corollary}

A different version of this observation is given in \cite[Theorem 2.1]{DLV21}. \cite[Theorem 2.8]{MuPr19} is similar in the translation invariant case. We are now in position to state the existence and the uniqueness of a solution to \eqref{eq:var-nonlocal-Neumann-bis} and hence to \eqref{eq:var-nonlocal-Neumann} up to additive constant. A direct application of the Lax-Milgram lemma leads to the following observation.

\begin{theorem} \label{thm:nonlocal-Neumann-var-trivial}
We assume that $\Omega\subset \R^d$ is open and bounded. Let $\nu: \R^d\to [0,\infty]$ be the density of a symmetric L\'{e}vy measure with full support. We further assume  that the couple $(\nu, \Omega)$ belongs to one of the class $\mathscr{A}_i,~i=0,1,2,3$. Let $ f \in \VnuOm'$ and $g\in \TnuOm'$.
\begin{enumerate}[(i)]
\item There exists a unique solution $u\in\VnuOm^{\perp} $ to the problem \eqref{eq:var-nonlocal-Neumann-gen-bis} satisfying
\begin{align*}%\label{eq:weak-regular-trivial}
\|u\|_{\VnuOm}\leq C \left(\|f\|_{\VnuOm'}+\|g\|_{\TnuOm'}\right)
\end{align*}
with a positive constant $C$, which depends only on $d,\Omega, \Lambda$ and $\nu$.
\item  Problem \eqref{eq:var-nonlocal-Neumann-gen} is solvable if and only if $\langle f , 1 \rangle  + \langle g , 1 \rangle = 0$. All solutions $w$ are of the form $w=u+c$ with $c\in \mathbb{R}$ and satisfy 
\begin{align*}%\label{eq:weak-regular-gen-trivial}
\|w-\hbox{$\fint_{\Omega}w$} \|_{\VnuOm}\leq C \left(\|f\|_{\VnuOm'}+\|g\|_{\TnuOm'}\right) \,.
\end{align*}
\end{enumerate}
\end{theorem}

\begin{proof}[Proof] 
The existence and the uniqueness of solutions of \eqref{eq:var-nonlocal-Neumann-gen-bis} follow from the Lax-Milgram lemma. The bilinear form $\mathcal{E}(\cdot, \cdot)$ is continuous on $\VnuOm^\perp$. From the Poincar\'{e} inequality \eqref{eq:poincare-inequalityV} we conclude 
\begin{align*}
\|v\|_{L^2(\Omega)}^2\leq C \mathcal{E}(v,v) \quad\quad \text{for all }~~~v\in \VnuOm^{\perp}
\end{align*}
for some positive constant $C$. This implies coercivity of $\mathcal{E}(\cdot, \cdot)$ on $\VnuOm^\perp$ and we obtain
\begin{align}\label{eq:coercivity-quadratic}
\mathcal{E}(v,v)\geq \big(1+ C\big)^{-1}\|v\|^2_{\VnuOm}.
\end{align}
Note that, due to the continuity  of the trace operator $\operatorname{Tr}:\VnuOm \to\TnuOm$, the mapping $v \mapsto \langle f , v \rangle  + \langle g , v \rangle$ is linear and continuous on $\VnuOm^\perp$. The Lax-Milgram lemma implies (i).

For $v\in\VnuOm$ set $v= \widetilde{v}+c'$ with $c'=\mbox{$ \fint_{\Omega} v\d x$}$ so that $ \widetilde{v} \in\VnuOm^\perp$. 
In addition, every constant function $w=c$ belongs to $\VnuOm$ for every $c\in \R^d$ because $\Omega$ is bounded. 
Hence, $\VnuOm= \VnuOm^{\perp}\oplus \mathbb{R}$. With this observation along with the identity 
$\mathcal{E}(u+c,v+c')=\mathcal{E}(u,v) $ for all $c,c'\in \mathbb{R}$ and the uniqueness of $u\in\VnuOm^{\perp}$
solving \eqref{eq:var-nonlocal-Neumann-gen-bis} it becomes easy to check that under the compatibility condition $\langle f , 1 \rangle  + \langle g , 1 \rangle = 0$, all solutions of \eqref{eq:var-nonlocal-Neumann-gen} are of the form $u+c$.

\end{proof}

\begin{remark}
It worth to mention that, \autoref{thm:nonlocal-Neumann-var-trivial} (i) implies that the operator $\Phi : \VnuOm' \times \TnuOm' \to \VnuOm^\perp $ mapping $(f,g)$ to the unique solution $ u\in\VnuOm^\perp $ of the variational problem
\eqref{eq:var-nonlocal-Neumann-gen-bis} is linear, one-to-one, continuous with 	%
\begin{align*}
\|\Phi(f,g)\|_{\VnuOm}\leq C \|(f,g)\|_{\VnuOm' \times \TnuOm'} \,.
\end{align*} 
\end{remark}

Let us apply \autoref{thm:nonlocal-Neumann-var-trivial} in order to prove our main existence result.

\begin{theorem} \label{thm:nonlocal-Neumann-var}
Under the assumptions of \autoref{thm:nonlocal-Neumann-var-trivial} with $ f \in L^2 (\Omega)$ and $g\in L^2(\Omega^c, \widetilde{\nu}^{-1})$ the following holds true:
\begin{enumerate}[(i)]
\item There exists a unique solution $u\in\VnuOm^{\perp} $ to the problem \eqref{eq:var-nonlocal-Neumann-bis} satisfying
\begin{align*}%\label{eq:weak-regular}
\|u\|_{\VnuOm}\leq C \left(\|f\|_{L^2(\Omega)}+\|g\|_{L^{2}(\Omega^c, \widetilde{\nu}^{-1})}\right)
\end{align*}
with a positive constant $C$, which depends only on $d,\Omega,\Lambda$ and $\nu$.
\item  Problem \eqref{eq:var-nonlocal-Neumann} is solvable if and only if \eqref{eq:compatible-nonlocal} holds true. All solutions $w$ are of the form $w=u+c$ with $c\in \mathbb{R}$ and satisfy 
\begin{align}\label{eq:weak-regular}
\|w-\hbox{$\fint_{\Omega}w$} \|_{\VnuOm}\leq C \left(\|f\|_{L^2(\Omega)}+\|g\|_{L^{2}(\Omega^c, \widetilde{\nu}^{-1})}\right) \,.
\end{align}
\end{enumerate}
\end{theorem}

\begin{proof}
It suffices to show the continuity of the associated linear forms. For $v \in\VnuOm^{\perp}$
\begin{align}\label{eq:linearform-f}
\Big|\int_{\Omega} fv\d x \Big|\leq \|f\|_{L^{2}(\Omega)}\|v\|_{L^{2}(\Omega)} \leq \|f\|_{L^{2}(\Omega)}\|v\|_{\VnuOm}\,.
\end{align}
From $g\in L^2(\Omega^c, \widetilde{\nu}^{-1})$ and the continuity of $ \operatorname{Tr}: \VnuOm \hookrightarrow \TnuOm,$ we obtain
\begin{align}\label{eq:linearform-g}
\Big|\int_{\Omega^c} g(x) v(x)\d x \Big|\leq \|g\|_{ L^2(\Omega^c, \widetilde{\nu}^{-1})}\|v\|_{L^2(\Omega^c, \widetilde{\nu})} \leq C\|g\|_{ L^2(\Omega^c, \widetilde{\nu}^{-1})}\|v\|_{\VnuOm}.
\end{align}
Application of \autoref{thm:nonlocal-Neumann-var-trivial} completes the proof. 
\end{proof}

There is an alternative formulation of \autoref{thm:nonlocal-Neumann-var}, which allows for more general inhomogeneities $g$.  Let us define a modified Neumann problem for the operator $L$ associated to the data $f$ and $g$ as follows: 
\begin{align}\label{eq:nonlocal-Neumann-mod}\tag{$N_*$}
L u = f \quad\text{in}~~~ \Omega \quad\quad\text{ and } \quad\quad \mathcal{N} u= g \widetilde{\nu} \quad \text{ on } \R^d\setminus\Omega.
\end{align}

\medskip

\begin{theorem} \label{thm:nonlocal-Neumann-var-weighted}
Under the assumptions of \autoref{thm:nonlocal-Neumann-var-trivial} with $ f \in L^2 (\Omega)$ and $g\in L^2(\Omega^c, \widetilde{\nu})$, then the following holds true:
\begin{enumerate}[(i)]
\item There exists a unique weak solution $u_{*}\in\VnuOm^{\perp} $ to the  problem \eqref{eq:nonlocal-Neumann-mod},  that is
\begin{align}\label{eq:var-nonlocal-Neumann-bis-weigthed}\tag{$V_{*}^\perp$}
\mathcal{E}(u_{*}, v) = \int_{\Omega} f(x)v(x)\d x +\int_{\Omega^c} g(y)v(y)\widetilde{\nu}(y)\d y \quad \text{ for all } v \in \VnuOm^{\perp}\,.
\end{align}
satisfying
\begin{align*}%\label{eq:weak-regular-weighted}
\|u_*\|_{\VnuOm}\leq C \left(\|f\|_{L^2(\Omega)}+\|g\|_{L^{2}(\Omega^c, \widetilde{\nu})}\right)
\end{align*}
with a positive constant $C$, which depends only on $d,\Omega,\Lambda$ and $\nu$.

\item Problem \eqref{eq:var-nonlocal-Neumann-weigthed} is solvable if and only if \eqref{eq:compatible-nonlocal-weighted} holds true, where 
\begin{align}
\label{eq:var-nonlocal-Neumann-weigthed}\tag{$V_{*}$}
\mathcal{E}(u,v) &= \int_{\Omega} f(x)v(x)\d x +\int_{\Omega^c} g(y)v(y)\widetilde{\nu}(y)\d y  \quad \text{ for all } v \in \VnuOm\,, \\
\tag{$C_{*}$}\label{eq:compatible-nonlocal-weighted}
&\int_{\Omega} f(x)\d x +\int_{\Omega^c} g(y)\widetilde{\nu}(y)\d y = 0 \,.
\end{align}
All solutions $w_*$ are of the form $w=u_{*}+c$ with $c\in \mathbb{R}$ and satisfy 
\begin{align*}%\label{eq:weak-regular-weighted}
\|w-\hbox{$\fint_{\Omega}w$} \|_{\VnuOm}\leq C \left(\|f\|_{L^2(\Omega)}+\|g\|_{L^{2}(\Omega^c, \widetilde{\nu})}\right)\,.
\end{align*}

\end{enumerate}
\end{theorem}
The proof of \autoref{thm:nonlocal-Neumann-var-weighted} is analogous to the one of \autoref{thm:nonlocal-Neumann-var}. Note that, if $g\in L^2(\Omega^c, \widetilde{\nu})$, then 
\begin{align*}%\label{eq:linearform-weighted-g}
\Big|\int_{\Omega^c} g(x) v(x)\widetilde{\nu}(x)\d x \Big| \leq C\|g\|_{L^2(\Omega^c, \widetilde{\nu})} \|v\|_{\VnuOm} \,.
\end{align*}

The last result in this section concerns the non-existence of weak solutions when the Neumann data $g$ is not in the weighted trace space $L^{2}(\Omega^c,\widetilde{\nu}^{-1})$. 

\begin{theorem}[\textbf{Non-existence of weak solution}]\label{thm:non-existence-Neumann}
Let $B_1= B_1(0)$ be the unit ball in $\R^d$. Let $\nu(h)=|h|^{-d-\alpha}$, $\alpha\in (0,2)$ so that $\widetilde{\nu}(h)\asymp (1+|h|)^{-d-\alpha}$. Let $f=0$. There exists $g\in L^1(B_1^c)\setminus L^2(B_1^c, \widetilde{\nu}^{-1})$ with $\int_{B_1^c} g (y)\d y=0$, for which  the  Neumann problem $Lu=0$ on $B_1$ and $\mathcal{N} u=g$ on $\R^d\setminus B_1$ has no weak solution in $V_{\nu}(B_1 |\R^d)$. 
\end{theorem}

\begin{proof}
We will construct a function $g$ of the form $g= g_\gamma \widetilde{\nu}$ where, given  $x \in \R^d$,  
\begin{align*}
g_\gamma(x)=\frac{x_1}{|x|} (|x|-1)^{\gamma}\mathds{1}_{B^c_1(0)} (x) \,,
\end{align*}
with an appropriate choice of $\gamma\in (-1, -\frac{\alpha+1}{2}) \cup (\frac{\alpha}{2}, \frac{\alpha+1}{2})$. Note that for $x\in B^c_1(0)$ we have $\dist(x, \partial B_1(0)) = (|x|-1)$ and 
\begin{align*}
\int_{B_1(0)}\frac{\d y}{|x-y|^{d+\alpha}}\asymp (|x|-1)^{-\alpha}\land (|x|-1)^{-d-\alpha}. 
\end{align*}
%$\bullet $ 
Claim 1: $g_\gamma\in V_{\nu}(B_1 |\R^d)$ if and only if $\gamma\in (\frac{\alpha-1}{2},\frac{\alpha}{2})$. Integration in polar coordinates yields
\begin{align*}
\|g_\gamma\|^2_{V_{\nu}(B_1 |\R^d)}&= 2\int_{B^c_1(0) }\frac{x_1^2}{|x|^2} (|x|-1)^{2\gamma} \int_{ B_1(0) }|x-y|^{-d-\alpha}\d y\,\d x\\
&\asymp \int_{B^c_1(0)} \frac{x_1^2}{|x|^2} (|x|-1)^{2\gamma-\alpha}(1\land (|x|-1)^{-d})\d x  \\
%&= 2|\mathbb{S}^{d-1}|\int_1^2 (r-1)^{2\gamma-\alpha} r^{d-1}\d r+ 2|\mathbb{S}^{d-1}|\int_2^\infty (r-1)^{2\gamma-\alpha} r^{-1}\d r\\
&\asymp 
%2|\mathbb{S}^{d-1}|K_{d,p}
\Big( \int_0^1 r^{2\gamma-\alpha} \d r+ \int_1^\infty r^{2\gamma-\alpha-1} \d r\Big).
\end{align*}
Claim 2: $g_{\gamma+\beta}\in L^1(B_1^c,\widetilde{\nu})$ if and only if $\gamma+\beta\in (-1,\alpha)$. Indeed, 
\begin{align*}
\|g_{\gamma+\beta}\|_{L^1(B_1^c, \widetilde{\nu})}&= \int_{B^c_1(0) }\frac{x_1^2}{|x|^2} (|x|-1)^{\gamma+\beta} (1+|x|)^{-d-\alpha} \d x\\
%&=|\mathbb{S}^{d-1}|\int_1^\infty (r-1)^{\gamma+\beta}(r+1)^{-d-\alpha} r^{d-1}\d r\\
&\asymp |\mathbb{S}^{d-1}|\Big( \int_0^1 r^{\gamma+\beta} \d r+ \int_1^\infty r^{\gamma+\beta-\alpha-1} \d r\Big).
\end{align*}
Claim 3: Analogously, $g_{\gamma}\in L^2(B_1^c,\widetilde{\nu})$ if and only if $\gamma\in (-\frac{1}{2}, \frac{\alpha}{2})$.  \\
Claim 4: $g\in L^1(B_1^c)\setminus L^2(B_1^c, \widetilde{\nu}^{-1})$ if and only if $g_\gamma\in L^1(B_1^c, \widetilde{\nu})\setminus L^2(B_1^c, \widetilde{\nu})$ if and only if $\gamma\in (-1, -\frac{1}{2}]\cup [\frac{\alpha}{2}, \alpha)$ by Claim 2 and Claim 3.

\medskip
From now on, we assume $\gamma\in (1, -\frac{\alpha+1}{2}) \cup (\frac{\alpha}{2}, \frac{\alpha+1}{2}) \subset (-1, -\frac{1}{2}]\cup [\frac{\alpha}{2}, \alpha)$. \\
Claim 5: Since $g_\gamma(x)=-g_\gamma(-x)$ it follows that $g= g_\gamma\widetilde{\nu}$ satisfies the compatibility condition
\begin{align*}
	\int_{B_1^c} g(y)\d y=\int_{B_1^c} g_\gamma(y)\widetilde{\nu}(y)\d y=0. 
\end{align*}  

\medskip

 We assume the Neumann problem $Lu=0$ on $B_1$ and $\mathcal{N} u=g$ on $\R^d \setminus B_1$ has a weak solution $u\in V_{\nu}(B_1 |\R^d)$, that is we have 
\begin{align*}
\mathcal{E}(u,v)=\int_{B_1^c} g_\gamma(y)v(y)\widetilde{\nu}(y)\d y\qquad \text{for all}\quad v\in V_{\nu}(B_1 |\R^d).
\end{align*}
This implies that for the constant $C= 1+ \|u\|_{V_{\nu}(B_1 |\R^d)} >0$ we have 
\begin{align}\label{eq:estim-guy}
\Big|\int_{B_1^c} g_\gamma(y)v(y)\widetilde{\nu}(y)\d y\Big|\leq C \|v\|_{V_{\nu}(B_1 |\R^d)} 
\quad \text{ for all}\quad  v\in V_{\nu}(B_1 |\R^d).
\end{align}
Claim 1 allows us to take $v=g_\beta\in V_{\nu}(B_1 |\R^d)$ for  $\beta\in (\tfrac{\alpha-1}{2}, \tfrac{\alpha}{2})$. Then we obtain  
\begin{align}\label{eq:estim-guy-second}
\|g_{\gamma+\beta}\|_{L^1(B_1^c, \widetilde{\nu})} \leq c(d) \int_{B_1^c} g_\gamma(y) g_\beta(y)\widetilde{\nu}(y)\d y \leq C \|g_\beta\|_{V_{\nu}(B_1 |\R^d)} \,,
\end{align}
where we have used \eqref{eq:estim-guy}. Finally, we consider two cases. If $\alpha\geq 1$, then we choose $\gamma\in (\frac{\alpha}{2}, \frac{\alpha+1}{2})$ and $\beta=\alpha-\gamma$. If $\alpha\leq1$, then we choose $\gamma\in (-1, -\frac{\alpha+1}{2})$ and $\beta= -\gamma-1$.  In both cases, $\gamma\in (-1, -\frac{1}{2}]\cup [\frac{\alpha}{2}, \alpha)$,  $\beta\in (\frac{\alpha-1}{2}, \frac{\alpha}{2})$ and $\gamma+\beta\in \{-1,\alpha\}$. This implies $g_\gamma\in L^1(B_1^c, \widetilde{\nu})\setminus L^2(B_1^c, \widetilde{\nu})$ and $g_\beta\in V_{\nu}(B_1 |\R^d)$ whereas $\|g_{\gamma+\beta}\|_{L^1(B_1^c, \widetilde{\nu})}=\infty$. This contradicts \eqref{eq:estim-guy-second} since $\|g_{\beta}\|_{V_{\nu}(B_1 |\R^d)}<\infty$. 

%\medskip
%
%Guy:
%
%\medskip
%
%Now if $\alpha\geq1$, consider $\gamma\in (\frac{\alpha}{2}, \frac{\alpha+1}{2})\subset  [\frac{\alpha}{2}, \alpha)$ and take $\beta=\alpha-\gamma$.
%If $\alpha\leq1$ consider $\gamma\in (-1, -\frac{\alpha+1}{2})\subset  (-1, -\frac{1}{2}]$ and take $\beta= -\gamma-1$. 
%In both cases, $\gamma\in (-1, -\frac{1}{2}]\cup [\frac{\alpha}{2}, \alpha)$ and 
%$\beta\in  (\frac{\alpha-1}{2}, \frac{\alpha}{2})$ and $\gamma+\beta\in \{-1,\alpha\}$. 
%In other words, $g_\gamma\in L^1(\Omega^c, \widetilde{\nu})\setminus L^{2}(\Omega^c, \widetilde{\nu})$ and $g_\beta\in \VnuOm$. Whence $\|g_{\gamma+\beta}\|_{L^1(\Omega^c, \widetilde{\nu})}=\infty$ and $\|g_{\beta}\|_{\VnuOm}<\infty$, which contradicts the above inequality. 
%
\end{proof}

\subsection{Neumann eigenvalues of $L$}
Let $f\in L^2(\Omega)$, for $g=0$ it is worthwhile to see that the variational problem \eqref{eq:var-nonlocal-Neumann} coincides with \eqref{eq:var-nonlocal-Neumann-weigthed} and both correspond to the variational(weak) formulation of the homogeneous Neumann problem $Lu=f$ in $\Omega$ and $\mathcal{N}u=0$ on $\Omega^c$. 

\begin{definition}[Neumann eigenvalue of $L$]
A non-zero function $u \in \VnuOm$ is called a Neumann eigenfunction of the operator $L$ on $\Omega$ if there exists a real number $\mu$, which is the eigenvalue associated to $u$, such that for all $v \in \VnuOm$
\begin{align*}
\mathcal{E}( u ,v) = \mu \int_{\Omega} u(x)v(x)\d x.
\end{align*}
One formally writes $Lu=\mu u $ in $\Omega$ and $\mathcal{N}u=0$ on $\Omega^c$, which corresponds to the aforementioned weak formulation provided that $u $ is sufficiently regular.
\end{definition} 

It is worth noticing that, if $u$ is a Neumann eigenfunction of $L$ with associated eigenvalue $\mu$, then either $u \in \VnuOm^\perp$ when $\mu\neq 0$ or else, $\mu=0$ and the constant functions $u=c, ~c\in \mathbb{R}\setminus\{0\}$, are the related eigenfunctions.

\begin{theorem} \label{thm:existence-of-eigenvalue-Neumann}
Assume $\Omega\subset \R^d$ is bounded and open and $\nu: \R^d\setminus\{0\}\to [0,\infty)$ is the density of a symmetric L\'{e}vy measure with full support. Assume that the couple $(\nu, \Omega)$ belongs to one of the classes $\mathscr{A}_i,~i=1,2,3$. Then there exists a sequence $(\phi_n)_{\in \mathbb{N}_0}$ in $\VnuOm$, which forms an orthonormal basis of $L^2(\Omega)$, and an increasing sequence of real numbers $ 0=\mu_0<\mu_1\leq \cdots\leq \mu_n\leq \cdots. $ such that $\mu_n \to \infty$ as $n \to \infty$ and each $\phi_n$ is a Neumann eigenfunction of $L$ with corresponding eigenvalue $\mu_n$. The number of each eigenvalue is given by its geometric multiplicity. 
\end{theorem}

\begin{proof}
For $f_1, f_2\in L^2(\Omega)$ let us denote $ u_{f_k}=\Phi_0(f_k) =\Phi(f_k,0)\in \VnuOm^{\perp}, ~k=1,2$ the unique solution of \eqref{eq:var-nonlocal-Neumann-bis} with Neumann data $f=f_k$ and $g=0$. Precisely,

\begin{align}\label{eq:var-nonlocal-Neumann-test}
\mathcal{E}( \Phi_0(f_k),v) = \int_{\Omega} f_k(x)v(x)\d x \quad \mbox{for all}~~v \in \VnuOm^{\perp}\,.
\end{align}

Testing \eqref{eq:var-nonlocal-Neumann-test} against $v= \Phi_0(f_2)$ and $v=\Phi_0(f_1)$ successively when $k=1$ and $k=2$ yields 
\begin{align*}
\big(f_1, \Phi_0(f_2)\big)_{L^2(\Omega)} = \mathcal{E}(\Phi_0(f_1), \Phi_0(f_2)) = \mathcal{E}(\Phi_0(f_2), \Phi_0(f_1)) = \big(f_2,\Phi_0(f_1)\big)_{L^2(\Omega)} .
\end{align*}
Therefore, the operator $R_\Omega\circ \Phi_0: L^2(\Omega)\xrightarrow[]{\Phi_0} \VnuOm^{\perp}\xrightarrow[]{R_\Omega} L^2(\Omega)^{\perp}$ is compact (by \autoref{thm:embd-compactness}) and symmetric hence self-adjoint. It is a fact from the spectral theory of compact self-adjoint operators that $L^2(\Omega)^\perp$ has an orthonormal basis $(e_n)_n$ whose elements are eigenfunctions of $R_\Omega\circ \Phi_0$ and the sequence of the corresponding eigenvalues are non-negative real numbers $(r_n)_n$ which we assume ordered in the decreasing order, $r_1\geq r_2\geq \cdots\geq r_n\geq \cdots 0$ such that $r_n\to 0$ as $n\to \infty$. Precisely, for each $n\geq 1$, $R_\Omega\circ \Phi_0(e_n) = r_n e_n$ or simply write $\Phi_0(e_n) = r_n e_n$ a.e in $\Omega$. Combining the latter relation with definition of $ \Phi_0(e_n)$ we get 
\begin{align*}
\mathcal{E}( \Phi_0(e_n),v) &= \int_{\Omega} e_n(x)v(x)\d x= r^{-1}_n \int_{\Omega}\Phi_0( e_n) (x)v(x)
\d x\, \quad \mbox{for all}~~v \in \VnuOm^{\perp}\,. 
\end{align*}
Equivalently, setting $\mu_n= r^{-1}_n$ and $\phi_n = \Phi_0(e_n)/\|\Phi_0(e_n)\|_{L^2(\Omega)}= r_n^{-1} \Phi_0(e_n) $ which is clearly an element of $\in \VnuOm^\perp$ yields 
\begin{align*}
\mathcal{E}( \phi_n,v) &= \mu_n \int_{\Omega} \phi_n (x)v(x)\d x\, \quad \mbox{for all}~~v \in \VnuOm^{\perp}\, .
\end{align*}
Hereby, along with $\mu_0=0$ and $\phi_0= |\Omega|^{-1}$ provides the sequences sought for. 
Now if we assume $\mu_1=0$ then we have $\phi_1\in \VnuOm^\perp$ and $\mathcal{E}(\phi_1, v) = 0$ for all $v \in \VnuOm^\perp$ in particular $\mathcal{E}(\phi_1, \phi_1) = 0$ i.e $\phi_1$ is a constant function in $\VnuOm^\perp$ necessarily $\phi_1=0$ since $u_1$ has zero mean over $\Omega$. We have therefore reached a contradiction as $\phi_1 $ is supposed to be an eigenfunction i.e $\phi_1\neq 0$. Thus, $\mu_1>0$ and the proof is complete.

\end{proof}

%\begin{remark}
% An alternative perspective to construct the eigenpair $(\mu_n, \phi_n)_n$ is provided by the Rayleigh iterative algorithm. It allows to realize the eigenvalue $\mu_n$ as the infimum of $\mathcal{E}(\cdot, \cdot)$ constrained on some specific manifold of $\VnuOm$. 
%Moreover this strategy offers a possibility to avoid techniques exclusively related to the framework of 
%Hilbert spaces. Thus one can extend the spectral decomposition to the corresponding $L^p(\Omega)$ space ($1<p<\infty$) with related functionals replaced appropriately. However the compactness result remains the cornerstone for the existence of such a family. 
%\end{remark}

\subsection{Robin boundary condition}
In this section we treat a Robin-type problem with respect to the nonlocal operator $L$ on $\Omega$. In the classical setting for the Laplace operator, the Robin boundary problem -- also known as Fourier boundary problem or third boundary problem -- is a combination of the Dirichlet and Neumann boundary problem in the form\footnote{According to the over 20 years survey work \cite{GuAb98},  there is no historical evidence why the problem \eqref{eq:robin-local} is termed  after Robin's name. The survey \cite[p.69]{GuAb98} also points out that the first mathematical appearance of the problem \eqref{eq:robin-local} goes back at least to the works on cooling law by Fourier(1822) and/or Newton (1701, but mathematical contribution  by Newton is uncertain).}
\begin{align}\label{eq:robin-local}
-\Delta u = f \text{ in }  \Omega\quad \text{ and } \quad \frac{\partial u}{\partial n} + \beta u= g \text{ on } \partial\Omega.
\end{align}
Here $f\in L^2(\Omega)$ and the measurable functions $\beta, g: \partial\Omega\to \mathbb{R}$ are given. Analogously, in the nonlocal set up, we assume that $\beta, g: \Omega^c\to \mathbb{R}$ are measurable functions. The Robin problem consists in finding a measurable function $u:\R^d\to \mathbb{R}$ such that 
\begin{align}\label{eq:robin-problem}
L u = f \text{ in }  \Omega\quad\text{ and }\quad \mathcal{N}u + \beta u= g \text{ on } \Omega^c.
\end{align}
\noindent
%For the sake of simplicity we assume $b(x)= \beta(x)\widetilde{\nu}(x)$ and $h(x) = g(x)\widetilde{\nu}(x)$, $x\in \Omega^c$. 
Note that, for  $\beta=0$ one recovers the inhomogeneous Neumann problem. Informally, for $\beta \to \infty$ it leads to the homogeneous Dirichlet problem. Define the quadratic form 
\begin{align*}
Q_\beta(u,v)= \mathcal{E}( u,v)+ \int_{\Omega^c} u(y)v(y)\beta(y)\d y. 
\end{align*} 
A function $u \in \VnuOm$ is called a weak solution of the Robin problem \eqref{eq:robin-problem} if 
\begin{align}\label{eq:weak-robin-problem}
Q_\beta(u,v)=\int_{\Omega} f(x)v(x)\d x+ \int_{\Omega^c} g(y)v(y)\d y \text{ for all } v\in \VnuOm.
\end{align}

\medskip

\begin{theorem} \label{thm:nonlocal-Robin-var}
Let $\nu$ and $\Omega$ be as in \autoref{thm:existence-of-eigenvalue-Neumann}. Assume that  $ \beta\widetilde{\nu}^{-1}: \Omega^c\to [0, \infty) $ is essentially bounded  and $\beta$ is non-trivial that is,  $|\Omega^c\cap \{ \beta>0\}|>0$.  Let $ f \in L^2 (\Omega)$ and $g\in L^2(\Omega^c, \widetilde{\nu}^{-1})$. There exists a unique function $u\in\VnuOm $ solution to \eqref{eq:weak-robin-problem} satisfying 
\begin{align}\label{eq:weak-regular-robin}
\|u \|_{\VnuOm}\leq C \left(\|f\|_{L^2(\Omega)}+\|g\|_{L^{2}(\Omega^c, \widetilde{\nu}^{-1})}\right) \,,
\end{align}
where $C: = C(d,\Omega,\Lambda, \nu, \beta)>0$ can be chosen independently of $u$, $f$ and $g$.
\end{theorem}

%The following formulation is equivalent.
%\begin{corollary} \label{thm:nonlocal-Robin-var-alt}
%	Let $\nu$ and $\Omega$ be as in \autoref{thm:existence-of-eigenvalue-Neumann}. Let $ \beta \widetilde{\nu}^{-1}: \Omega^c\to [0, \infty) $ be essentially bounded such that $\beta >0~ a.e$ on a subset of positive of $\Omega^c$.  Let $ f \in L^2 (\Omega)$ and $g\in L^2(\Omega^c, \widetilde{\nu}^{-1})$. There exists a unique function $u\in\VnuOm $ satisfying for all $v\in \VnuOm$
%	 \begin{align}
%	\label{eq:weak-robin-problem-alt}
%	 \mathcal{E}( u,v)+ \int_{\Omega^c} u(y)v(y) \beta(y)\d y &=\int_{\Omega} f(x)v(x)\d x+ \int_{\Omega^c} g(y)v(y) \d y \,, \\
%	\label{eq:weak-regular-robin-alt}
%	\|u \|_{\VnuOm} &\leq C \left(\|f\|_{L^2(\Omega)}+\|g\|_{L^{2}(\Omega^c, \widetilde{\nu}^{-1})}\right) \,,
%	\end{align}
%	where $C: = C(d,\Omega, \nu, \beta)>0$ can be chosend independently of $u$, $f$ and $g$.
%\end{corollary}

\begin{remark}
The operator $\Psi : L^{2}(\Omega) \times L^{2}(\Omega^c, \widetilde{\nu}^{-1})\to\VnuOm$ mapping the data $(f,g)$ to the unique solution $ u\in\VnuOm$ of the variational problem \eqref{eq:weak-robin-problem} is linear, one-to-one, and continuous. Moreover, with $C$ as above, 
\begin{align*}
\|\Psi(f,g)\|_{\VnuOm}\leq C \|(f,g)\|_{L^{2}(\Omega)\times L^{2}(\Omega^c, \widetilde{\nu}^{-1})}.
\end{align*} 
\end{remark}

\medskip

\begin{proof}
First of all, we claim that the form $Q_\beta(\cdot, \cdot)$ is coercive on $\VnuOm$. Assume it is not true. Then for each $n\geq 1$ there exists $u_n \in\VnuOm $ with $\|u_n\|_{\VnuOm}=1$ such that 
\begin{align*}
\mathcal{E}( u_n,u_n)+ \int_{\Omega^c} |u_n(y)|^p\beta(y)\d y= Q_\beta(u_n , u_n )
<\frac{1}{2^n}\,. 	
\end{align*}
In virtue of our compactness result, \autoref{thm:embd-compactness}, $(u_n)_n$ converges up to a subsequence in $L^2(\Omega)$ to some $u\in \VnuOm$. We deduce $\|u\|_{L^2(\Omega)}= 1$, since $\mathcal{E}(u_n , u_n ) \xrightarrow[]{n \to \infty}0$ and for all $n\geq 1$, $\|u_n\|_{\VnuOm}=1$. From $\mathcal{E}(u_n , u_n )\xrightarrow{n \to \infty} 0$ and $\|u_n-u\|_{L^p(\Omega)} \xrightarrow{n \to \infty}0$ we obtain that $u_n$ converges to $u$ in $\VnuOm$ with $\mathcal{E}(u,u)=0$. Thus $u $ is constant almost everywhere in $\R^d$. On the other hand, since $\beta$ is bounded and the embedding $\VnuOm\hookrightarrow L^2(\Omega^c, \widetilde{\nu})$, see \autoref{lem:natural-norm-on-V}, is continuous, we have 
\begin{align*}
\int_{\Omega^c} u^2(y)\beta(y)\d y&\leq 2 \int_{\Omega^c} u_n^2(y)\beta(y)\d y+2\|\beta\widetilde{\nu}^{-1}\|_{L^\infty (\Omega^c)}\int_{\Omega^c} (u_n(y)-u(y))^2\widetilde{\nu}(y)\d y\\
&\leq 2Q_\beta(u_n, u_n)+ C\|u_n-u\|^2_{\VnuOm}\xrightarrow[]{n \to \infty}0\,.
\end{align*}

From this, we conclude $u=0$ since we know that $u$ is a constant function and $\beta>0$ almost everywhere on a set of positive measure $U\subset \Omega^c$ on which $u $ vanishes. This contradicts $\|u\|_{L^2(\Omega)}=1$ and hence our initial assumption was wrong. Therefore there exists a constant $C=C(d, \Omega, \nu, \beta)>0$ such that
\begin{align}\label{eq:coercive-robin}
Q_\beta(u,u)\geq C\|u\|^2_{\VnuOm}\quad\text{for all }~u\in\VnuOm.
\end{align}

The remaining requirements for the application of the Lax-Milgram lemma can be checked easily. Existence of a unique solution to \eqref{eq:weak-robin-problem} follows. The estimate \eqref{eq:weak-regular-robin} is a direct consequence of \eqref{eq:coercive-robin}. 
\end{proof}

\subsection{Dirichlet-to-Neumann map}

\smallskip

In this section we define the Dirichlet-to-Neumann map related to  the nonlocal L\'evy operator $L$ under consideration. Afterwards we prove that its spectrum is strongly connected to the Robin eigenvalues of the operator $L$. This was originally introduced in \cite{guy-thesis}. We refer the  interested reader to the expositions \cite{ArMa12,Bet15} where the Dirichlet-to-Neumann map is treated in the local setting for the Laplacian. We point out that an attempt to define the Dirichlet-to-Neumann map is provided in \cite{Von21}. For the case for the fractional Laplacian a different Dirichlet-to-Neumann map to ours is derived in \cite{GSU16}, see also the variant for fractional regional operators in \cite{War15,War18}. Let us first review the nonlocal Dirichlet problem. In the spirit of \cite{FKV15} one can easily prove the following 
\begin{theorem}\label{thm:nonlocal-Dirichlet-var}
Let $\Omega\subset \R^d$ be open and bounded.  Given $f\in L^2(\Omega)$ and $g\in \TnuOm$, there exists a unique function $u\in \VnuOm$ with $u=g$ a.e. on $\Omega^c$ and 
\begin{align}
\mathcal{E}(u,v)
=\int_{\Omega} f(x) v(x)\d x\qquad \text{for all }~~ v\in V_{\nu,0}(\Omega|\R^d).
\end{align}
In fact, $u$ is the weak solution to the nonlocal Dirichlet problem $Lu=f$ in $\Omega$ and $u=g$ on $\Omega^c$. Moreover, there exists $C= C(\Omega, d,\Lambda,\nu)>0$ independent of $f$ and $g$, 
\begin{align}\label{eq:dirichlet-regular}
\|u\|_{\VnuOm}\leq C(\|f\|_{L^2(\Omega)} + \|g\|_{\TnuOm}).
\end{align}
%
%The function $u$ is the weak solution to the nonlocal Dirichlet problem $Lu=f$ in $\Omega$ and $u=g$ on $\Omega^c$.
\end{theorem}

The result follows from the Lax-Milgram lemma because the linear form $v\mapsto\int_{\Omega} fv$ is continuous on 
$V_{\nu,0}(\Omega|\R^d) $ 
and the bilinear form $\mathcal{E}(\cdot, \cdot)$ bounded and coercive on $V_{\nu,0}(\Omega|\R^d) $ (see \autoref{thm:poincare-friedrichs-ext}). It is noteworthy to recall that under the non-integrability condition \eqref{eq:non-integrability-condition} and the L\'{e}vy integrability condition \eqref{eq:levy-cond}, $V_{\nu,0}(\Omega|\R^d) $ is compactly embedded in $L^2(\Omega)$. With this at hand, analogously to \autoref{thm:existence-of-eigenvalue-Neumann} there exist a family $(\psi_n)_n$ elements of $V_{\nu,0}(\Omega|\R^d)$, orthonormal basis of $L^2(\Omega)$ and an increasing sequence of real number $ 0<\lambda_1\leq \cdots\leq \lambda_n\leq \cdots. $ such that $\lambda_n \to \infty$ as $n \to \infty$ and each $\psi_n$ is a Dirichlet eigenfunction of $L$ whose corresponding eigenvalue is $\lambda_n$ namely
\begin{align*}
\mathcal{E}(\psi_n,v) = \lambda_n \int_{\Omega} \psi_n(x) v(x)\d x\quad \text{for all }~~ v\in V_{\nu,0}(\Omega|\R^d).
\end{align*}
Note that the constants $\mu_1>0$ and $\lambda_1>0$ respectively satisfy the Poincar\'e inequalities
\begin{align*}
\mathcal{E}(u,u)&\geq \mu_1\|u\|^2_{L^2(\Omega)}, \quad\text{for all}~ u \in \VnuOm^\perp,\\
% \intertext{and }
\mathcal{E}(u,u)&\geq \lambda_1\|u\|^2_{L^2(\Omega)}, \quad\text{for all}~ u \in V_{\nu,0}(\Omega|\R^d).
\end{align*}
Before we formally define the Dirichlet-to-Neumann map, some prerequisites are required. Let $f\in L^2(\Omega)$ and $g\in \TnuOm$. Assume, $\lambda< \lambda_1$, then the bilinear form $\mathcal{E}_{-\lambda}(u, u) = \mathcal{E}(u, u) -\lambda \|u\|^2_{L^2(\Omega)}$ is coercive on $V_{\nu,0}(\Omega|\R^d)$. Thus there exists a function $u \in \VnuOm$ unique weak solution to the Dirichlet problem $Lu - \lambda u = f$ in $\Omega$ and $u=g$ on $\Omega^c$. Explicitly, $u=g$ on $\Omega^c$ and 
\begin{align}\label{eq:weak-dirichlet-lambda}
\mathcal{E}(u,v)- \lambda \int_{\Omega} u(x) v(x)\d x = \int_{\Omega} f(x) v(x)\d x \quad \text{for all }~~ v\in V_{\nu,0}(\Omega|\R^d).
\end{align}
Moreover, the estimate \eqref{eq:dirichlet-regular} (with the estimating constant depending on $\lambda$) remains true. More generally, by the mean of Fredholm alternative and the closed graph theorem, the preceding facts \eqref{eq:weak-dirichlet-lambda} and \eqref{eq:dirichlet-regular} respectively remain true for the operator $L -\lambda$, whenever $\lambda\in \mathbb{R}\setminus\{\lambda_n:n\geq1\}$. 

\medskip

From now on we suppose $f=0$ and $\lambda\in \mathbb{R}\setminus\{\lambda_n:n\geq1\}$ and label the solution of \eqref{eq:weak-dirichlet-lambda} by $u =u_g$. Then the mapping $g\mapsto u_g$ is linear and continuous from $\TnuOm$ to $\VnuOm$ since by \eqref{eq:dirichlet-regular} we have 
\begin{align*}
\|u_g\|_{\VnuOm}\leq C \|g\|_{\TnuOm}. 
\end{align*} 
Given $v \in \TnuOm$, put $\widetilde{v}= \operatorname{ext}(v)\in \VnuOm$ as an extension of $v$. Let $\langle\cdot, \cdot \rangle$ be the dual pairing between $\TnuOm$ and $\TnuOm'$. 

\medskip

\begin{definition}
Let $\lambda\in \mathbb{R}\setminus\{\lambda_n:n\geq1\}$. We call the mapping $\mathscr{D}_{\lambda}: \TnuOm \to \TnuOm'$ with $g \mapsto \mathscr{D}_\lambda g= \mathcal{E}_{-\lambda}(u_g, \widetilde{\cdot})$ such that $\langle\mathscr{D}_\lambda g, v\rangle = \mathcal{E}_{-\lambda}(u_g, \widetilde{v})$, the Dirichlet-to-Neumann map with respect to the operator $L-\lambda$. 
%In particular, taking $\widetilde{g}=u_g$ we have $\langle\mathscr{D}_\lambda g, g\rangle =\mathcal{E}_{-\lambda}(u_g, u_g)$.  This expression in known from the Douglas identiy. see \eqref{eq:douglas-formula}
\end{definition}

\medskip

\begin{theorem}\label{thm:DN-map}
The Dirichlet-to-Neumann operator $\mathscr{D}_{\lambda}: \TnuOm \to \TnuOm'$ with $g \mapsto \mathscr{D}_\lambda g= \mathcal{E}_{-\lambda}(u_g, \widetilde{\cdot})$ is well defined, linearly bounded and self-adjoint. Moreover, if we take $c=\min(1,-\lambda)$ then for all $g \in \TnuOm$ we have 
\begin{align*}
\langle \mathscr{D}_\lambda g , g\rangle\geq c \|u_g\|^2_{\VnuOm}.
\end{align*}
In particular if $c>0$ it follows that 
\begin{align*}
	\langle \mathscr{D}_\lambda g , g\rangle\geq c \|g\|^2_{\TnuOm}.
\end{align*}
\end{theorem}

\medskip

\begin{proof}

Consider $v' \in \VnuOm$ another extension of $v$ then $v'-\widetilde{v}\in V_{\nu,0}(\Omega|\R^d)$ and by definition of $u_g$ we have 
\begin{align*}
\mathcal{E}_{-\lambda}(u_g,v'- \widetilde{v})=0\quad \text{that is}\quad \mathcal{E}_{-\lambda}(u_g, \widetilde{v})=\mathcal{E}_{-\lambda}(u_g,v').
\end{align*}
Therefore the mapping $v\mapsto\mathcal{E}_{-\lambda}(u_g, \widetilde{v})$ is well defined, linear and bounded on $\TnuOm$. Indeed, 
\begin{align*}
|\mathcal{E}_{-\lambda}(u_g, \widetilde{v})|\leq (|\lambda|+1)\|u_g\|_{\VnuOm} \|\widetilde{v}\|_{\VnuOm}.
\end{align*}
Since the extension $\widetilde{v}$ of $v$ is arbitrarily chosen, upon the estimate \eqref{eq:weak-dirichlet-lambda} we obtain
\begin{align*}
|\mathcal{E}_{-\lambda}(u_g, \widetilde{v})|\leq C\|g\|_{\TnuOm} \|v\|_{\TnuOm}.
\end{align*}
This shows that, $\mathcal{E}_{-\lambda}(u_g, \widetilde{\cdot})$ belongs $\TnuOm'$. Subsequently it also 
follows from this estimate that the mapping $\mathscr{D}_{\lambda}: \TnuOm \to \TnuOm'$ with $g \mapsto \mathscr{D}_\lambda g= \mathcal{E}_{-\lambda}(u_g, \widetilde{\cdot})$ is linear and bounded. Now let $g,h \in \TnuOm$ specializing the definition of $\mathscr{D}_\lambda$ with $\widetilde{g} = u_g$ and $\widetilde{h} =u_h$ the self-adjointness is obtained as follows 
\begin{align*}
\langle \mathscr{D}_\lambda g , h\rangle= \mathcal{E}_{-\lambda}(u_g, u_h) = \mathcal{E}_{-\lambda}(u_h, u_g)= \langle \mathscr{D}_\lambda h, g\rangle. 
\end{align*}
The choice $c=\min(1,-\lambda)$ leads to $\langle \mathscr{D}_\lambda g , g\rangle= \mathcal{E}_{-\lambda}(u_g, u_g)\geq c \|u_g\|^2_{\VnuOm}$. 
\end{proof}

\medskip

\begin{remark}
The above definition is motivated by the following observation. Assume $u_g$ is as before and $\varphi\in C_c^\infty(\R^d)$. The Gauss-Green formula \eqref{eq:green-gauss-nonlocal} gives 
\begin{align}\label{eq:DN-weak-formula}
\langle \mathscr{D}_\lambda g ,\varphi\rangle &= \mathcal{E}_{-\lambda}(u_g, \varphi) = \int_{\Omega^c} \mathcal{N}u_g(y)\varphi(y)\d y.
%= \int_{\Omega^c} \widetilde{\nu}^{-1}(y)\mathcal{N}u_g(y)\varphi(y)\widetilde{\nu}(y)\d y\, . 
\end{align}
From the second  equality we can identify $\mathscr{D}_\lambda g= \mathcal{N}u_g\in L^2(\Omega^c, \widetilde{\nu}^{-1}) \subset \TnuOm'$. Hence $\mathscr{D}_\lambda : g\mapsto \mathcal{N}u_g$, which agrees with conceptual idea behind the Dirichlet-to-Neumann map in the classical case.
\end{remark}

%For the sake of consistency with the setting for the Robin problem, we will consider the following equivalent (up to a multiplicative weight)  alternative identification of the Dirichlet-to-Neumann operator: $\mathscr{D}_\lambda : \TnuOm \to L^2(\Omega^c, \widetilde{\nu}) $ with $\mathscr{D}_\lambda g=\widetilde{\nu}^{-1} \mathcal{N}u_g $. 
% 
\begin{theorem}
Let the assumptions of \autoref{thm:nonlocal-Robin-var} be in force. Denote by $L_\beta$ the operator $L$ subject to the Robin boundary condition $\mathcal{N} u+ \beta u=0$. Then the point spectrum $\sigma_p(L_\beta)=(\gamma_n(\beta))_n $ of $L_\beta$ is infinitely countable say $0<\gamma_1(\beta)\leq \gamma_2(\beta) \leq \cdots\leq \gamma_n(\beta)\leq \cdots$, with $\gamma_n(\beta)\to \infty$ as $n\to\infty$, and the corresponding eigenfunctions belong to $\VnuOm$ and form an orthonormal basis of $L^2(\Omega)$. 
\end{theorem}

\medskip

\begin{proof}
It suffices to proceed as  in  the proof of \autoref{thm:existence-of-eigenvalue-Neumann}, see also \cite[Theorem 4.36]{guy-thesis}. 
\end{proof}

\medskip

Next, we see the relation between the spectrum of the operator $L$ subject to Robin boundary condition and that of Dirichlet-to-Neumann operator.

\begin{theorem}\label{thm:DN-map-spectral} Let $\lambda\in \mathbb{R}\setminus\{\lambda_n:n\geq1\}$ and $\beta: \Omega^c \to \R$ be measurable. Consider the Dirichlet-to-Neumann map $\mathscr{D}_\lambda : \TnuOm\to \TnuOm', \,\, \mathscr{D}_\lambda g =\mathcal{N}u_g$. 
Then, $0\in \sigma_p(\mathscr{D}_\lambda+\beta)$ if and only if $\lambda\in \sigma_p(L_\beta)$. In addition, $\dim \ker ( L_\beta -\lambda) = \dim \ker (\mathscr{D}_\lambda + \beta) $. 
\end{theorem}

\medskip

\begin{proof}
Let $u\in \ker ( L_\beta -\lambda)$ then for all $v\in \VnuOm$, 
\begin{align*}
Q_\beta(u,v) = \lambda\int_{\Omega} u(x)v(x)\d x\quad
\text{equivalently} \quad
\mathcal{E}_{-\lambda}(u,v) = -\int_{\Omega^c}u(y)v(y)\beta(y)\d y.
\end{align*}

Set $g= \operatorname{Tr}(u)= u|_{\Omega^c}$, with the aid of \eqref{eq:DN-weak-formula} the above relation reduces to 
\begin{align*}
\int_{\Omega^c}\mathcal{N} u_g(y)v(y)\d y= -\int_{\Omega^c}g(y)v(y)\beta(y)\d y.
\end{align*}
Thus $ g\in \ker ( \mathscr{D}_\lambda + \beta)$. We have shown that the mapping $T: \ker ( L_\beta -\lambda)\to \ker ( \mathscr{D}_\lambda + \beta)$ with $u\mapsto
\operatorname{Tr}(u)$ is well defined and onto. Both assertions will follow once we show that $T$ defines a bijection, in other words we only have to show that $T$ is one-to-one. For $u\in \ker ( L_\beta -\lambda)$ if $\operatorname{Tr}(u)=0$ then from the first relation above, we have $\mathcal{E}(u,v) = \lambda\int_{\Omega} u(x)v(x)\d x$ for all $v\in V_{\nu,0}(\Omega|\R^d)$. Necessarily, $u=0$ otherwise $\lambda$ is a Dirichlet eigenvalue which is not the case by assumption.
\end{proof}

% Let us begin with some weak formulations of the operators $L$ and $\mathcal{N}$. For $u\in \VnuOm$ we say that, $Lu\in L^2(\Omega)$ if there exists $h\in L^2(\Omega)$ such that for all $\varphi\in C_c^\infty(\Omega)$
% %
% \begin{align*}
% \mathcal{E}(u, \varphi) = \int_{\Omega}h(x)\varphi(x)\d x
%\end{align*} 
% 
% In this case we formally write $h=Lu$. Let us remind that, if $ u\in \VnuOm$ then $u \in L^2(\Omega^c, \nu_\Omega) $. Assume now that $Lu \in L^2(\Omega) $ in the above sense then we say that $\mathcal{N}u \in L^2(\Omega^c, \nu_\Omega)$ if the exists $h \in L^2(\Omega^c, \nu_\Omega)$ for which, for all $v\in \VnuOm$
% %
% \begin{align*}
% \mathcal{E}(u, v) - \int_{\Omega}Lu (x) v(x)\d x= \int_{\Omega^c} h(x)v(x) \nu_\Omega(x)\d x. 
% \end{align*}
%
%We will simply set $\mathcal{N}u= h\nu_\Omega$. In this way, the Gauss-Green formula remains consistent. 
% Regarding \autoref{thm:linear-form-characto} 

\section{Transition from nonlocal to local}\label{sec:transition}

The main purpose of this section is to prove the convergence of a sequence of nonlocal Neumann problems to a local Neumann problem, i.e., the corresponding solutions converge. The main result of this section is \autoref{thm:phase-transition}. We consider the following set-up: Let $(\nu_\alpha)_{\alpha\in(0,2)}$ be a family of L\' evy radial functions  approximating the Dirac measure at the origin, i.e., for every $\alpha, \delta > 0$
\begin{align}\label{eq:assumption-nu-alpha}
\begin{split}
\nu_\alpha\geq 0\,\,\text{ is radial}, \quad \int_{\mathbb{R}^d}	(1\land |h|^2)\nu_\alpha (h)\d h=d , \quad \lim_{\alpha\to 2}\int_{|h|>\delta}	\nu_\alpha(h)\d h=0\,.
\end{split}
\end{align}                      
Note that there is no restriction on the support of $\nu_\alpha$. The above definition of  $(\nu_\alpha)_{0<\alpha<2}$ generalizes the spectrum of possible approximation sequences in \cite{FGKV20, DTZ22}. Note that, convergence of nonlocal variational structures including finite dimensional Galerkin methods have already been considered in \cite{MD15} and \cite{BMP15} for homogeneous nonlocal problems of vanishing horizon-type. 

\medskip

We denote $L_\alpha$ and $\mathcal{N}_\alpha$ to be the nonlocal operators associated with $\nu_{\alpha}$, i.e.,
\begin{align*}
L_\alpha u(x) &= 2 \pv \int_{\R^d }(u(x)-u(y)) \nu_{\alpha}(x-y)\,\d y,\\
\mathcal{N}_\alpha u(x) &= 2\int_{\Omega}(u(x)-u(y)) \nu_{\alpha}(x-y)\,\d y.
\end{align*}
The associated energy forms are defined by 
\begin{align*}
\mathcal{E}^{\alpha}_{\Omega}(u,v) &=  \iil_{\Omega \Omega} \big(u(y)-u(x)\big) \big(v(y)-v(x)\big) \nu_\alpha(x-y)\d x \, \d y\,, \\%\label{eq:inner-form} \\
\mathcal{E}^{\alpha}(u,v) &= \iil_{(\Omega^c\times \Omega^c)^c} \big(u(y)-u(x)\big) \big(v(y)-v(x) \nu_\alpha(x-y) \d x \, \d y.% \label{eq:ext-form}
\end{align*}

Let us mention two prototypical examples of interest here. For more concrete examples we refer the reader to \cite{FGKV20, guy-thesis}. 

\begin{example}
Define $\nu_\alpha(h) = a_{d,\alpha}  |h|^{-d-\alpha}$ with $ a_{d,\alpha}  = \tfrac{d\alpha(2-\alpha)}{2 |\mathbb{S}^{d-1}|} .$ Indeed, passing through polar coordinates yields 
\begin{align*}
& \int_{\mathbb{R}^d} (1\land |h|^2) |h|^{-d-\alpha}\,\d h= |\mathbb{S}^{d-1}| \Big(\int_{0}^{1} r^{1-\alpha}\,\d r+ \int_{1}^{\infty}r^{-1-\alpha}\,\d r\Big) 
%&= |\mathbb{S}^{d-1}| \Big(\frac{1} {2-\alpha}+ \frac{1} {\alpha}\Big)
=\frac{2d|\mathbb{S}^{d-1}| }{d\alpha(2-\alpha)}= da^{-1}_{d,\alpha}. 
\end{align*}
For $\delta>0,$ a similar computation gives
\begin{align*}
&a_{d,\alpha}
\il_{|h|\geq \delta} (1\land |h|^2) |h|^{-d-\alpha}\,\d h\leq	\frac{d\alpha(2-\alpha)}{2d} \int_{\delta}^{\infty}r^{-1-\alpha}\,\d r
= \frac{d} {2} (2-\alpha)\delta^{-\alpha}\xrightarrow{\alpha \to 2}0.
\end{align*}
The choice of $\nu_\alpha(h) = a_{d,\alpha}  |h|^{-d-\alpha}$ gives rise to a multiple of fractional Laplace operator, i.e., $L_\alpha= \frac{ a_{d,\alpha} }{C_{d,\alpha}} (-\Delta)^{\alpha/2}$, where we recall that  $C_{d, \alpha}$ is  the normalizing constant of $(-\Delta)^{\alpha/2}$. Note however that $ \frac{ a_{d,\alpha} }{C_{d,\alpha}}\to 1$  as $\alpha\to 2$ see \cite{AAS67,Hitchhiker, guy-thesis}.

%	\end{enumerate}
\end{example}

\begin{example}\label{exa:vanishing-horizon}
Let $\nu\in L^1(\R^d, 1\land|h|^2) $ be any radial L\'evy density that is normalized, i.e.,  $$\int_{\R^d}1\land|h|^2\,\nu(h)\d h=d.$$	
Let a family $(\nu_{\alpha})_{\alpha\in (0,2)}$ be defined by $\nu_{\alpha}= \nu^{2-\alpha}$ where $\nu^\varepsilon$ is a rescaled version of $\nu$ in the following sense:
\begin{align*}%\label{eq:rescalled-levy-measure-bis}
\begin{split}
\nu^\varepsilon(h) = 
\begin{cases}
\varepsilon^{-d-2}\nu\big(h/\varepsilon\big)& \text{if}~~|h|\leq \varepsilon\\
\varepsilon^{-d}|h|^{-2}\nu\big(h/\varepsilon\big)& \text{if}~~\varepsilon<|h|\leq 1\\
\varepsilon^{-d}\nu\big(h/\varepsilon\big)& \text{if}~~|h|>1.
\end{cases}
\end{split}
\end{align*}
\end{example}

Then, as shown in \cite[Proposition 2.2]{Fog21}, $(\nu_{\alpha})_{\alpha\in (0,2)}$ satisfies \eqref{eq:assumption-nu-alpha}. Note that, as a possible simple example, one could consider $\nu(h)= c \mathds{1}_{B_1(0)}(h)$, so that $(\nu_\alpha)$ would correspond to what is known as vanishing horizon in peridynamics, see \cite{DEYu21, DTZ22}.

\medskip

The next result implies the convergence of the nonlocal normal derivative to the local one.  We point out that a similar convergence has been recently established in \cite{HK23}. 

\begin{lemma}\label{lem:colapsing-to-boundary}
Assume $\Omega\subset \R^d$ is an open bounded set with Lipschitz boundary. Let $\varphi \in C^2_b(\R^d)$ and $v\in \VnuOma$. The following assertions hold true. 
\begin{enumerate}[$(i)$]
\item There is a constant $C>0$ independent of $\alpha$ such that 
\begin{align*}
\sup_{\alpha\in (0,2)}\Big|\int_{\Omega^c} \mathcal{N}_\alpha \varphi(y)v(y)\d y\Big|\leq C\|\varphi\|_{C^2_b(\R^d)} \|v\|_{\VnuOma}\,.
\end{align*}
\item Assume $v\in H^1(\R^d)$ then
\begin{align*}
\lim_{\alpha\to 2}\int_{\Omega^c} \mathcal{N}_\alpha \varphi(y)v(y)\d y =\int_{\partial\Omega} \frac{\partial \varphi}{\partial n}(x) v(x)\d\sigma(x)\,.
\end{align*}
\end{enumerate}
\end{lemma}

\medskip 

\begin{proof}
In view of the estimates  \eqref{eq:second-difference} and \eqref{eq:first-order-diff} respectively,  we have 
\begin{align*}
|L_{\alpha}\varphi|\leq  4d \|\varphi\|_{C^2_b(\R^d)}\quad\text{and}\quad \mathcal{E}^\alpha(\varphi, \varphi)\leq 4d|\Omega|\|\varphi\|^2_{C^1_b(\R^d)} \quad\text{for all }~~\alpha\in (0,2)\,.
\end{align*}
%
%Where we recall that,
%\begin{align*}
%\mathcal{E}^\alpha(v,v) =\frac{1}{2} \iint\limits_{(\Omega^c\times\Omega^c)^c} (v(x)-v(y))^2\nu_\alpha(x-y)\d x\d y\,.
%\end{align*}
%%
By the continuity of the  linear mapping $ v\mapsto \mathcal{E}^\alpha(\varphi, v) - \int_{\Omega} L_{\alpha}\varphi(x)v(x)\d x$, the Gauss-Green formula \eqref{eq:green-gauss-nonlocal} is applicable for $\varphi\in C_b^2(\R^d)$ and $v\in \VnuOma$.  Therefore, with the help of the above estimates we get $(i)$ as follows

\begin{align*}
\Big|\int_{\Omega^c} \mathcal{N}_\alpha \varphi(y)v(y)\d y\Big|
&= \Big| \mathcal{E}^\alpha(\varphi, v) - \int_{\Omega} L_{\alpha}\varphi(x)v(x)\d x\Big|\\
&\leq \mathcal{E}^\alpha(\varphi, \varphi)^{1/2} \mathcal{E}^\alpha(v,v)^{1/2} +\|L_\alpha\varphi\|_{L^2(\Omega)} \|v\|_{L^2(\Omega)}\\
&\leq C\|\varphi\|_{C^2_b(\R^d)}\|v\|_{\VnuOma}\,.
\end{align*}
Noting that $L_{\alpha}\varphi(x)\xrightarrow[]{\alpha\to 2} -\Delta\varphi(x)$ for all $x\in \R^d$ (see \cite[Proposition 2.4]{Fog21}) and that $|L_{\alpha}\varphi|\leq \frac{4}{d} \|\varphi\|_{C^2_b(\R^d)}$, the Lebesgue dominated convergence theorem yields
\begin{align*}
\int_{\Omega}L_{\alpha}\varphi(x)v(x)\d x \xrightarrow[]{\alpha\to 2} \int_{\Omega}-\Delta \varphi(x)v(x)\d x\,.
\end{align*}
On the other hand, according to \cite{Fog21} and \cite[Theorem 3.4]{FGKV20}, we have that 
\begin{align}\label{eq:convergence-energy}
\begin{split}
&\iint\limits_{\Omega\Omega}(v(x)-v(y))^2\nu_\alpha(x-y)\d x\d y \xrightarrow[]{\alpha\to 2}\int_{\Omega}|\nabla v(x)|^2\d x\\
&\iint\limits_{\Omega\Omega^c} (v(x)-v(y))^2 \nu_{\alpha}(x-y)\d x\d y\xrightarrow{\alpha\to 2}0.%\label{eq:convergence-energy-collapse}. 
\end{split}
\end{align}
So that, $\mathcal{E}^\alpha(v, v) \xrightarrow[]{\alpha\to 2}\int_{\Omega}|\nabla v(x)|^2\d x$. Thus we also have 
\begin{align*}
\mathcal{E}^\alpha(\varphi, v) \xrightarrow[]{\alpha\to 2}\int_{\Omega}\nabla\varphi(x)\cdot \nabla v(x)\d x.
\end{align*}

Finally from the foregoing and the local Gauss-Green formula we obtain $(ii)$ as follows
\begin{align*}
\lim_{\alpha\to 2}\int_{\Omega^c} \mathcal{N}_\alpha \varphi(y)v(y)\d y 
&=
\lim_{\alpha\to 2} \mathcal{E}^\alpha(\varphi, v)-\lim_{\alpha\to 2} \int_{\Omega}L_{\alpha}\varphi(x)v(x)\d x\\
&= \int_{\Omega}\nabla\varphi(x)\cdot \nabla v(x)\d x - \int_{\Omega}\Delta\varphi(x) v(x)\d x\\
&=\int_{\partial\Omega} \frac{\partial \varphi}{\partial n}(x) v(x)\d\sigma(x)\,.
\end{align*}

\end{proof}

%\todo[inline]{@GF: Change to general Levy. Add nice comments on fractional Laplace.}

\begin{theorem}[\textbf{Convergence of weak solution}]\label{thm:phase-transition} Let $\Omega\subset \R^d$ be an open bounded and connected domain with Lipschitz boundary. Let $(f_\alpha)_\alpha$ be functions converging in the weak sense to another function $f$ in $L^2(\Omega)$ and let $g_\alpha =\mathcal{N}_\alpha\varphi$ and $g=\frac{\partial \varphi}{\partial n}$ for some $\varphi \in C^2_b(\R^d)$. Assume $u_\alpha \in  \VnuOma^\perp$ is a weak solution to $L_\alpha u= f_\alpha$ on $\Omega$ and $\mathcal{N}_\alpha u= g_\alpha$ on $\Omega^c$ that is, 
\begin{align*}
\mathcal{E}^\alpha(u_\alpha, v) = \int_{\Omega}f_\alpha(x) v(x) + \int_{\Omega^c}g_\alpha(x) v(x)\quad\text{for all }~~v\in \VnuOma^\perp\,.
\end{align*}
Let $u\in H^1(\Omega)^\perp$ be the unique weak solution in  to the Neumann problem $-\Delta u=f$ in $\Omega$ and $\frac{\partial u}{\partial n} =g$ on $\partial\Omega$ i.e. 
\begin{align*}
\int_{\Omega } \nabla u(x)\cdot\nabla v(x)\d x= \int_{\Omega } f(x) v(x)\d x + \int_{\partial\Omega } g(x)v(x)\d\sigma(x)\qquad\text{for all $u\in H^1(\Omega)^\perp$}. 
\end{align*}

Then $(u_\alpha)_\alpha$ strongly converges to $u$ in $L^2(\Omega)$, i.e.,  $\|u_\alpha-u\|_{L^2(\Omega)}\xrightarrow[]{\alpha\to 2}0$. Moreover, the following weak convergence of the energy forms holds true
\begin{align}\label{eq:weak-convergence-energy}
\mathcal{E}^\alpha(u_\alpha, v) \xrightarrow[]{\alpha\to 2} \int_{\Omega}\nabla u(x)\cdot \nabla v(x)\d x
\qquad\text{for all $v\in H^1(\R^d)$}.
\end{align}
\end{theorem}

\begin{remark}{\ }
\begin{enumerate}[(i)]
	\item In case of the homogeneous problem, i.e., for $\varphi=0$ and $f_\alpha=f$, the corresponding result is a direct consequence of the Mosco-convergence of $(\mathcal{E}^\alpha(\cdot, \cdot), \VnuOma)_\alpha$ to the gradient form $\int_{\Omega}|\nabla u(x)|^2\d x$ with domain $H^1(\Omega)$, see \cite{FGKV20}.
	\item The convergence in result of Theorem \ref{thm:phase-transition} remains true if one replaces the Neumann condition with the Dirichlet condition, see \cite{guy-thesis}. 
	\item Examples of the type of \autoref{exa:vanishing-horizon} have been considered in relation to models in peridynamics, see \cite[Section 4.2]{DTZ22} and \cite{BMP15} for a natural nonlinear setting.
	\item The assertion of the theorem remains true under the weaker assumption that $(g_\alpha, \psi)_{L^2(\Omega^c)}$ convergences to $(g, \psi)_{L^2(\partial \Omega)}$ for all $\psi \in H^1(\R^d)$.
	\item It is desirable to study \autoref{thm:phase-transition} under more general assumptions, e.g., under a weaker assumption than $\varphi \in C^2_b(\R^d)$. A sufficient condition to be expected is $g \in H^{1/2}(\partial \Omega)$.
\end{enumerate}
\end{remark}	

\begin{proof}
A compactness argument as in \cite[Corollary 2.1]{Ponce2004}, see \cite[Chapter 5]{guy-thesis}, shows that for certain $\alpha_0\in (0,2)$ there exists a constant positive $C>0$ depending only on $\alpha_0, \Omega$ and $d$ such that for all $v \in L^2(\Omega)^\perp$ and all $\alpha\in (\alpha_0, 2)$
\begin{align}\label{eq:uniform-coercivity}
\|v\|^2_{\VnuOma}\leq C\mathcal{E}^\alpha(v,v). 
\end{align}

In view of the weak convergence, we can assume without loss  generality that $\sup_{\alpha\in (0,2)}\|f_\alpha\|_{L^2(\Omega)}<\infty$. This together with the definition of $u_\alpha$ along with \autoref{lem:colapsing-to-boundary} $(i)$ yields 
\begin{align*}
\mathcal{E}^\alpha(u_\alpha,u_\alpha) &= \int_{\Omega} f_\alpha(x) u_\alpha (x)\d x + \int_{\Omega^c} g_\alpha(y) u_\alpha (y)\d y\\
&\leq \|u_\alpha\|_{\VnuOma}(\|f_\alpha\|_{L^2(\Omega)} + \|\varphi\|_{C_b^2(\R^d)})\\
&\leq C \|u_\alpha\|_{\VnuOma}. 
\end{align*}
Combining this with \eqref{eq:uniform-coercivity}, then for a generic constant $C>0$ independent of $\alpha$ we have the following uniform boundedness
\begin{align}\label{eq:uniform-boundedness}
\|u_\alpha\|_{\HnuOma}\leq \|u_\alpha\|_{\VnuOma}\leq C\qquad\text{for all } \alpha\in (\alpha_0, 2)\,.
\end{align}
Recall that, see \cite{Fog21,guy-thesis}, $ \|u\|_{\HnuOma}\xrightarrow[]{\alpha\to 2} \|u\|_{H^1(\Omega)}$ for all $u \in H^1(\Omega)$. Whence from \cite[Lemma 2.2]{KS03} there exists $u'\in H^1(\Omega)$ and a subsequence $\alpha_n \xrightarrow[]{n\to \infty}2$ such that, 
\begin{align*}
\lim_{n\to \infty}\big(u_{\alpha_n}, v \big)_{H_{\nu_{\alpha_n}} (\Omega)} = \big(u', v \big)_{H^1 (\Omega)}, 
\end{align*}
where, we recall that
\begin{align*}
\big(w ,\, v \big)_{H_{\nu_{\alpha}} (\Omega)}&= \int_{\Omega} w(x)v(x)\d x+ \iint\limits_{\Omega\Omega} (w(x)-w(y))(v(x)-w(y))\nu_\alpha(x-y)\d x\d y\\
\big(w, v \big)_{H^1 (\Omega)} &= \int_{\Omega} w(x)v(x)\d x+ \int_{\Omega} \nabla w(x)\cdot \nabla v(x)\d x\,.
\end{align*}

By virtue of the asymptotic compactness, see \cite{BBM01,guy-thesis, Ponce2004}; see also  \cite[Section 4]{AAS67}, there exists a further subsequence that we still denote by $(\alpha_n)_n$ and a function $ u''\in H^1(\Omega)$ such that $\|u_{\alpha_n}-u''\|_{L^2(\Omega)} \xrightarrow[]{n\to \infty}0$. It is not difficult to show that $u'= u''$ almost everywhere in $\Omega$, $u\in H^1(\Omega)^\perp$ where we let $u=u'$, and that for all $v\in H^1(\Omega)$
\begin{align}\label{eq:weak-con-semi-omega}
\iint\limits_{\Omega\Omega} (u_{\alpha_n}(x)-u_{\alpha_n}(y))(v(x)-v(y))\nu_{\alpha_n}(x-y)\d x\d y\xrightarrow[]{n\to \infty} \int_{\Omega} \nabla u(x)\cdot \nabla v(x)\d x\,.
\end{align}

It remains to show that $u$ is the weak solution of the corresponding local Neumann problem. To this end, we fix $v\in H^1(\Omega)^\perp$, given that $\Omega$ has a Lipschitz boundary we let $\overline{v}\in H^1(\R^d)$ be an extension of $v$. The uniform boundedness in \eqref{eq:uniform-boundedness}  and the convergence in \eqref{eq:convergence-energy} yield 
\begin{align*}
\iint\limits_{\Omega\Omega^c} \big|(u_{\alpha_n}(x)-u_{\alpha_n}(y))(\overline{v}(x)-\overline{v}(y))\big|&\nu_{\alpha_n}(x-y)\d x\d y\\& \leq C \iint\limits_{\Omega\Omega^c} (\overline{v}(x)-\overline{v}(y))^2 \nu_{\alpha_n}(x-y)\d x\d y\xrightarrow[]{n\to \infty}0\,.
\end{align*}
This combined with \eqref{eq:weak-con-semi-omega} gives 
\begin{align*}
\mathcal{E}^{\alpha_n}(u_{\alpha_n}, \overline{v}) \xrightarrow[]{n\to \infty} \int_{\Omega} \nabla u(x)\cdot \nabla v(x)\d x\,.
\end{align*}
In particular, since $\overline{v}\in H^1(\R^d)$ can be arbitrarily chosen, we have the weak convergence 
\begin{align*}
\mathcal{E}^{\alpha_n}(u_{\alpha_n}, v) \xrightarrow[]{n\to \infty} \int_{\Omega} \nabla u(x)\cdot \nabla v(x)\d x\,\qquad\text{for all $v\in H^1(\R^d)$}.
\end{align*}
We know that $\overline{v}\in\VnuOma^\perp$ for all $\alpha\in (0,2)$, thus by definition of $u_{\alpha_n} $ it follows that, 
\begin{align*}
\mathcal{E}^{\alpha_n}(u_{\alpha_n}, \overline{v}) = \int_{\Omega}f_{\alpha_n}(x) v (x)\d x+ \int_{\Omega^c}g_{\alpha_n}(y) \overline{v}(y)\d y\,.
\end{align*}
By \autoref{lem:colapsing-to-boundary} $(ii)$ and the fact that $f_{\alpha_n}\rightharpoonup f$ weakly in $L^2(\Omega)$, letting $n \to \infty$ we obtain 
\begin{align*}
\int_{\Omega} \nabla u(x)\cdot \nabla v(x)\d x\, = \int_{\Omega}f(x) v (x)\d x+ \int_{\partial \Omega}g(x)v(x)\d \sigma(x)\,.
\end{align*}
By virtue of the uniqueness of the limit  $u\in H^1(\Omega)^\perp$, the same reasoning can be applied to any other subsequence $(\alpha_n)_n$ with $\alpha_n\xrightarrow[]{n\to \infty}2$ and hence the claimed convergences hold true for the whole sequence as desired. 
\end{proof}

%%%%%%%%%%%%%%%%%%%%%%%%%%%%%%%%%%%%%%%%%%%%%%%%%%%%%%%%%%%%%%%%%%%%%%%%%%%%%%%
%%%%%%%%%%%%%%%%%%%%%%%%%%%%%%%%%%%%%%%%%%%%%%%%%%%%%%%%%%%%%%%%%%%%%%%%%%%%%%%
\appendix
\section{} \label{sec:appendix}
%%%%%%%%%%%%%%%%%%%%%%%%%%%%%%%%%%%%%%%%%%%%%%%%%%%%%%%%%%%%%%%%%%%%%%%%%%%%%%%
%%%%%%%%%%%%%%%%%%%%%%%%%%%%%%%%%%%%%%%%%%%%%%%%%%%%%%%%%%%%%%%%%%%%%%%%%%%%%%%

In the following appendices we explain basic properties of translation-invariant nonlocal operators $L$ driven by the density of a L\'{e}vy measure $\nu :\R^d \setminus \{0\} \to [0, \infty)$ satisfying condition \eqref{eq:levy-cond}. Throughout this section we assume $k(x,y)=\nu(x-y)$ for all $x \ne y$. The main goals include a definition of $Lu$ as a distribution in Proposition \ref{prop:fund-prop} and the Gauss-Green formula for nonlocal operators in \autoref{prop:gauss-green}.

\subsection{Basics on the operator $L$} Given $k\in \mathbb{N}$, denote $C_b^k(\R^d)$ as the space of bounded functions of class $C^k$ whose derivatives up to order $k$ are bounded. Recall that for  a sufficiently smooth function $v:\R^d\to \mathbb{R}$,  the operator $L$ is defined by 
\begin{align*}
Lv (x)&= \pv \int\limits_{\R^d}(v(x)-v(y))\nuxminy\d y = \lim_{\varepsilon\to 0^+} L_\varepsilon v (x)
\end{align*}
where
\begin{align*}
L_\varepsilon v (x)&= \int\limits_{\R^d\setminus B_{\varepsilon}(x) }\!\! (v(x)-v(y))\nuxminy\d y \qquad\hbox{$(x\in \R^d; \eps >0)$}.
\end{align*}

Here are some basic properties of the operator $L$.

\begin{proposition}\label{prop:uniform-cont}
Let $ u\in C^2_b(\R^d) $. Then the following properties are satisfied.
\begin{enumerate}[(i)]
\item The map $x\mapsto L u(x)$ is bounded and uniformly continuous. Moreover,  
\begin{align*}
Lu (x)= -\frac12\int_{\R^d}(u(x+h)+u(x-h)-2u(x)) \nu(h)\,\d h. 
\end{align*}

\item For each $\varepsilon>0$, the map $x\mapsto L_\varepsilon u(x)$ is uniformly continuous.
\item The family $(L_\varepsilon u(x))_\varepsilon$ is uniformly bounded and uniformly converges to $Lu,$ i.e.
\begin{align*}
\|L_\eps u-Lu\|_{L^\infty(\R^d)}\xrightarrow[]{\eps \to 0}0.
\end{align*}
\end{enumerate}
\end{proposition}

\bigskip

\begin{proof} 
Let $u \in C^2_b(\R^d)$. A simple change of variables implies 
\begin{align*}
L_\varepsilon u(x)= -\frac12\int\limits_{\R^d \setminus B_\varepsilon(0)}(u(x+h)+u(x-h)-2u(x)) \nu(h)\,\d h.
\end{align*}
An application of the fundamental theorem of calculus yields 
\begin{align*}
(u(x+h)+u(x-h)-2u(x)) &= \int_{0}^{1} \big[\nabla u(x+th) - \nabla u(x-th)\big]\cdot h\, \d t\\
&= \int_{0}^{1} \int_{0}^{1} 2t \big[D^2 u(x-th + 2sth) \cdot h\big]\cdot h\, \d s\d t\, .
\end{align*}
Since $u$ and its Hessian $D^2 u $ are bounded functions, we deduce 
\begin{align}\label{eq:second-difference}
\left|u(x+h)+u(x-h)-2u(x)\right|\leq 2\|u\|_{ C^2_b(\R^d)}(1 \land |h|^2),\quad \quad x,h \in \R^d.
\end{align}
The integrability of the function $h\mapsto (1 \land |h|^2)\nu(h)$ entails the boundedness of $x\mapsto L u(x)$
and the uniform boundedness of $x\mapsto L_\varepsilon u(x)$. It also allows us to get rid of the principal value formulation. Furthermore, we can prove the uniform convergence of $(L_\varepsilon u)_\varepsilon$ to $Lu $ by 
\begin{align*}
\|L_\eps u-Lu\|_{L^\infty(\R^d)} \leq 2\|u\|_{ C^2_b(\R^d)} \int_{B_\varepsilon(0)}(1 \land |h|^2)\nu(h)\d h \xrightarrow[]{\varepsilon\to0}0. 
\end{align*}
In order to prove the uniform continuity, we fix $x,z\in \R^d $ close enough, say $|x-z|\leq \delta$ with $0<\delta<1$. Then for every h $\in \R^d$, $h \ne 0$,
\begin{align*}
2|u(x)-u(z)| + |u(x+h)-u(z+h)|+ |u(x-h)-u(z-h)| \leq 4\delta \|u\|_{ C^2_b(\R^d)}.
\end{align*} 
This combined with \eqref{eq:second-difference} yields the uniform continuity via the integrability of $h\mapsto (1 \land |h|^2)\nu(h)$ as follows, 
\begin{align*}
\|L u(x)-L u(z) \|_{L^\infty(\R^d)} \leq 2\|u\|_{ C^2_b(\R^d)} \int_{\R^d}(\delta \land |h|^2)\nu(h)\d h \xrightarrow[]{\delta \to0}0. 
\end{align*}
The uniform continuity of $x\mapsto L_\varepsilon u(x)$ follows analogously. 
\end{proof}

\medskip

In order for $Lu(x)$ to be defined, $u$ needs to possess two properties: some regularity in the neighborhood of the point $x$ and some weighted integrability for $|x| \to \infty$. As shown above, being $C^2$ in the neighborhood of $x$ is more than sufficient as is boundedness for $|x| \to \infty$. Let us investigate some mild condition on $u$ as $|x| \to \infty$ that still allows a suitable definition of $Lu$. In order to do so, we additionally assume that $\nu$ is unimodal.

\begin{proposition}\label{prop:fund-prop} Let $\nu$ be a unimodal L\'{e}vy measure. Define a weight $\widehat{\nu}$ on $\R^d$ by $\widehat{\nu}(x)=\nu(\frac{1}{2}(1+|x|))$.  
\begin{enumerate}[(i)]
\item For $u\in C^2(\R^d)\cap L^1(\R^d, \widehat{\nu})$, the expression $Lu(x)$ exists for every $x\in \R^d$.
\item Assume that $\nu$ has full support. For $u \in L^1(\R^d, \widehat{\nu})$ the expression $Lu$ is defined in the 
distributional sense via the mapping $\varphi\mapsto \langle Lu, \varphi \rangle = (u, L\varphi)_{L^2(\R^d)}$. 
\item Assume that $\nu$ satisfies the scaling condition \eqref{eq:global-scaling_infinity}. Let $\Omega \subset \R^d$ be open and bounded and $u \in V_\nu(\Omega|\R^d)$. Then $Lu$ is defined in the 
distributional sense.
\end{enumerate}
\end{proposition}	

\begin{remark}
Note that $\widehat{\nu}\in L^1(\R^d)\cap L^\infty(\R^d)$ and that $L^1(\R^d,\widehat{\nu})$ contains $L^\infty(\R^d)$. 
\end{remark}

\begin{example}
If $\nu(h) = |h|^{-d-\alpha}$ for some $\alpha\in (0,2)$, then $\widehat{\nu}(h) \asymp (1+|h|)^{-d-\alpha}$.
\end{example}

\begin{proof}
For the proof of (i) we decompose the integral in the definition of $Lu(x)$ into the two domains $\{|y| \leq 2|x|+1\}$ and $\{|y| > 2|x|+1\}$. In the first domain we employ the Taylor formula as in the proof of \autoref{prop:uniform-cont}. For $y$ in the second domain we observe 
\begin{align*}
|x-y| \geq |y|-|x| \geq \frac{|y|}{2} + \frac{|y|}{2} - |x| \geq \frac{|y|+1}{2} \,.
\end{align*}
Thus, for $y$ from the second domain, by property \autoref{def:unimodality} we conclude $\nuxminy \leq c \widehat{\nu}(y)$ and thus 
\begin{align*}
\int\limits_{\{|y| > 2|x|+1\}} |u(x)| \nuxminy \d y  + \int\limits_{\{|y| > 2|x|+1\}} |u(y)| \nuxminy \d y \leq K_\nu |u(x)| + \|u\|_{L^1(\R^d, \widehat{\nu})}. 
\end{align*}
For the proof of (ii), let $\varphi\in C_c^\infty(\R^d)$ be supported in $B_R(0)$ for some $R\geq 1$. We claim 
\begin{align}\label{eq:estimate-test}
|L\varphi(x)|\leq C\|\varphi\|_{ C^2_b(\R^d)} \widehat{\nu}(x)\qquad \text{for all}~x\in \R^d. 
\end{align}
with some constant $C=C(R,d,\nu)$ depending only on $R,d$ and $\nu$. Indeed, suppose $|x|\geq 4R $, so that $\varphi(x)=0$. Since $|x-y|\geq \frac{|x|}{2}+R\geq \frac{1}{2}(1+|x|)$ for $y \in B_R(0)$, the property \autoref{def:unimodality} implies $\nuxminy\leq c \widehat{\nu}(x)$. Accordingly, 
\begin{align*}
|L\varphi(x) |\leq \int_{B_R(0)} |\varphi(y)|\nuxminy\d y \leq c|B_R(0)|\|\varphi\|_{C^2(\R^d)} \widehat{\nu}(x).
\end{align*}
Whereas, if $|x|\leq 4R$ the proof of \eqref{eq:estimate-test} is complete using \eqref{eq:second-difference} as follows. Since $\frac{1}{2}(1+ |x|)\leq 4R$ we have $\widehat{\nu}(x)\geq c_1$ for an appropriate constant $c_1>0$ depending on $R$ and $\nu$. Thus we conclude
\begin{align*}
|L\varphi(x)|\leq 4\Theta \|\varphi\|_{C^2_b(\R^d)}\leq c_1^{-1}4\Theta \|\varphi\|_{C^2_b(\R^d)} \widehat{\nu}(x)
\end{align*} 
with $\Theta= \int_{\R^d} (1\land |h|^2)\nu(h)\d h$. Note that in case of the fractional Laplace operator the estimate \eqref{eq:estimate-test} is analogous to \cite[Lemma 2.1]{FW12}. Finally, \eqref{eq:estimate-test} yields 
\begin{align*}
|(u, L\varphi)_{L^2(\R^d)}|\leq C \|\varphi\|_{ C^2_b(\R^d)}\int_{\R^d}|u(x)|\widehat{\nu}(x)\d x. 
\end{align*}
This shows that $Lu $ is a distribution when $u\in L^1(\R^d, \widehat{\nu})$. With regard to (iii) let $\Omega\subset \R^d$ be open and bounded. We show that the embedding $\VnuOm\hookrightarrow L^1(\R^d, \widehat{\nu})$ is continuous under the additional scaling assumption \eqref{eq:global-scaling_infinity}. 
Indeed, for $u\in\VnuOm$ we assume $ \Omega\subset B_R(0)$ for some $R\geq 1 $. Then $ |x-y|\leq R(1+|x|)$ for all $x\in \R^d$ and all $y\in \Omega$ so that by \eqref{eq:global-scaling_infinity} and \autoref{def:unimodality} we deduce $\widehat{\nu}(x)\leq  C\nu(R(1+|x|))\leq c C\nuxminy$. Here $c,C>0$ are constants independent of $x$ and $y$. Proceeding as in  \autoref{lem:natural-norm-on-V}, one arrives at the estimate
\begin{align*}
\int_{\R^d}|u(x)|\widehat{\nu}(x)\d x\leq C\|u\|_{\VnuOm}. 
\end{align*}
Therefore, regarding the preceding arguments $Lu$ is also a distribution whenever $u\in \VnuOm$. 
\end{proof}

\medskip
\subsection{Gauss-Green type formula}
Having at hand a nonlocal analog of the normal derivative as in \autoref{def:nonlocal-normal}, it makes sense to study a formula that resembles the classical Gauss-Green formula. Such formulas have been established in several contexts. See  \cite{DGLK13a} for numerous identities of a nonlocal vector calculus in the case of bounded kernels and \cite{DROV17} for the case of the fractional Laplace operator. Recall the classical Gauss-Green formula (see \cite[Chap 3]{Necas67}, \cite[Appendix A.3]{Trie92} or \cite[Theorem III.1.8]{BF13}) says for all $u\in H^{2}(\Omega) ~\hbox{and}~v\in H^{1}(\Omega),$
\begin{align}\label{eq:green-Gauss}
\int_{\Omega} (-\Delta) u(x) v(x) \, \d x = \int_{\Omega} \nabla u(x) \cdot \nabla v(x) \, \d x- \int_{\partial \Omega} \gamma_{1} u(x) \gamma_{0}v (x)\, \d \sigma(x). 
\end{align}
A reasonable explanation to this terminology is given in the \autoref{lem:colapsing-to-boundary}. 
%Recall 
%\begin{align*}
%\mathcal{E}(u,v)&=\frac{1}{2} \iil_{(\Omega^c\times \Omega^c)^c} \big(u(x)-u(y) \big) \big(v(x)-v(y) \big) \, \nuxminy \d x \, \d y 
%\end{align*}
%for $u,v \in C_b^1(\R^d)$. 
For a function $u\in C_b^1(\R^d)$ we know 
\begin{align}\label{eq:first-order-diff}
&|u(x)-u(y)|\leq 2\|u\|_{C_b^1(\R^d)}(1\land |x-y)|)\qquad\qquad (x,y\in \R^d) \,,  
\end{align}
which implies
\begin{align*}
&\iil_{(\Omega^c\times \Omega^c)^c} \big(u(x)-u(y) \big)^2 \, \nuxminy \d x \, \d y\leq 8 \|u\|^2_{C_b^1(\R^d)}  \iil_{\Omega\R^d}(1\land |x-y|^2)\, \nuxminy \d x \, \d y<\infty. 
\end{align*}

\medskip

\begin{proposition}[\textbf{Gauss-Green type formula}]\label{prop:gauss-green}
Let $\Omega$ be open and bounded. For $u \in C_b^{2}(\R^d)$ and $v \in C_b^{1}(\R^d)$
\begin{align}\label{eq:green-gauss-nonlocal}
\int_{\Omega} [Lu(x)]v(x)\d x= \mathcal{E}(u,v) -\int_{\Omega^c}\mathcal{N}u(y)v(y)\d y.
\end{align}
In particular, by choosing $v=1$ one deduces
\begin{align}\label{eq:integration-by-part-nonlocal}
\int_{\Omega} Lu(x)\d x=-\int_{\Omega^c}\mathcal{N}u(y)\d y.
\end{align}
\end{proposition}

%\todo[inline]{Question on improving the Gauss-Green Formula }
%\vspace{.25cm}
%
%%
%\textcolor{blue}{ Question C: Hope and Conjecture
%The above Gauss-Green formula may be update for $u \in H^{1+s}(\Omega|\R^d)= \{u\in H^1_{loc}(\R^d): \nabla u\in \VnuOm \}$ and $v \in \VnuOm$. 
%In short what could be the nonlocal counterpart of the space $H^2(\Omega)$(we refer to the Green Gauss formula \eqref{eq:green-Gauss})? This may leads to a suitable global regularity for the problem \eqref{eq:nonlocal-Neumann}.} 
%%
% 
%
%\todo[inline]{ For what suitable space is the linear form $v\mapsto \int_{\Omega^c} \mathcal{N} w(y)v(y)\d y$ ins continuous when $w$ is nice enough? }

%\medskip

\begin{proof}
Let $u \in C_b^{2}(\R^d)$ and $v \in C_b^{1}(\R^d)$. With the aid of \autoref{prop:uniform-cont} we can write 
\begin{align*}
&\int_{\Omega} [Lu(x)] v(x)\d x =\lim_{\varepsilon\to 0} \int_{\Omega} v(x) \d x 
\int\limits_{\R^d\setminus B_\varepsilon(x)} ((u(x)-u(y))\nuxminy\,\d y\\
&= \lim_{\varepsilon\to 0} \int\limits_{\Omega} \int\limits_{\Omega \setminus B_\varepsilon(x)}(u(x)-u(y))v(x)\nuxminy\,\d y \d x + \int\limits_{\Omega} \int\limits_{\Omega^c }(u(x)-u(y))v(x)\nuxminy\,\d y \d x .
\end{align*}

On one side, by a symmetry argument we have 
\begin{align*}
\lim_{\varepsilon\to 0} &\int\limits_{\Omega} \int\limits_{\Omega \setminus B_\varepsilon(x)}(u(x)-u(y))v(x)\nuxminy\,\d y \d x
= \lim_{\varepsilon\to 0} \hspace{-2ex}\iint\limits_{\Omega \times \Omega\cap \{|x-y|>\varepsilon\}}\hspace*{-3ex}(u(x)-u(y))v(x)\nuxminy\,\d y \d x \\
&= \lim_{\varepsilon\to 0} \frac{1}{2}\hspace{-4ex} \iint\limits_{\Omega \times \Omega\cap \{|x-y|>\varepsilon\}} \hspace{-4ex} (u(x)-u(y))(v(x)-v(y))\nuxminy\,\d y \d x \\
&= \frac{1}{2} \iint\limits_{\Omega \Omega } (u(x)-u(y))(v(x)-v(y))\nuxminy\,\d y \d x
\end{align*}
where one gets rid of the principal value using the estimate \eqref{eq:first-order-diff} applied to $u$ and $v$. 
On the other side, with the help of Fubini's theorem we have
\begin{align*}
& \iint\limits_{\Omega \Omega^c}(u(x)-u(y))v(x)\nuxminy\,\d y \d x \\
&= \iint\limits_{\Omega \Omega^c} (u(x)-u(y))(v(x)-v(y))\nuxminy\,\d y \d x
+ \int\limits_{\Omega^c} v(y)\d y \int\limits_{\Omega }(u(x)-u(y))\nuxminy\, \d x\\
&= \frac{1}{2}\iint\limits_{\Omega \Omega^c} (u(x)-u(y))(v(x)-v(y))\nuxminy\,\d y \d x
+ \frac{1}{2}\iint\limits_{\Omega^c \Omega}(u(x)-u(y))(v(x)-v(y))\nuxminy\,\d y \d x\\
& \qquad\qquad- \int\limits_{\Omega^c}\mathcal{N} u(y) v(y)\d y \, .
\end{align*}
Altogether inserted in the initial relation provide the desired relation.
\end{proof}

\medskip 

As a direct consequence of \autoref{prop:gauss-green} we have the following. 
\begin{corollary}[Second Gauss-Green identity]
For all $u,v\in C_b^2(\R^d)$ we have 
\begin{align}\label{eq:secondgreen-gauss-nonlocal}
\int_{\Omega} v(x)Lu(x)-u(x)Lv(x) \, \d x= \int_{\Omega^c} u(y)\mathcal{N}v(y)- v(y)\mathcal{N}u(y)\,\d y.
\end{align}
\end{corollary}

We now look at certain aspect of the dual of the trace space $\TnuOm$ in relation with the nonlocal normal derivative operator $\mathcal{N}$. 

\begin{theorem}\label{thm:linear-form-characto}
%	Assume $\TnuOm$ is endowed with the norm $\|\cdot\|_{ \TnuOm}$
For any linear continuous form $\ell :\TnuOm\to \mathbb{R}$ there exists a function $w\in \VnuOm$ such that for every $v\in C^\infty_c(\overline{\Omega}^c)$ 
\begin{align*}
\ell(v) = \int_{\Omega^c} \mathcal{N} w(y)v(y)\d y.
\end{align*}
In particular, given a measurable function $g:\Omega^c \to \mathbb{R}$ if the linear mapping $ \ell_g: v\mapsto \int_{\Omega^c} g(y)v(y)\d y$ is continuous on $\TnuOm $ then, there exists $w\in \VnuOm $ such that $g= \mathcal{N} w$ almost everywhere on $\Omega^c$.
\end{theorem}

\begin{proof}

Let $\ell \in \TnuOm'$ then because of the continuity of the trace operator $\operatorname{Tr}$, the linear form $\ell \circ \operatorname{Tr}$ is also continuous on $ \VnuOm $. By Riesz's representation theorem there exists $w \in \VnuOm$ such that $\ell \circ \operatorname{Tr}(v) = \left(v, w\right)_{( \VnuOm } $ for each $v \in \VnuOm$. In particular, for $v\in C^\infty_c(\overline{\Omega}^c)$ identified with its zero extension on $\Omega$ so that $\operatorname{Tr}(v) =v$, we remain with 

\begin{align*}
\ell(v) &= \int_{\Omega} w(x)v(x) \d x +\frac12 \iint\limits_{(\Omega^c\times\Omega^c)^c} (w(x)-w(y))(v(x)-v(y)) \nuxminy\,\d x\,\d y\\ 
%
%&= \int_{\Omega^c}\d y \int_{\Omega} (w(x)-w(y))(v(x)-v(y)) \nuxminy\,\d x\\
%
&= \int_{\Omega^c}v(y)\d y \int_{\Omega} (w(y)-w(x))\nuxminy\,\d x= \int_{\Omega^c} \mathcal{N} w(y)v(y)\d y.
\end{align*}

Furthermore, if $g: \Omega^c\to \mathbb{R}$ is such that $\ell_g$ is continuous on $\TnuOm$ then by the above computation, it follows that $g=\mathcal{N} w$ almost everywhere on $\Omega^c$ since 

\begin{align*}
\int_{\Omega^c} g(y)v(y)\d y= \int_{\Omega^c} \mathcal{N} w(y)v(y)\d y\qquad \hbox{for all }~~v\in C_c^\infty(\overline{\Omega}^c).
\end{align*}
\end{proof}

\begin{remark}
The second statement of \autoref{thm:linear-form-characto} particularly suggests that the space of all measurable functions $g:\Omega^c\to \mathbb{R}$ for which the  linear form $v\mapsto\int_{\Omega^c}g(y)v(y)\d y$ is continuous on $\TnuOm$ is contained in $\mathcal{N}(\VnuOm)$ (the range of $\mathcal{N}$). 
\end{remark}

\begin{remark}
The nonlocal normal derivative $\mathcal{N}u$ of a function measurable $u:\R^d\to \R$ can be thought of as the restriction of the regional operator on $\Omega$ associated with $k(x,y)=\nu(x-y)$ on $\R^d\setminus\Omega$. It might be interesting to know some situations where the pointwise definition $\mathcal{N}u(x)$ makes sense at least almost everywhere. It is straightforward to verify  the following: (i) if $u\in L^\infty(\Omega)$ then $\mathcal{N}u(x)$ exists for almost every $x\in \R^d\setminus \overline{\Omega}$, (ii) if $u\in \VnuOm$ then $\mathcal{N}u\in L^2_{\loc}(\R^d\setminus \overline{\Omega})$, (iii) more generally, if $u\in \VnuOm$ then $\mathcal{N}u\in L^2(\R^d\setminus \Omega, w^{-1}(x)\d x )$ where $w(x)=\int_{\Omega}\nu(x-y)\d y$, $x\in\R^d\setminus\Omega$. 
\end{remark}

%\vspace{3mm}

%%%%%%%%%%%%%%%%%%%%%%%%%%%%%%%%%%%%%%%%%%%%%%%%%%%%%%%%%%%%%%%%%%%%%
%%%%%% produces listoftodos and seems to work with the amsart package
%%%%%%%%%%%%%%%%%%%%%%%%%%%%%%%%%%%%%%%%%%%%%%%%%%%%%%%%%%%%%%%%%%%%%
\makeatletter
\providecommand\@dotsep{5}
\makeatother
%\listoftodos\relax
%%%%%%%%%%%%%%%%%%%%%%%%%%%%%%%%%%%%%%%%%%%%%%%%%%%%%%%%%%%%%%%%%%%%%%

\bibliographystyle{alpha}
%\bibliography{nonloc-bvp}
%\newcommand{\etalchar}[1]{$^{#1}$}

\newcommand{\etalchar}[1]{$^{#1}$}

\end{document}